\documentclass{article}
\usepackage {amsmath,amssymb,amsthm,sectsty,hyphenat,enumitem,cancel,mathtools,mathrsfs}
\usepackage{palatino}
\usepackage[all]{xy}
\usepackage[nottoc]{tocbibind}
\usepackage[hidelinks,pdfencoding=auto, psdextra]{hyperref}
\usepackage{tocloft}

\usepackage{color}

\def\mf#1{\mathfrak{#1}}
\def\mc#1{\mathcal{#1}}
\def\mb#1{\mathbb{#1}}
\def\tx#1{\textrm{#1}}

\def\ms#1{\mathsf{#1}}
\def\tr{\tx{tr}\,}
\def\C{\mathbb{C}}
\def\Q{\mathbb{Q}}
\def\R{\mathbb{R}}
\def\Z{\mathbb{Z}}
\DeclareMathOperator\Gal{Gal}

\def\sm{\smallsetminus}
\def\<{\langle}
\def\>{\rangle}

\theoremstyle{plain}
\newtheorem{thm}[equation]{Theorem}
\newtheorem{lem}[equation]{Lemma}
\newtheorem{pro}[equation]{Proposition}
\newtheorem{cor}[equation]{Corollary}

\newtheorem{exa}[equation]{Example}
\theoremstyle{definition}
\newtheorem{dfn}[equation]{Definition}
\newtheorem{dfnlem}[equation]{Definition/Lemma}

\newtheorem{rem}[equation]{Remark}
\newtheorem{notn}[equation]{Notation}

\sectionfont{\center\sc\normalsize}
\subsectionfont{\bf\normalsize}

\numberwithin{equation}{section}

\DeclareMathOperator\Ad{Ad}

\DeclareMathOperator\Aut{Aut}
\DeclareMathOperator\Cent{Cent}
\DeclareMathOperator\disc{disc}
\DeclareMathOperator\fNorm{N}

\DeclareMathOperator\PGL{PGL}

\DeclareMathOperator\Hom{Hom}

\DeclareMathOperator\Ind{Ind}
\DeclareMathOperator\ind{Ind}

\DeclareMathOperator\inv{inv}
\DeclareMathOperator\Irr{Irr}
\DeclareMathOperator\Lang{L}
\DeclareMathOperator\Lie{Lie}
\DeclareMathOperator\Or{O}
\DeclareMathOperator\ord{ord}
\DeclareMathOperator\Pin{Pin}
\DeclareMathOperator\Res{Res}
\DeclareMathOperator\rk{rk}
\DeclareMathOperator\sgn{sgn}

\DeclareMathOperator\SL{SL}
\DeclareMathOperator\SO{SO}
\DeclareMathOperator\Spin{Spin}
\DeclarePairedDelimiterX\pair[2]\<\>{#1, #2}

\newcommand\card[1]{\lvert#1\rvert}
\newcommand\ldef{\mathrel{:=}}

\newcommand\pleft[1]{\mathchoice{\Bigl#1}{\bigl#1}{#1}{#1}}
\newcommand\pright[1]{\mathchoice{\Bigr#1}{\bigr#1}{#1}{#1}}
\newcommand\res[1][\relax]{\mathclose{#1|}}
\newcommand\sdiff{\mathop\Delta}
\newcommand\stbar{\mathrel|}

\renewcommand\^[1]{\prescript{#1}{}}
\newcommand\abgp{_\tx{ab}}
\newcommand\scab{_\tx{sc,ab}}
\newcommand\adform{_\tx{ad}}
\newcommand\scform{_\tx{sc}}
\newcommand\dergp{_\tx{der}}
\newcommand\scder{_\tx{sc,der}}
\newcommand\Gad{G\adform}
\newcommand\Gsc{G\scform}
\newcommand\Mad{M\adform}
\newcommand\Msc{M\scform}
\newcommand\Mscab{M\scab}
\newcommand\Mscder{M\scder}

\newcommand\Tsc{T\scform}
\newcommand\gsc{\fg\scform}
\newcommand\msc{\fm\scform}
\newcommand\tsc{\ft\scform}
\newcommand\anondot{{-}}
\newcommand\dota{\cdot}
\newcommand\dotm{\cdot}
\newcommand\subb{_\tx b}
\newcommand\hO{\widehat O}
\newcommand\STheta{S\Theta}

\def\lmod{\backslash}
\def\mb#1{\mathbb{#1}}

\newcommand{\bC}{\mathbb{C}}

\newcommand{\bF}{\mathbb{F}}

\newcommand{\bR}{\mathbb{R}}

\newcommand{\cB}{\mathcal{B}}

\newcommand{\fg}{\mathfrak{g}}
\newcommand{\fh}{\mathfrak{h}}

\newcommand{\fm}{\mathfrak{m}}

\newcommand{\fp}{\mathfrak{p}}

\newcommand{\ft}{\mathfrak{t}}

\newcommand{\fS}{\mathfrak{S}}

\newcommand{\wt}{\widetilde}
\newcommand{\wh}{\widehat}
\newcommand{\eps}{\epsilon}
\DeclareMathOperator{\Sp}{Sp}

\usepackage{marginnote}

\newcommand{\sepfield}{{\bar F}}

\begin{document}
\title{A twisted Yu construction, Harish-Chandra characters, and endoscopy}
\author{\renewcommand\footnotemark{\relax}
Jessica Fintzen\thanks{J.F.\ is supported in part by NSF grant DMS-1802234/DMS-2055230 and a Royal Society University Research Fellowship.}\and
\renewcommand\footnotemark{\relax}
Tasho Kaletha\thanks{T.K.\ is supported in part by NSF grant DMS-1801687 and a Sloan Fellowship.}\and
\renewcommand\footnotemark{\relax}
Loren Spice\thanks{L.S.\ is supported in part by Simons grant 636151.}
\thanks{The authors gratefully acknowledge the hospitality of
the American Institute of Mathematics, which supported this research
as part of a SQuaRE program.}}
\maketitle

\begin{abstract}
We give a modification of Yu's construction of supercuspidal representations of a connected reductive group $G$ over a non\-archimedean local field $F$.
This modification restores the validity of certain key intertwining property claims made by Yu, which were recently proven to be false for the original construction. This modification is also an essential ingredient in the construction of supercuspidal $L$-packets in \cite{KalSLP}. As further applications, we prove the stability and many instances of endoscopic character identities of these supercuspidal $L$-packets, subject to some conditions on the base field $F$. In particular, for regular supercuspidal parameters we prove all instances of standard endoscopy. In addition, we prove that these supercuspidal $L$-packets 
satisfy \cite[Conjecture 4.3]{KalDC}, which, together with standard endoscopy, uniquely characterizes the local Langlands correspondence for supercuspidal $L$-packets (again subject to the conditions on $F$). These results are based on a statement of the Harish-Chandra character formula for the supercuspidal representations arising from the twisted Yu construction.
\end{abstract}

\tableofcontents

\section{Introduction}

The construction of supercuspidal representations of reductive $p$-adic groups provided by Yu (\cite{Yu01}) has been widely used. It was recently noticed by one of us (L.S.)\ that the proofs of two essential results in this construction, \cite[Proposition 14.1 and Theorem 14.2]{Yu01}, are incorrect, due to the usage of a misstated lemma from \cite{gerardin:weil}. In fact, \cite[\S4]{Fin19} provided counterexamples showing that the results themselves are false. These results are essential in Yu's proof that his construction produces irreducible supercuspidal representations. A different proof was given in \cite[Theorem 3.1]{Fin19} that Yu's construction still produces irreducible supercuspidal representations, despite the failure of \cite[Proposition 14.1 and Theorem 14.2]{Yu01}.

However, \cite[Proposition 14.1 and Theorem 14.2]{Yu01} turn out to have important consequences beyond their original intention. For example, they are used in \cite{Spice17} and the forthcoming \cite{Spice21} for the computation of the Harish-Chandra characters of the representations produced by Yu's construction, generalizing earlier work \cite{AS09}.  Moreover, even in special cases where the Harish\-Chandra characters have already been computed, such as \cite{AS09} or \cite{KalRSP}, the resulting formulas contain auxiliary characters that interfere with applications, and whose existence can be traced back to the failure of \cite[Proposition 14.1 and Theorem 14.2]{Yu01}. From this point of view, restoring the validity of \cite[Proposition 14.1 and Theorem 14.2]{Yu01} becomes desirable.  We accomplish this by modifying Yu's construction.

The modification is obtained by twisting the data in the original construction by a sign character $\epsilon_x$. The construction of this sign character and the computation of its values in Theorem \ref{thm:main} are the technical heart of this paper.  Before we discuss Theorem \ref{thm:main}, we list some applications.

Our first main result states that the twisted Yu construction satisfies the analogues of \cite[Proposition 14.1 and Theorem 14.2]{Yu01}, and, as a consequence, produces irreducible supercuspidal representations, cf.\ Corollaries \ref{cor:14.1} and \ref{cor:14.2}, and Theorem \ref{thm:twistedYu}. Our second main result is a formula for the Harish-Chandra character function of regular (or more generally, non\-singular and possibly reducible) supercuspidal representations, Theorem \ref{thm:char}. For shallow elements this formula is based on \cite[\S4.4]{KalRSP}. For general elements this formula uses \cite[Theorem 10.2.1]{Spice21}. The formula for shallow elements is used in an essential way in \cite{KalSLP} for the construction of supercuspidal $L$-packets. Without the twisted Yu construction obtained in this paper and the resulting character formula of Proposition \ref{pro:char}, the construction of \cite{KalSLP} would not be possible, due to the significant interference of the above mentioned auxiliary characters.

Before describing our third main result, let us discuss the last point more precisely. A formula for the Harish-Chandra character function of the supercuspidal representations arising from Yu's construction was given in \cite[Theorem 7.1]{AS09} under a strict compactness hypothesis. It was shown in \cite[\S4.4]{KalRSP} that the character formula of \cite{AS09}, adapted as in \cite[Theorem 4.28]{DS18}, is valid without the compactness assumption in the special case of regular supercuspidal representations and certain shallow elements. Moreover, after suitably interpreting the roots of unity occurring in this formula, it can be used as a guide in the construction of $L$-packets, which was done in \cite[\S5]{KalRSP}. The reinterpreted character formula, stated in \cite[Corollary 4.10.1]{KalRSP}, contained certain auxiliary characters that from a philosophical point of view should not have been there. More precisely, the character formula for a regular supercuspidal representation $\pi_{(S,\theta)}$ at a shallow regular element $\gamma \in S(F)$ stated in \cite[Corollary 4.10.1]{KalRSP} is
\[ e(G)\epsilon_L(X^*(T)_\C-X^*(S)_\C,\Lambda)\!\!\!\!\!\!\!\!\!\!\!\!\sum_{w \in N(S,G)(F)/S(F)}\!\!\!\!\!\!\!\!\!\!\!\!\Delta_{II}^\tx{abs}[a,\chi'](\gamma^w)\epsilon_{f,\tx{ram}}(\gamma^w)\epsilon^\tx{ram}(\gamma^w)\theta(\gamma^w). \]
The auxiliary characters are $\epsilon_{f,\tx{ram}}$ and $\epsilon^\tx{ram}$. There is also another part of the formula that is less natural than it should be, and it involves the construction of the $\chi$-data $\chi'$ in \cite[(4.7.2)]{KalRSP}. We will delay this discussion to the body of this paper (\S\ref{sub:root}), and limit ourselves in this introduction to the remark that the $\chi$-data $\chi'$ fail to reflect the functoriality on the Galois side that is suggested by the inductive nature of Yu's construction.

While these auxiliary terms did not present a mathematical problem for the construction of regular supercuspidal $L$-packets, they do become a problem if one wants to go beyond the regular case. The reason is discussed in detail in the introduction of \cite{KalSLP}. The modification of Yu's construction obtained in this paper removes the auxiliary characters $\epsilon_{f,\tx{ram}}$ and $\epsilon^\tx{ram}$ from the character formula and leads to a more natural construction of the $\chi$-data used in it. More precisely, Proposition \ref{pro:char} states that the character formula
for a regular supercuspidal representation $\pi_{(S,\theta)}$ at a shallow regular element $\gamma \in S(F)$ is given by
\[ e(G)\epsilon_L(X^*(T)_\C-X^*(S)_\C,\Lambda)\sum_{w \in N(S,G)(F)/S(F)}\Delta_{II}^\tx{abs}[a,\chi''](\gamma^w)\theta(\gamma^w). \]
We have denoted here by $\chi''$ the new $\chi$-data, which now reflect functoriality along $L$-embeddings between the various twisted Levi subgroups that are part of the Yu datum. This formula applies more generally to non-singular supercuspidal representations. With the auxiliary characters now removed, this formula allows the construction of supercuspidal $L$-packets to proceed beyond the regular case, which is done in \cite{KalSLP}.

Our third main result, Theorem \ref{thm:endo}, is the proof of stability and endoscopic transfer of the regular supercuspidal $L$-packets constructed in \cite{KalRSP}, and more generally the stability and certain cases of endoscopic transfer of the torally wild supercuspidal $L$-packets constructed in \cite{KalSLP}. This proof is valid when the characteristic of $F$ is zero and the residual characteristic $p$ is larger than a certain bound depending on $G$. We note that this condition on $p$ implies in particular that every supercuspidal Langlands parameter is torally wild. To state our result slightly more precisely, consider a supercuspidal parameter $\varphi : W_F \to \^LG$ and let $S_\varphi=\Cent(\varphi,\wh G)$. There is an associated torus $S$ defined over $F$, a character $\theta : S(F) \to \C^\times$, and a canonical exact sequence
\[ 1 \to \wh S^\Gamma \to S_\varphi \to \Omega(S,G)(F)_\theta \to 1, \]
where $\Omega(S,G)$ is a subgroup of the automorphism group of $S$ that is identified with the Weyl group coming from $G$ in a natural way. When $\varphi$ is regular, $\Omega(S,G)(F)_\theta$ is trivial, hence $S_\varphi \cong \wh S^\Gamma$. In this paper we prove, for any supercuspidal parameter $\varphi$, the stability of the $L$-packet $\Pi_\varphi$ and the endoscopic character identities for all $s \in \wh S^\Gamma \subset S_\varphi$, subject to the above mentioned conditions on the base field $F$.

On the way to proving the stability of supercuspidal $L$-packets, we also prove that they satisfy \cite[Conjecture 4.3]{KalDC}. We recall that this conjecture, together with the endoscopic character identities, uniquely characterizes the local Langlands correspondence for supercuspidal parameters, again subject to the conditions on $F$.

Let us also mention that the twisted Yu construction of this paper aligns with the geometric construction of unramified toral supercuspidal representations of Chan--Ivanov \cite{ChanIvanov21}. This was noticed by Chan--Oi, who computed some character values of these representations and noticed the absence of the auxiliary character $\epsilon^\tx{ram}$ (the auxiliary character $\epsilon_{f,\tx{ram}}$ is automatically trivial in the unramified setting), cf. \cite[Theorem 7.2]{ChanOi}.

Having described the applications of our main technical result, Theorem \ref{thm:main}, let us now briefly discuss its proof, and in particular the construction of the sign character $\epsilon_x$. It is performed in the following abstract situation. Let $F$ be a non-archimedean local field whose residual characteristic is not $2$, $G$ a connected adjoint $F$-group that splits over a tame extension of $F$, $M \subset G$ a tame twisted Levi subgroup, $x$ a point in the enlarged Bruhat--Tits building of $M$, and $X \in \Lie^*(M\scab)(F)$ a $G$-generic element, where $M\scab$ is the maximal abelian quotient of the preimage of $M$ in the simply connected cover of $G$. To these data we associate a character $\epsilon^{G/M}_x : M(F)_x \to \{\pm1 \}$, which is the product of three distinct characters --- a character $\_M\epsilon_\tx{sym,ram}$ obtained from the Moy--Prasad filtration, a character $\_M\epsilon^\tx{sym,ram}_s$ obtained via a hypercohomology construction, and a character $\_M\epsilon_{0}$ obtained via a spinor-norm construction.
The hypercohomology and spinor-norm constructions are quite general and work over arbitrary fields of odd characteristic; here they are applied to the residue field of $F$.

For every tame maximal torus $T \subset M$ whose enlarged building contains $x$, we give an explicit formula for the restriction of $\epsilon^{G/M}_x$ to $T(F)_x$. This formula has the property that it recovers the product of the auxiliary characters $\epsilon^\tx{ram}$ and $\epsilon_{f,\tx{ram}}$ mentioned above, together with a third character $\epsilon_\flat$ which measures the difference between the two versions $\chi'$ and $\chi''$ of the $\chi$-data in the Harish-Chandra character formula. In order to unify notation, we have denoted $\epsilon^\tx{ram}$ by $\epsilon_{\sharp, x}$ and $\epsilon_{f,\tx{ram}}$ by $\epsilon_f$, and we have decorated each of them by the superscript $G/M$. Part of the character $\epsilon_\flat$, denoted here by $\epsilon_{\flat,0}$, appeared in \cite[Proposition 5.27]{KalDC}.

In the application of $\epsilon^{G/M}_x$ to Yu's construction, we are given a Yu datum for an irreducible supercuspidal representation, which in particular contains a tower of twisted Levi subgroups $G^0 \subset G^1 \subset \dotsb \subset G^d=G$ as well as a point $x$ in the enlarged Bruhat--Tits building of $G^0$. The character $\epsilon_x$ by which we twist this construction is the product over all $i$ of characters derived in a straightforward way from $\epsilon_x^{G^{i+1}/G^i}$, cf. \S\ref{sub:Yu}, in particular Definition \ref{dfn:epsGM}.

The essential challenge in the proof of Theorem \ref{thm:main} was that the constructions of the three pieces $\_M\epsilon_\tx{sym,ram}$, $\_M\epsilon^\tx{sym,ram}_s$, and $\_M\epsilon_{0}$
 of $\epsilon^{G/M}_x$ were in no way suggested by the shape of the auxiliary characters $\epsilon_{\sharp, x}$ and $\epsilon_f$ that reflect the issues with Yu's original construction, or the character $\epsilon_{\flat,0}$ that was presented by considerations with the dual side. In fact, none of the three pieces $\_M\epsilon_\tx{sym,ram}$, $\_M\epsilon^\tx{sym,ram}_s$, and $\_M\epsilon_0$ is directly related to any of the three auxiliary characters $\epsilon_{\sharp, x}$, $\epsilon_f$, and $\epsilon_{\flat,0}$. The path from the latter to the former led us through many sharp turns and down many dark alleys. In fact, we were intially only concerned with extending the product $\epsilon_{\sharp, x} \dotm \epsilon_{\flat,0}$ from one fixed tame maximal torus $T(F)$ to $M(F)_x$. That there is a natural construction of a character of $M(F)_x$ which does this simultaneously for \emph{all} tame maximal tori $T \subset M$, and that it moreover produces, in addition to $\epsilon_{\sharp, x}$ and $\epsilon_{\flat,0}$, also $\epsilon_f$ and the remainder of $\epsilon_\flat$, was a great and pleasant surprise. The full force of this statement ended up being essential in our applications.

\numberwithin{equation}{section}
\section{Notation}
\label{sec:notation}

For any field $F$ we fix a separable closure of $\sepfield$. 
When we refer to a separable extension \(E/F\), we tacitly assume that
\(E \subset \sepfield\).
If \(E/F\) is finite, then we write \(\disc(E/F) \in F^\times/F^{\times,2}\) for the discriminant of the extension $E/F$, i.e., the determinant of the trace pairing, and $\fNorm_{E/F} : E^\times \to F^\times$ for the norm map.

For every quadratic extension \(E/E_\pm\)
of fields of characteristic not \(2\),
we write \(\sigma_{E/E_\pm}\) for the generator of
\(\Gal(E/E_\pm)\),
$E^1$ for the kernel of \(\fNorm_{E/E_\pm}\) (if \(E_\pm\) is understood),
and \(\Lang_{E/E_\pm}\) for the isomorphism
\(E^1 \to E^\times/E_\pm^\times\) that is inverse to
\(e \mapsto e/\sigma_{E/E_\pm}(e)\).

We denote by \(\Gamma_F\) the absolute Galois group
\(\Gal(\sepfield/F)\) and put \(\Sigma_F = \Gamma_F \times \{\pm1\}\). If \(F\) is understood, then we may abbreviate \(\Gamma_F\) and \(\Sigma_F\)
to \(\Gamma\) and \(\Sigma\), respectively.
Let \(\fS\) be a set on which \(\Sigma\) operates,
and \(i\) an element of \(\fS\).
Then \(i\) is called \emph{asymmetric} if $-i \notin \Gamma \dota i$, and \emph{symmetric} if $-i \in \Gamma \dota i$. These notions depend only on the $\Sigma$-orbit of $i$. We have the disjoint-union decomposition $\fS=\fS_\tx{asym}\sqcup \fS_\tx{sym}$ into the subsets of asymmetric, respectively symmetric, elements. We write $\Gamma_i$, respectively \(\Gamma_{\pm i}\), for the subgroup of $\Gamma$ that leaves invariant the set $\{i\}$, respectively \(\{i, -i\}\).  Thus $\Gamma_i$ is a subgroup of $\Gamma_{\pm i}$, of index $1$ if $i$ is asymmetric, and of index $2$ if $i$ is symmetric. We write $F_i$, respectively $F_{\pm i}$, for the fixed field in $\sepfield$ of $\Gamma_i$, respectively $\Gamma_{\pm i}$.  If \(i\) is symmetric, then we abbreviate \(\Lang_{F_i/F_{\pm i}}\) to \(\Lang_{\pm i}\), and \(\sigma_{F_i/F_{\pm i}}\) to \(\sigma_i\).

If $F$ is a non-archimedean local field, then we denote its ring of integers by $O_F$ and the maximal ideal of \(O_F\) by $\mf{p}_F$; write \(k_F = O_F/\mf p_F\); and let \(p_F\) and \(q_F\) be the characteristic and size, respectively, of \(k_F\).
We denote by $I_F \subset \Gamma_F$ the inertia subgroup. Again we drop the subscript if $F$ is understood. Given $i \in \fS$, we write $O_i$, $\mf{p}_i$, and $k_i$, for the ring of integers, maximal ideal, and residue field of $F_i$, and employ the analogous notation for $F_{\pm i}$. We let $F^\tx{tr}$, respectively \(F^\tx{ur}\), be the maximal tamely ramified, respectively unramified, extension of $F$ (inside $\sepfield$). A symmetric element $i$ is called \emph{unramified} if $-i \notin I \dota i$, and \emph{ramified} if $-i \in I \dota i$. The extension $[k_i:k_{\pm i}]$ is of degree $1$ if $i$ is ramified, and $2$ if $i$ is unramified. We have the disjoint-union decomposition $\fS_\tx{sym}=\fS_\tx{sym,unram} \sqcup \fS_\tx{sym,ram}$ of the symmetric elements of $\fS$ into the subsets of unramified symmetric, respectively ramified symmetric, elements.

If $G$ is a connected reductive group defined over $F$, then we write $\cB(G,F)$ for the enlarged Bruhat--Tits building of $G$. For \(x \in \cB(G, F)\),
we denote by $G(F)_x$ the stabilizer of $x$
and by $\ms{G}_x$ the
(usually disconnected) $k$-group scheme with reductive identity
component for which $\ms{G}_x(k_E)=G(E)_x/G(E)_{x,0+}$ for every
unramified extension $E/F$. Its identity component is denoted by
$\ms{G}_x^\circ$, and satisfies $\ms{G}_x^\circ(k_E)=G(E)_{x,0}/G(E)_{x,0+}$.
Here $G(E)_{x,0}$ and $G(E)_{x,0+}$ are the (connected) parahoric subgroup
of $G(E)_x$ and its pro-unipotent radical. For a non-negative real number $r$ and a separable extension $E/F$ of finite ramification degree, we denote by $G(E)_{x,r}$, respectively \(G(E)_{x, r+}\), the Moy--Prasad filtration subgroup of depth $r$, respectively the union of such subgroups of all depths greater than \(r\), normalized with respect to the valuation on $E$
that sends a uniformizer of \(F\) to \(1\).
Note that \(G(E)_x\), \(G(E)_{x, r}\), and \(G(E)_{x, r+}\), depend only on the image of \(x\) in the reduced building. We may abbreviate $G(E)_{x,r}/G(E)_{x,r+}$ to $G(E)_{x, r:r+}$. We use the analogous notation for the Lie algebra.

Given a torus $T$ defined over $F$, we write \(X^*(T)\) and \(X_*(T)\) for the character and cocharacter
lattices of \(T_{\sepfield}\), respectively. These are finitely generated free abelian groups with $\Sigma$-action.  We write $\ms{T}$ for the
special fiber of the ft-N\'eron model of \(T\). Thus for every unramified
extension $E/F$ we have $\ms{T}(k_E)=T(E)\subb/T(E)_{0+}$ and
$\ms{T}^\circ(k_E)=T(E)_0/T(E)_{0+}$, where $T(E)\subb$ is the maximal
bounded subgroup of $T(E)$, $T(E)_0$ is the Iwahori subgroup, and
$T(E)_{0+}$ the pro-unipotent radical of  $T(E)_0$. For $x \in \cB(T,F)$
we have $\ms{T}=\ms{T}_x$.

If $T$ is a torus in a connected reductive group $G$, then we write $R(T,G)$ for the set of non-zero weights of $T_{\sepfield}$ acting on $\Lie(G_{\sepfield})$. If $M \subset G$ is a connected reductive subgroup containing the torus $T$, we write $R(T,G/M) = R(T,G) \sm R(T,M)$. The sets $R(T,G)$, $R(T,M)$, and $R(T,G/M)$ are equipped with a $\Sigma$-action. We will be most interested in the cases when  $M$ is a twisted Levi subgroup of $G$ and $T$ is either a maximal torus of $G$ or the connected center of $M$.

We use $\Lie(\anondot)$ or Fraktur letters to denote Lie algebras. Thus we will write $\Lie(G)$, $\Lie(M)$, and $\Lie(T)$, or $\fg$, $\fm$, and $\ft$ for the Lie algebras of $G$, $M$, and $T$, respectively. A point $x \in \cB(G,F)$ specifies an $O_E$-lattice $\fg(E)_{x,0}$
in the $E$-vector space $\fg(E)$ for every separable extension $E/F$.
When $E/F$ has finite ramification degree and is tame we have
$\fg(F)_{x,0}=\fg(E)_{x,0}^\Gamma$.

We will write $\Ad(g) : \fg \to \fg$ for the differential of the homomorphism $x \mapsto gxg^{-1}$, and define $\Ad^*(g) : \fg^* \to \fg^*$ by $\<\Ad^*(g)X,Y\>=\<X,\Ad(g)^{-1}Y\>$.

We write $\Gad$ for the adjoint quotient of $G$, $G\dergp$ for the derived subgroup of $G$, and $\Gsc$ for the simply-connected cover of $G\dergp$. If $M$ is a twisted Levi subgroup of $G$ (for example, if \(M = T\) is a maximal torus), then we write $\Msc$ and $\Mad$ for the preimage of $M$ in $\Gsc$ and the image of $M$ in $\Gad$, respectively. Note that $\Msc$ is not the simply connected cover of the derived subgroup of $M$ unless $M=G$ (the derived subgroup of $\Msc$ is simply connected, but \(\Msc\) will have a non-trivial connected center) and $\Mad$ is not the adjoint quotient of $M$ unless $M=G$ (the center of $\Mad$ is connected, but not trivial).

Let \(T\) be a maximal torus in \(G\).
For $\alpha \in R(T,G) \subset X^*(T)$, we write $\alpha^\vee \in R^\vee(T,G) \subset X_*(T)$ for the coroot and $d\alpha^\vee$ for the corresponding element of $\Hom(\Lie(\mathbb G_{m, \sepfield}), \ft_{\sepfield})$. Write $\alpha\scform$ for the root corresponding to $\alpha$ under the identification of $R(T,G)$ and $R(\Tsc,\Gsc)$. We set  $H_\alpha=d\alpha\scform^\vee(1) \in \tsc(\sepfield)$.

If $K$ is a subgroup of $G(F)$, $\rho$ a representation of $K$ and $g \in G(F)$, then we write $\^gK = {gKg^{-1}}$, \(\^{g\cap}K = K \cap \^gK\), and $\^g\rho(k)=\rho(g^{-1} k g)$.

\numberwithin{equation}{section}
\section{Statement of Theorem \ref{thm:main}} \label{sec:main-stmt}

Let $F$ be a non-archimedean local field. We assume $p \neq 2$.
Let $G$ be a connected adjoint group defined over $F$ and split over $F^\tx{tr}$,
and let $M$ be a twisted Levi subgroup of \(G\) split over $F^\tx{tr}$. Note that the center $Z_M$ of $M$ is connected.

We fix an embedding of Bruhat--Tits buildings $\cB(M,F) \subset \cB(G,F)$ and a point $x \in \cB(G,F)$ that lies in the image of $\cB(M,F)$. We regard $x$ as an element of \(\cB(M, F)\) via the fixed embedding. The image of $x$ in the reduced building of $M$ is independent of the choice of embedding $\cB(M,F) \to \cB(G,F)$; therefore, none of the constructions depend on the chosen embeddings. Note further that, since all embeddings of $\cB(M,F)$ into $\cB(G,F)$ have the same image, the property of $x$ belonging to the image of $\cB(M,F)$ depends only on $M$ and $x$, and not on the chosen embedding. We can apply this discussion in particular to any tame maximal torus $T \subset G$.

We denote by $\Gsc$ the simply-connected cover of $G$, by $\Msc$ the preimage of $M$ in \(\Gsc\), and by \(\Mscab = \Msc/\Mscder\) and $Z_{\Msc}$ the abelianization and center of $\Msc$, respectively.
We assume there exist $X \in \Lie(\Mscab)^*(F) \subset \Lie^*(\Msc)(F)$ and $r \in \R$ such that for all tamely ramified maximal tori $T \subset M$ and all $\alpha \in R(T,G/M)$ we have $\ord(\pair X{H_\alpha})=-r$. Henceforth, we call such an element \textit{\(G\)-good}, or, if \(G\) is understood, just good. We fix one such $X$ and put $s=r/2$.

If $T \subset G$ is a tame maximal torus with $x \in \cB(T,F)$ and $\alpha \in R(T,G)$, then we write
\[ \ord_x(\alpha)
= \{t \in \R
	\stbar \fg_\alpha(F_\alpha)_{x,t+}
		\subsetneq \fg_\alpha(F_\alpha)_{x,t} \}. \]
This set is unchanged upon replacing \(F\) by an unramified
extension, but it changes upon tame ramified base change. The set $\ord_x(\alpha) \subset \R$ is stable under translation by elements of $\ord(F_\alpha^\times) \subset \R$ and satisfies
$\ord_x(\sigma\alpha)=\ord_x(\alpha)$ for every $\sigma \in \Gamma$.
It was proven in \cite[Corollary 3.11]{DS18} that \(\ord_x(-\alpha) = -\ord_x(\alpha)\) under the assumption
that \(p \ne 2\),
but in fact it holds even when \(p = 2\), at least when \(G\) is tame; see Lemma \ref{lem:q-depth}.

For $\alpha \in R(T,G/M)$ we write $\alpha_0$ for the restriction of
$\alpha$ to $Z_M$. Then the set of non-zero weights $R(Z_M,G)$ for the
action of $Z_M$ on $\fg$ is precisely the set
$\{\alpha_0\stbar\alpha \in R(T,G/M)\}$.

Given $\alpha \in R(T,G/M)$ we denote by $e_\alpha$ or $e(\alpha)$,
respectively $e_{\alpha_0}$ or $e(\alpha_0)$,
the ramification degree of the extensions $F_\alpha/F$,
respectively $F_{\alpha_0}/F$,
and by $e(\alpha/\alpha_0)$ their quotient, which
is the ramification degree of $F_\alpha/F_{\alpha_0}$. The set
$\ord_x(\alpha)$ is an $e_\alpha^{-1}\Z$-torsor in $\R$. Since $r \in e_\alpha^{-1}\Z$ the set $\ord_x(\alpha)$ is stable under translations by $r$. Therefore $s \in \ord_x(\alpha)$ if and only if $s \in \ord_x(-\alpha)$. When $\alpha$ is symmetric then $\ord_x(\alpha)=\ord_x(-\alpha)=-\ord_x(\alpha)$.

Since $p \neq 2$, the groups $k_\alpha^\times$ and $k_\alpha^1$ are finite
cyclic of even order and we write
$\sgn_{k_\alpha}$ and $\sgn_{k_\alpha^1}$ for the corresponding
unique non-trivial $\{\pm 1\}$-valued characters.

\begin{dfn} \label{dfn:lstk}
For \(\gamma \in \ms T(k)\), we define
\begin{align*}
\epsilon_{\sharp, x}^{G/M}(\gamma)
={} & \prod_{\substack{
	\alpha \in R(T,G/M)_\tx{asym}/\Sigma\\
	s \in \ord_x(\alpha)
}}
	\sgn_{k_\alpha}(\alpha(\gamma))
\dotm\prod_{\substack{
	\alpha \in R(T,G/M)_\tx{sym,unram}/\Gamma\\
	s \in \ord_x(\alpha)
}}
	\sgn_{k_\alpha^1}(\alpha(\gamma)), \\
\epsilon_{\flat,0}^{G/M}(\gamma)
={} & \prod_{\substack{
	\alpha \in R(T,G/M)_\tx{asym}/\Sigma\\
	\alpha_0 \in R(Z_M,G/M)_\tx{sym,ram}\\
	2 \nmid e(\alpha/\alpha_0)
}}
	\sgn_{k_\alpha}(\alpha(\gamma))
\dotm\prod_{\substack{
	\alpha \in R(T,G/M)_\tx{sym,unram}/\Gamma\\
	\alpha_0 \in R(Z_M,G/M)_\tx{sym,ram}\\
	2 \nmid e(\alpha/\alpha_0)
}}
	\sgn_{k_\alpha^1}(\alpha(\gamma)), \\
\epsilon_{f}^{G/M}(\gamma)
={} & \prod_{\substack{
	\alpha \in R(T,G/M)_\tx{sym,ram}/\Gamma\\
	\alpha(\gamma) \in -1 + \mf p_\alpha
}}
	f_{(G, T)}(\alpha),\\
\epsilon_{\flat,1}^{G/M}(\gamma)
={} & \prod_{\substack{
	\alpha \in R(T,G/M)_\tx{sym,ram}/\Gamma\\
	\alpha(\gamma) \in -1 + \mf{p}_\alpha
}}
	(-1)^{[k_\alpha : k]+1}
	\sgn_{k_\alpha}(e_\alpha\ell_{p'}(\alpha^\vee)),\\
\epsilon_{\flat,2}^{G/M}(\gamma)
={} & \prod_{\substack{
	\alpha \in R(T,G/M)_\tx{sym,ram}/\Gamma\\
	\alpha(\gamma) \in -1 + \mf{p}_\alpha
}}
	\sgn_{k_\alpha}(-1)^{(e(\alpha/\alpha_0) - 1)/2},\\
\intertext{and}
\epsilon_\flat^{G/M}={}& \epsilon_{\flat,0}^{G/M} \dotm \epsilon_{\flat,1}^{G/M} \dotm \epsilon_{\flat,2}^{G/M},
\end{align*}
where
\(f_{(G, T)}(\alpha)\) is the toral invariant \cite[\S4.1]{KalEpi}
and
\(\ell_{p'}(\alpha^\vee)\) is the prime-to-\(p\) part of the normalized square length of \(\alpha^\vee\) as in Definition \ref{dfn:regular-length}.
\end{dfn}

\begin{rem}
\label{rem:well-defined}
We will see in Lemma \ref{lem:r1} that, when \(\alpha\) is ramified symmetric, the integer $e(\alpha/\alpha_0)$ is odd, so that the exponent \((e(\alpha/\alpha_0) - 1)/2\) occurring in the definition of $\epsilon_{\flat,2}$ is an integer. We prove in Lemma \ref{lem:ers1} that the maps $\epsilon_f^{G/M}$, $\epsilon_{\flat,1}^{G/M}$, and $\epsilon_{\flat,2}^{G/M}$ are characters of \(\ms T(k)\).
By inflation, we regard them as characters of \(T(F)\subb\) when convenient.
\end{rem}

\begin{rem}
\label{rem:compare-notn}
The character $\epsilon_{\sharp, x}$ was denoted by
$\epsilon^\tx{ram}_{x,r/2}$ in \cite[\S4.3]{DS18} and by
$\epsilon^\tx{ram}$ in \cite[(4.3.3)]{KalRSP}, while the character
$\epsilon_{\flat,0}$ is that of \cite[Proposition 5.27]{KalDC} and was
denoted by $\delta$ in \cite{KalSLP}.
Both arise in connection with the local Langlands correspondence;
for example, \(\epsilon_{\sharp, x}\) is the obstruction to stability in
Reeder's conjectural local Langlands correspondence
(see \cite[\S6.6]{Ree08}
and \cite[Example 5.5 and Theorem 5.8]{DS18}). The character $\epsilon_f^{G/M}$ is the quotient of the characters $\epsilon_{f,\tx{ram}}$ of \cite[Definition 4.7.3]{KalRSP} for the groups $G$ and $M$.
\end{rem}

Up to now, we have fixed, in addition to the building point \(x \in \cB(G, F)\), a specific tame maximal torus \(T\).  It is important in Theorem \ref{thm:main} that a \emph{single} sign character works simultaneously for \emph{all} choices of \(T\).

\begin{thm} \label{thm:main}
There is a sign character
$\epsilon_x^{G/M} : \ms{M}_x(k) \to \{\pm 1\}$ with the
following property: for every tame maximal torus $T \subset M$ with
$x \in \cB(T,F)$ the restriction of $\epsilon_x^{G/M}$
to $\ms{T}(k)$ equals
$\epsilon_{\sharp, x}^{G/M}
\dotm \epsilon_{\flat}^{G/M}
\dotm\epsilon_{f}^{G/M}$.
\end{thm}

By inflation, we regard \(\epsilon_x^{G/M}\) as a character of \(M(F)_x\) when convenient.
Although Theorem \ref{thm:main} only claims existence, in fact we have uniqueness.  We first need a technical result, Lemma \ref{lem:p-vs-2}.   Corollary \ref{cor:unique} doesn't use the full power of this result; for that, we wait until Lemma \ref{lem:diff}.  This lemma actually remains valid for any connected, reductive group \(M\) over \(F\), whether or not it arises as a twisted Levi subgroup of an adjoint group \(G\), so we state it in that generality.

\begin{lem}
\label{lem:p-vs-2}
Let \(H\) be a connected, reductive group over \(F\), and \(x, y \in \cB(H, F)\). Let \(\chi : H(F)_x \cap H(F)_y \to \{\pm1\}\) be a smooth character that is trivial on \(T(F)\subb\) for every tame maximal torus \(T \subset H\) with \(x, y \in \cB(T, F)\).  Then \(\chi\) is trivial.
\end{lem}

\begin{proof}
By \cite[Theorem~2.38, Corollary~2.31, and Proposition~1.7(1)]{Spice08}, it suffices to show that \(\chi\) is trivial on topologically \(p\)-unipotent elements, in the sense of \cite[Definition~1.3]{Spice08}, and on absolutely \(F\)-semisimple elements, in the sense of \cite[Definition~2.15]{Spice08}.  Since \(p \ne 2\), we have that \(\chi\) is trivial on topologically \(p\)-unipotent elements \(\gamma\), i.e. the elements that satisfy \(\lim_{n \to \infty} \gamma^{p^n} = 1\).

Now let \(\gamma\) be an absolutely \(F\)-semisimple, hence tame \cite[Corollary 2.37]{Spice08}, element of \(H(F)_x \cap H(F)_y\).  By \cite[Proposition~2.33]{Spice08}, the Bruhat--Tits building $\cB(\Cent_H(\gamma),F)$ of the centralizer of $\gamma$ in \(H\) can be identified with the fixed-point set of \(\gamma\) in $\cB(H, F)$, hence, in particular, contains \(x\) and \(y\); so there is a tame maximal torus \(T\) in \(\Cent_H(\gamma)\) whose building contains $x$ and $y$.  By assumption, \(\chi\) is trivial on \(T(F)\subb\), hence, in particular, on \(\gamma\).
\end{proof}

\begin{cor}
\label{cor:unique}
The character \(\epsilon_x^{G/M}\) of Theorem \ref{thm:main} is unique.
\end{cor}
\begin{proof}
Lemma \ref{lem:p-vs-2}, with \(y = x\), shows that a character of \(M(F)_x\) is determined by its restrictions to \(T(F)\subb\) for the various tame maximal tori \(T \subset M\) with \(x \in \cB(M, F)\).
\end{proof}

\numberwithin{equation}{subsection}

\section{Applications of Theorem \ref{thm:main}}
\newcommand\xad{x\adform}

Throughout this section, $G$ denotes a connected reductive group over $F$ that splits over $F^\tx{tr}$. Before proving Theorem \ref{thm:main}, we will show how it resolves some problems in representation theory and harmonic analysis. The first problem is the failure of the intertwining results \cite[Proposition~14.1 and Theorem~14.2]{Yu01} in Yu's construction of supercuspidal representations. We produce in \S\ref{sub:Yu} a twisted version of Yu's construction for which Yu's proofs show that these intertwining results do hold.  In \S\ref{sub:char} we give a formula for the Harish-Chandra character of a regular supercuspidal representation (more generally, a non-singular supercuspidal Deligne--Lusztig packet) that has been constructed using the twisted Yu construction of \S\ref{sub:Yu}. This formula is based on the computations of roots of unity in \S\ref{sub:root}, where Theorem \ref{thm:main} is actually applied. In \S\ref{sub:endo} we apply the formula of \S\ref{sub:char} to prove the stability and endoscopic character identities for regular supercuspidal $L$-packets (as well as certain character identities for the more general supercuspidal $L$-packets).

\subsection{Twist of Yu's construction and validity of intertwining results} \label{sub:Yu}

To obtain the twisted Yu construction, we need a character $\eps$ that has a certain intertwining property (see Proposition \ref{pro:14.1}). The existence of such a character follows from our main Theorem \ref{thm:main}, which actually provides a distinguished one. This was in fact one of the motivations for this theorem.

Before we recall the specificities of Yu's construction and how we modify it, we extract the key information we need from Theorem \ref{thm:main}. Let $M \subset G$ be a tame twisted Levi subgroup. For $x \in \cB(M,F)$, we denoted by $[x]$ the images of $x$ in the reduced building of $G$ (not of $M$!) and by $M(F)_{[x]}$ the stabilizer of $[x]$ in $M(F)$.

\begin{dfn} \label{dfn:uxy}
Let $x,y \in \cB(M,F)$. Let $r \in \bR$ and set $s=r/2$. 
\begin{enumerate}
\item For $t=s$ or $t=s+$, we write $(M,G)(F)_{x,(r,t)}$ for the group
$$ G(F) \cap \left<T(E)_{r},\ U_\alpha(E)_{x,r},\ U_\beta(E)_{x,t} \, | \, \alpha \in R(T,M),\ \beta \in R(T,G/M) \, \right>,$$
where $T$ is a maximal torus of $M$ that splits over a tamely ramified extension $E/F$ with $x \in \cB(T,F) \subset \cB(G,F)$, and
$U_\alpha(E)_{x,r}$ denotes the Moy--Prasad filtration subgroup of depth $r$ at $x$ of the root group $U_\alpha(E) \subset G(E)$ corresponding to the root $\alpha$, and similarly for \(U_\beta(E)_{x, t}\).
\item Let \(U^x_y\) be the quotient of
\[
((M, G)(F)_{x, (r, s)} \cap (M, G)(F)_{y, (r, s+)})(M, G)(F)_{x, (r, s+)}
\]
by
\[
(M, G)(F)_{x, (r, s+)}.
\]
\item For $m \in M(F)_{[x]} \cap M(F)_{[y]}$ let \(\delta^x_y(m) = \sgn_{\bF_p}(\det\nolimits_{\bF_p}(c_m\res_{U^x_y}))\), where $c_m$ is the action of $m$ on $U^x_y$ by conjugation.
\end{enumerate}
\end{dfn}

We have used that $U^x_y$ has a natural structure of an $\bF_p$-vector space, being an abelian $p$-group, and that the conjugation action of $G(F)$ on itself induces an action of $M(F)_{[x]} \cap M(F)_{[y]}$ on $U^x_y$.

We now formulate the key result, Lemma \ref{lem:diff}, that connects Theorem \ref{thm:main} to Yu's construction. The usefulness of Lemma \ref{lem:diff} stems from Proposition \ref{pro:14.1} below. Recall that we write \(\Msc\) for the pre-image of \(M\) in \(\Gsc\), and \(\Mad\) for the image of \(M\) in \(\Gad\).

\begin{lem} \label{lem:diff} 
Assume the existence of an element $X$ of $\tx{Lie}^*(M\scab)(F)$ that satisfies  $\ord(\pair X{H_\alpha})=-r$ for all $\alpha \in R(T,G/M)$. For every $y \in \cB(M,F)$ denote by $\epsilon^{G/M}_y$ the pull back to $M(F)_{[y]}$ of the inflation of $\epsilon^{\Gad/\Mad}_{[y]}$ of Theorem \ref{thm:main}. Then
\begin{enumerate}
\item\label{lem:diff:conj} For every $y \in \cB(M,F)$ and \(m \in M(F)\), we have that \(\^{m^{-1}}{\smash{\epsilon_{m\dota y}^{G/M}}}\) equals \(\epsilon_{y}^{G/M}\). 
\item\label{lem:diff:move} For every \(x,y \in \cB(M, F),\) 
\[
(\eps_x^{G/M}\res_{M(F)_{[x]} \cap M(F)_{[y]}})\delta^x_y
= (\eps_y^{G/M}\res_{M(F)_{[x]} \cap M(F)_{[y]}})\delta^y_x.
\]
\end{enumerate}
\end{lem}

\begin{proof}
We claim that the quotient map $G \to \Gad$ maps $U^x_y$ isomorphically onto the analogous space defined in terms of $\Gad$. To see this, write $\bar U^x_y$ for the latter space. Since the quotient map restricts to a map $(M,G)(F)_{x,(r,s)} \to (\Mad,\Gad)(F)_{x,(r,s)}$ and analogously for $(r,s+)$, it induces a map $U^x_y \to \bar U^x_y$. According to \cite[Corollary 2.3]{Yu01}, $U^x_y = (\_EU^x_y)^{\Gal(E/F)}$ and \(\bar U^x_y = (\_E{\smash{\bar U}}^x_y)^{\Gal(E/F)}\), where $E/F$ is a finite tamely ramified Galois extension splitting a maximal torus $T$ of $M$ with $x \in \cB(T,F)$, and $\_EU^x_y$ and \(\_E{\smash{\bar U}}^x_y\) are constructed in the same way as \(U^x_y\) and \(\bar U^x_y\) but relative to $E$. Therefore it is enough to show that $\_EU^x_y \to \_E{\smash{\bar U}}^x_y$ is an isomorphism. But the quotient map $G \to \Gad$ is an isomorphism when restricted to each root group with respect to $T$ and this isomorphism respects the filtrations induced by $x$ and $y$, respectively. While the quotient is not an isomorphism when restricted to $T$, the contribution of $T(E)$ to $U_x^y$ is trivial. This proves that $_EU^x_y \to {_E\bar U^x_y}$ is an isomorphism, and the claim follows.

The claim implies that the characters $\epsilon_x^{G/M}$, $\epsilon_y^{G/M}$, $\delta_x^y$, and $\delta_y^x$, are inflated from $\Mad(F)_{[x]}$, $\Mad(F)_{[y]}$ and $\Mad(F)_{[x]} \cap \Mad(F)_{[y]}$, respectively. We may thus assume in this proof that $G$ is adjoint. Then $x=[x]$, so we drop the brackets from the notation.

By Corollary \ref{cor:unique}, for (\ref{lem:diff:conj}), it suffices to show that the desired equality holds for the restrictions to \(T(F)\subb\), where \(T \subset M\) is a tame maximal torus with \(x \in \cB(M, F)\).  For every such torus, we have that \(\alpha \mapsto \^m\alpha\) is a length-preserving, \(\Sigma\)-equivariant bijection \(R(T, G/M) \to R(\^mT, G/M)\).  Since, for all \(\alpha \in R(T, G/M)\), we have \(\ord_x(\alpha) = \ord_{m\dota x}(\^m\alpha)\) and \(f_{G, \^mT}(\^m\alpha) = f_{G, T}(\alpha)\), and \(\^m\alpha(m\gamma m^{-1}) = \alpha(\gamma)\) for all \(\gamma \in T(F)\subb\), it follows that the restriction of \(\^{m^{-1}}{\smash{\epsilon_{m\dota x}^{G/M}}}\) to \(T(F)\subb\) is \(\epsilon_{\sharp, x}^{G/M}\epsilon_\flat^{G/M}\epsilon_f^{G/M}\).

By Lemma \ref{lem:p-vs-2}, for (\ref{lem:diff:move}), it suffices to show the desired equality for the restrictions to \(T(F)\subb\), where now \(T \subset M\) is a tame maximal torus with \(x, y \in \cB(M, F)\); which, since the other characters in Definition \ref{dfn:lstk} do not change when we replace \(x\) by \(y\), reduces to showing that \(\epsilon_{\sharp, x}^{G/M}\delta^x_y = \epsilon_{\sharp, y}^{G/M}\delta^y_x\). 

Note that $U_y^x$ is equipped with a $k$-vector space structure via the Moy--Prasad isomorphism, and the $\bF_p$-structure of $U_y^x$ is obtained by forgetting the $k$\-structure. Therefore, $\delta_y^x = \sgn_k(\det\nolimits_k(\anondot\res_{U^x_y}))$.

We now compute the \(k\)-structure more precisely.
Put \(\lambda = y - x\), regarded as a \(\Gamma\)-invariant element of \(X_*(T) \otimes_\Z \R\).  Then \(U^x_y\) is identified with
\[
\Bigl(\bigoplus_{\substack{
    \alpha \in R(T, G/M) \\
    s \in \ord_x(\alpha) \\
    \pair\alpha\lambda > 0
}}
    \fg_\alpha
\Bigr)(F)_{x, s:s+},
\]
so that, for any \(\gamma \in T(F)\subb\),
\[
\delta^x_y(\gamma) = \prod_{\substack{
    \alpha \in R(T, G/M)/\Gamma \\
    s \in \ord_x(\alpha) \\
    \pair\alpha\lambda > 0
}}
    \sgn_{k_\alpha}(\alpha(\gamma)).
\]
If \(\alpha\) is symmetric, then \(\pair\alpha\lambda = \pair{\sigma_\alpha\alpha}{\sigma_\alpha\lambda} = \pair{-\alpha}\lambda\), so that \(\pair\alpha\lambda = 0\).  Therefore,
\[ 
\delta^x_y(\gamma) = \prod_{\substack{
    \alpha \in R(T, G/M)_\tx{asym}/\Sigma \\
    s \in \ord_x(\alpha) \\
    \pair\alpha\lambda \ne 0
}}
    \sgn_{k_\alpha}(\alpha(\gamma)).
\]
Combining this with the definition of $\epsilon_{\sharp,x}^{G/M}$ we obtain
\[
\eps_{\sharp, x}^{G/M}(\gamma)\delta^x_y(\gamma)
= \prod_{\substack{
    \alpha \in R(T, G/M)_\tx{asym}/\Sigma \\
    s \in \ord_x(\alpha) \\
    \pair\alpha\lambda = 0
}}
    \sgn_{k_\alpha}(\alpha(\gamma))
\dotm\prod_{\substack{
    \alpha \in R(T, G/M)_\tx{sym,unram}/\Sigma \\
    s \in \ord_x(\alpha) \\
    \pair\alpha\lambda = 0
}}
    \sgn_{k_\alpha^1}(\alpha(\gamma)).
\]
We have an analogous equality for \(\eps_{\sharp, y}^{G/M}(\gamma)\delta^y_x(\gamma)\), with \(\ord_x(\alpha)\) replaced by \(\ord_y(\alpha)\).  However, when \(\pair\alpha\lambda = 0\), the statements \(s \in \ord_x(\alpha)\) and \(s \in \ord_y(\alpha)\) are equivalent.  The result follows.
\end{proof}

We now come to Yu's construction and its twist. The construction of \cite[\S4]{Yu01} involves a tower of twisted Levi subgroups; but due to its inductive nature most of the time it only refers to two consecutive subgroups.  Accordingly, we may continue to work with a tame twisted Levi subgroup  \(M\) of \(G\), and a point $x \in \cB(M,F)$. We recall that we write \([x]\) for the image of \(x\) in \(\cB(\Mad, F) \subset \cB(\Gad, F)\). We assume there exists and fix a character $\phi$ of $M(F)$ that is $G$-generic relative to $x$ of depth $r$ for some non-negative real number $r$ in the sense of \cite[\S9, p.~599]{Yu01}, and put \(s = r/2\). To the character $\phi$ one can associate a generic element of the dual lie algebra of $M$. Since our convention differs slightly from Yu's, we first explain what we mean.

\begin{rem}
\label{rem:generic}
In \cite[\S8]{Yu01}, Yu uses elements of $\Lie^*(Z_{M}^\circ)(F)$ to represent characters of $M(F)$, while our good elements lie in $\Lie^*(\Mscab)(F)$. 

Yu states that the map $\Lie^*(M)^M \to \Lie^*(Z_M^\circ)$ obtained by restriction is an isomorphism, and regards elements of $\Lie^*(Z_M^\circ)$ as elements of $\Lie^*(M)^M$ via this identification. However, this map is not always an isomorphism. For example, assume that $F$ has characteristic $p$ and let $S$ be an anisotropic torus in $\SL_2$. Then $M=\PGL_p \times S$ is a twisted Levi subgroup of $G=\PGL_p \times \SL_2$, $\Lie^*(M)^M$ is $2$-dimensional, having a 1-dimensional contribution from $\Lie^*(\PGL_p)^{\PGL_p}$ given by the trace map, while $\Lie^*(Z_M^\circ)$ is 1-dimensional.

On the other hand, we show in Lemma \ref{lem:hiss} that the natural map $\Lie^*(\Mscab) \to \Lie^*(\Msc)^{\Msc}$ is an isomorphism (assuming $p\neq 2$). Therefore, if we interpret Yu's generic elements as lying in \(\Lie^*(M)^M(F)\), then such an element can be mapped to $\Lie^*(\Msc)^{\Msc}(F)$ by pulling back along the natural map $\Msc \to M$, and then transporting it to $\Lie^*(\Mscab)(F)$, thereby producing a good element in our sense.
\end{rem}

\begin{rem}
\label{rem:dual-Lie-subspace}
Whenever \(H\) is a subgroup of \(G\), we have that \(\fh^*\) is naturally a \emph{quotient} of \(\fg^*\), via the restriction map.

When \(H\) is a twisted Levi subgroup of $G$, this quotient has a natural section that allows us also to regard \(\fh^*\) as a \emph{subspace} of \(\fg^*\).  Namely, we have the $H$-invariant decomposition \(\fg=\fh \oplus \fh^\perp\), where $\fh^\perp$ is the sum of the weight spaces for $Z_H^\circ$ corresponding to non-trivial weights, and the annihilator of $\fh^\perp$ in $\fg^*$ provides the desired section. 
\end{rem}

\begin{lem} \label{lem:hiss}
The natural maps $\Lie^*(\Mscab) \to \Lie^*(\Msc)^{\Msc}$ and $\Lie^*(\Gsc)^{\Msc} \to \Lie^*(\Msc)^{\Msc}$ are isomorphisms (recall that we have assumed $p\neq 2$).
\end{lem}
\begin{proof}
Let us assume that $G$ is simply connected and drop the subscripts $\tx{sc}$. For the second map we let $\fm^\perp$ be the direct sum of the weight spaces for the action of $Z_M^\circ$ on $\fg$ corresponding to non-zero weights. Then $\fg=\fm \oplus \fm^\perp$ is an $M$-invariant decomposition. Since $Z_M^\circ$ has no non-zero invariants in $\fm^\perp$, neither does $M$, and we conclude that the second map is an isomorphism, as claimed.

For the first map we consider the exact sequence $0 \to \Lie(M\dergp) \to \Lie(M) \to \Lie(M\abgp) \to 0$, where $M\dergp$ is the derives subgroup of $M$, automatically simply connected. To prove the claim it is enough to show that $\Lie^*(M\dergp)$ has no $M$-invariants.

Write $H=M\dergp$, a semi-simple simply connected group. We will show that $\fh^*$ has no non-zero $H$-invariants. This problem immediately reduces to the case that $H$ is (almost) simple. Consider an $H$-invariant element of $\fh^*$ and let $V \subset \fh$ be its kernel, an $H$-invariant subspace of codimension $0$ or $1$. We want to show $V=\fh$. The structure of the $H$-representation $\fh$ has been studied extensively. We will use the papers \cite{Hurley82} and \cite{Hiss84}. It is possible that the Lie algebra $\fh$ has a non-trivial center, and the different possibilities are described in \cite[Theorem 3.1]{Hurley82}. Since we are assuming that the characteristic of $F$ is not $2$, the center of $\fh$ is at most 1-dimensional, and occurs only when $H$ is of type $A_{n-1}$ with $n=0$ in $F$, or when $H$ is of type $E_6$ and $3=0$ in $F$.

According to \cite[Hauptsatz]{Hiss84}, $\fh$ is an indecomposable $H$-module. Therefore $V$ must contain the center of $\fh$, for otherwise it would be an $H$-invariant complement to it. Let $V'=V/Z_{\fh}$. Using again \cite[Hauptsatz]{Hiss84} we see that $\fh/Z_{\fh}$ is an irreducible $H$-module except in the case of the group $G_2$ when $F$ has characteristic $3$, in which case the only possibly submodule has codimension equal to $7$. Therefore, in all cases, $V'=\fh/Z_{\fh}$, hence $V=\fh$.
\end{proof}

With this discussion in mind, we let $X \in \Lie^*(\Mscab)(F)_{-r}$ be a good element associated to $\phi$. While the genericity of $\phi$ depends on the point $x$, the goodness of $X$ does not. Therefore, using $X$ we can apply Theorem \ref{thm:main} to any point $y \in \mc{B}(M,F)$ and obtain a character $\epsilon^{\Gad/\Mad}_{[y]}$ of  $\Mad(F)_{[y]}/\Mad(F)_{[y],0+}$.

Following \cite[\S3, p.~591]{Yu01}, we set
\begin{equation*}
	J=(M, G)(F)_{x,(r,s)},
\quad	J_+=(M, G)(F)_{x,(r,s+)},
\quad	K'=M(F) \cap G(F)_{[x]},
\end{equation*}
where \(G(F)_{[x]}\) denotes the stabilizer of $[x]$. We have $K' = M(F)_{[x]}$, since $[x]$ is a point of $\cB(\Mad,F)$. Hence we have a map from $K'$ to $\Mad(F)_{[x]}/\Mad(F)_{[x],0+}$ that is trivial on $K' \cap G(F)_{x,0+} = M(F)_{x, 0+}$.

We extend the (restriction of the) character $\phi$ to a character $\hat \phi$ of $J$ following \cite[\S4 and \S9]{Yu01}, and we denote by $N$ the kernel of $\hat \phi$. We also equip $J/J_+$ with the  pairing $\pair\anondot\anondot$ defined by $\pair a b=\hat \phi(aba^{-1}b^{-1})$ to obtain a non-degenerate symplectic vector space over $\bF_p$ (after choosing an isomorphism between $p$-th roots of unity in $\bC^\times$ and $\bF_p$) \cite[Lemma~11.1]{Yu01}. Then $J/N$ is a Heisenberg $p$-group with center $J_+/N$ and Yu provides a symplectic action under which we can pull back the Weil--Heisenberg representation of $\Sp(J/J_+) \ltimes (J / N)$ to $K' \ltimes J$  \cite[Proposition~11.4]{Yu01}. Here the map from $K'$ to $\Sp(J/J_+)$ results from the conjugation action of $K'$ on $J/J_+$, which preserves the symplectic form \cite[Lemma~11.3]{Yu01}. Following Yu we denote the resulting representation of $K' \ltimes J$  by $\wt \phi$, and we write $\phi'$ for the representation of $K'J$ whose inflation to $K' \ltimes J$ is $(\phi\res_{K'} \ltimes 1) \otimes \wt \phi$ \cite[\S14]{Yu01}.

Note that the action of $\^{g\cap}{K'}$ on  $J/J_+$ preserves the totally isotropic subspace $U_{0,x,g}\ldef(\^g{J_+} \cap J)J_+/J_+$ of \(J/J_+\),
and we define the character $\chi_{x,g}: \^{g\cap}{(K'J)} \to \{\pm 1\}$ as follows.
First note that $(\^{g\cap}{K'})(\^{g\cap}J)=\^{g\cap}{(K'J)}$ by \cite[Lemma~13.7]{Yu01}.  Then we define $\chi_{x, g}(k') = \sgn_{\bF_p}(\det\nolimits_{\bF_p} k'\res_{U_0,x,g})$ for all \(k' \in \^{g\cap}{K'}\), and let \(\chi_{x, g}\) be trivial on \(\^{g\cap}J\).  This definition is consistent, since \(K' \cap J\) is contained in \(M(F)_{x, 0+}\), and hence acts trivially on \(J/J_+\).

\begin{pro}[Twist of {\cite[Proposition~14.1]{Yu01}}] \label{pro:14.1}
	Let $\eps: K'J \to \{ \pm1 \}$ be a quadratic character such that
	$$ {(\eps \chi_{x,g})}\res_{\^{g\cap}{(K'J)}}= {((\^g\eps) \chi_{g\dota x,g^{-1}})}\res_{\^{g\cap}{(K'J)}} $$
	for all $g\in M(F)$.
	Then for all $g\in M(F)$, we have
	$$\dim \Hom_{\^{g\cap}{(K'J)}}(\^g{(\eps\phi')},\eps\phi')=1 . $$
\end{pro}

\begin{rem}
	The statement of Proposition \ref{pro:14.1} without the twist by $\eps$ is the false statement \cite[Proposition~14.1]{Yu01}. Yu's proof relied on a misprinted version of \cite[Theorem~2.4(b)]{gerardin:weil}, and \cite[\S4]{Fin19} provides a counterexample to \cite[Proposition~14.1]{Yu01}. The misprint in \cite[Theorem~2.4(b)]{gerardin:weil} is that a required quadratic character is missing in the statement (but not the proof). We have therefore introduced the twist by $\eps$ that will allow us to apply Yu's original proof to the twisted representation to prove Proposition \ref{pro:14.1}. The existence of such an $\epsilon$ follows from Theorem \ref{thm:main} and is the content of Corollary \ref{cor:14.1}.
\end{rem}

In order to prove Proposition \ref{pro:14.1} we will follow Yu's approach \cite[\S14]{Yu01}. The mistake in Yu's proof occurs when studying the restriction of Weil representations \cite[Lemma~14.6]{Yu01}. In order to correct the statement, we introduce some of the notation that Yu uses. Let $(V, \pair\anondot\anondot)$ be a non-degenerate symplectic space over $\bF_p$ and $(\omega, W)$ the Weil--Heisenberg representation of $\Sp(V) \ltimes V^\sharp$ associated to a non-trivial character $\bF_p \to \bC^\times$, where $V^\sharp$ is the Heisenberg group of $V$ with underlying set $V \times \bF_p$.
For a subspace $U \subset V$, we write $U^\sharp$ for its preimage under the map $V^\sharp \to V$ and $U^0$ for the subspace $U \times \{0 \} \subset U^\sharp$. Let $U_0$ be a totally isotropic subspace of $V$ and let $U$ be the orthogonal complement $U_0^\perp$.  Let $P_0=\{g \in \Sp(V) \stbar g\dota U_0 \subset U_0\}$ and denote by $\chi^{U_0}$ the quadratic character $P_0 \to \{ \pm1 \}$ given by $p \mapsto \sgn(\det p\res_{U_0})$.

\begin{lem}[Correction of {\cite[Lemma~14.6]{Yu01}}] \label{lem:14.6}
	The group $U^\sharp/U_0^0$ is naturally isomorphic to $(U/U_0^0)^\sharp$ by the map $(u, c) +U_0^0 \mapsto (u+U_0, c)$. The subspace $W^{U_0^0}$ is stable under the action of $P_0 \ltimes U^\sharp$. Write $(\omega', W')$ for the composition of the surjection $P_0 \ltimes U^\sharp \to \Sp(U/U_0) \ltimes (U/U_0)^\sharp$ with the Weil--Heisenberg representation of $\Sp(U/U_0) \ltimes (U/U_0)^\sharp$ associated to the same character as $(\omega, W)$. Then the representation $W^{U_0^0}$ of $P_0 \ltimes U^\sharp$ is isomorphic to
	$\omega' \otimes (\chi^{U_0} \ltimes 1)$.
\end{lem}
\begin{proof}
	The first sentence is as in \cite[Lemma~14.6]{Yu01}, and is obvious.

	By \cite[Theorem~2.4(b)]{gerardin:weil}, corrected to include the factor \(\chi^{U_0}\) that appears in the proof \cite[p.~67, (2.19)]{gerardin:weil}, the restriction of the Weil--Heisenberg representation $(\omega, W)$ from $\Sp(V) \ltimes V^\sharp$ to $P_0 \ltimes V^\sharp$ is
	given by
	$$\Ind_{P_0 \ltimes U^\sharp}^{P_0 \ltimes V^\sharp} \bigl( \omega' \otimes (\chi^{U_0} \ltimes 1) \bigr) .$$
	Thus $W^{U_0^0}$ is the subspace of the above induction consisting of the functions that are supported by $P_0 \ltimes U^\sharp$. Hence $W^{U_0^0}$ is stable under  the action of $P_0 \ltimes U^\sharp$ and isomorphic to $ \omega' \otimes (\chi^{U_0} \ltimes 1)$ as a representation of $P_0 \ltimes U^\sharp$.
\end{proof}

Following \cite[\S14]{Yu01}, we apply Lemma \ref{lem:14.6} to the totally isotropic subspace $U_0 = U_{0,x,g}=(\^g{J_+} \cap J)J_+/J_+$ of $J/J_+$, with orthogonal complement $U = (\^{g\cap}J)J_+/J_+$. As mentioned above, the resulting action of $\^{g\cap}{K'}$ on  $J/J_+$ preserves $U_0=U_{0,x,g}$ \cite[Lemma~14.5]{Yu01}, and the composition of the action map $\^{g\cap}{K'} \to P_0$ with $\chi^{U_0}$ is the character $\chi_{x,g}$ defined before Proposition \ref{pro:14.1}.

This allows us to correct \cite[Proposition~14.7(iii)]{Yu01}  based on the correction Lemma \ref{lem:14.6} of   \cite[Lemma~14.6]{Yu01}. (\cite[Proposition~14.7(i, ii)]{Yu01} needs no correction.)

\begin{cor}[Correction of {\cite[Proposition~14.7(iii)]{Yu01}}] \label{cor:14.7}
 Denote by $\pi_1$ the ($\^g{\smash{\hat\phi}}\res_{\^g{J_+}\cap J}$)-isotypic subspace of the representation $\wt \phi$ of $K' \ltimes J$. Then $\pi_1$ is naturally a representation of $\^{g\cap}{K'} \ltimes \^{g\cap}J$, and this representation is the Weil--Heisenberg representation coming from Yu's symplectic action $(f_1, j_1)$ defined in  \cite[Proposition~14.7(i)-(ii)]{Yu01}, tensored with the character $\chi_{x,g} \ltimes 1$.
\end{cor}
\begin{proof}
	The proof is the same as Yu's proof of Proposition~14.7(iii), except that we use Lemma \ref{lem:14.6} in place of \cite[Lemma~14.6]{Yu01}. More precisely, Yu shows \cite[Lemma~14.3]{Yu01} that the ($\^g{\smash{\hat\phi}}\res_{\^g{J_+}\cap J}$)-isotypic subspace $\pi_1$ of $\wt \phi$ is the same as the subspace fixed by $j^{-1}(U_0^0)$, where $j:J/N \to (J/J_+)^\sharp$ is the special isomorphism defined in \cite[Proposition~11.4]{Yu01}. Now the result follows from Lemma \ref{lem:14.6}.
\end{proof}

\begin{proof}[Proof of Proposition \ref{pro:14.1}]
Yu's proof of \cite[Proposition~14.2]{Yu01} reduces it to showing the inequality $\dim \Hom_{\^{g\cap}{K'} \ltimes \^{g\cap}J}(\^g{(\eps\smash{\wt \phi})},\eps \wt\phi) \geq 1$. This reduction does not depend on the mis-stated result \cite[Theorem 2.4(b)]{gerardin:weil}, and we refer the reader to  \cite[\S14]{Yu01} for the details.

Thus it remains to exhibit a nontrivial element of $\Hom_{\^{g\cap}{K'} \ltimes \^{g\cap}J}(\^g{(\eps\smash{\wt \phi})},\eps\wt\phi)$. Following Yu, we consider the submodule $\eps\pi_1$ of $\eps\wt \phi$ arising from Corollary \ref{cor:14.7} and the analogous submodule $(\^g \eps)\pi_1^\bullet$ in $(\^g{\eps})\^g{\smash{\wt \phi}}$ arising from the same construction by replacing $J$ by $\^gJ$,  $x$ by $g\dota x$, and $g$ by $g^{-1}$. As Yu shows, but with the correction Corollary \ref{cor:14.7} to \cite[Proposition~14.7(iii)]{Yu01} and our twist \(\eps\) taken into account, these submodules are isomorphic to Weil--Heisenberg representations of $\^{g\cap}{K'} \ltimes \^{g\cap}J$ tensored with $\eps(\chi_{x,g} \ltimes 1)$ or $\^g\eps(\chi_{g\dota x,g^{-1}} \ltimes 1)$, respectively. Since the characters $\eps(\chi_{x,g} \ltimes 1)$ and $\^g\eps(\chi_{g\dota x,g^{-1}} \ltimes 1)$ agree on $\^{g\cap}{K'} \ltimes \^{g\cap}J$  by our assumption on $\eps$, the submodules $\eps\pi_1$ and $(\^g\eps)\pi_1^\bullet$ are isomorphic as representations of $\^{g\cap}{K'} \ltimes \^{g\cap}J$, which completes the proof.
\end{proof}

\begin{dfn} \label{dfn:epsGM}
We denote by $\eps_x^{G/M}$ the character of \(K'J\) that is trivial on \(J\), and whose restriction to \(K'\) is the composition of the above map $K' \to \Mad(F)_{[x]}/\Mad(F)_{[x],0+}$ with the sign character $\eps_{[x]}^{\Gad/\Mad}$. \end{dfn}

\begin{cor}  [Explicit twist of {\cite[Proposition~14.1]{Yu01}}] \label{cor:14.1}
	For all $g\in M(F)$, we have $\dim \Hom_{\^{g\cap}{(K'J)}}(\^g{(\smash{\eps_x^{G/M}}\phi')},\eps_x^{G/M}\phi')=1$.
\end{cor}
\begin{proof}
By Proposition \ref{pro:14.1} it suffices to show that
 $${(\eps_x^{G/M}\dotm\chi_{x,g})}\res_{\^{g\cap}{(K'J)}}= {((\^g{\smash{\eps_x^{G/M}}}) \chi_{g\dota x,g^{-1}})}\res_{\^{g\cap}{(K'J)}} $$
 for all $g\in M(F)$.   This follows from Lemma \ref{lem:diff}, since, in the notation of that result, we have \(U_{0, x, g} = U^x_{g\dota x}\) and \(U_{0, g\dota x, g^{-1}} = U^{g\dota x}_x\).
\end{proof}

\begin{cor}  [Explicit twist of {\cite[Theorem~14.2]{Yu01}}] \label{cor:14.2}
	Let $C$ be a subgroup of $K'J$ that contains $J$. Then for all $g\in M(F)$, we have
	 $$\dim \Hom_{\^{g\cap}C}(\^g{(\smash{\eps_x^{G/M}}\phi')},\eps_x^{G/M}\phi')=1.$$
\end{cor}
\begin{proof}
	The proof is the same as the proof of \cite[Theorem~14.2]{Yu01}. We quickly recall it for the reader. We have $\dim \Hom_{\^{g\cap}J}(\^g{\smash{\wt \phi}},\wt \phi)=1$ according to \cite[Proposition~12.3]{Yu01}. Since \(\eps_x^{G/M}\phi'\) and \(\wt\phi\) agree on \(J\), we infer
	$$\dim \Hom_{\^{g\cap}J}(\^g{(\eps_x^{G/M}\phi')},\eps_x^{G/M}\phi')=\dim \Hom_{\^{g\cap}J}(\^g{\phi'},\phi')=1 .  $$
	 Now the result follows immediately from Corollary \ref{cor:14.1}.
\end{proof}

Let $((G^0, \dotsc, G^d=G), x, (r_0, \dotsc, r_d), \rho, (\phi_0, \dotsc, \phi_d))$ be a generic datum as defined in \cite[\S15, p.~615]{Yu01},
and denote by $(K_\tx{Yu}, \rho_\tx{Yu})$ the compact-mod-center open subgroup and smooth representation that Yu attaches to this datum and calls $K_d$ and $\rho_d$ in \cite[\S3, p.~591]{Yu01} and \cite[\S4, p.~592]{Yu01}, respectively. In particular, $K_\tx{Yu}=G^0_{[x]}G^1(F)_{x,{r_0}/{2}}\dotsm G^d(F)_{x,r_{d-1}/2}$. For $1 \leq i \leq d$, we write $\eps_x^{G^i/G^{i-1}}$ for the quadratic character of $K_\tx{Yu}$ whose restriction to $G^0(F)_{[x]}$ is given by ${\eps_x^{G^i/G^{i-1}}}\res_{G^0(F)_{[x]}}$ defined in Definition \ref{dfn:epsGM} and that is trivial on $G^1(F)_{x,{r_0}/{2}}\dotsm G^d(F)_{x,r_{d-1}/2}$. Set $\eps\ldef\prod_{1 \leq i \leq d} \eps_x^{G^i/G^{i-1}}$.
Recall that Yu claims in \cite[Theorem~15.1]{Yu01} that the compactly induced representation $\pi\ldef\ind_{K_\tx{Yu}}^{G(F)}  \rho_\tx{Yu}$ is an irreducible supercuspidal representation of depth $r_d$.  This can be proven by other means \cite[Theorem 3.1]{Fin19}, but Yu's proof relies on the wrong statement \cite[Theorem~14.2]{Yu01} and therefore does not go through as written.  On the other hand, we can apply Yu's proof directly to the twisted construction  $\ind_{K_\tx{Yu}}^{G(F)}  (\eps\rho_\tx{Yu})$ to obtain the following result.

\begin{thm}[Explicit twist of {\cite[Theorem~15.1]{Yu01}}] \label{thm:twistedYu}
Let $((G^0, \dotsc, G^d=G), x, (r_0, \dotsc, r_d), \rho, (\phi_0, \dotsc, \phi_d))$ be a generic datum. Then the associated representation $\ind_{K_\tx{Yu}}^{G(F)} (\eps\rho_\tx{Yu})$, as defined above, is irreducible, supercuspidal, and of depth $r_d$.
\end{thm}
\begin{proof}
	This follows from Yu's proof of \cite[Theorem~15.1]{Yu01} by replacing $\phi'_{i-1}$  in his proof by $\eps_x^{G^i/G^{i-1}}\dotm\phi_{i-1}'$ for $1 \leq i \leq d$ and replacing \cite[Theorem~14.2]{Yu01} by Corollary \ref{cor:14.2}.
\end{proof}

\subsection{Change of $\chi$-data} \label{sub:root}

This section is concerned with the term $\Delta^\tx{abs}_{II}$ of the Harish-Chandra character formula for supercuspidal representations arising from Yu's construction, as stated in \cite[Corollary 4.10.1]{KalRSP}. This term can be written as a product of corresponding terms, one for each step in the twisted Levi subgroup tower $G^0 \subset \dotsb \subset G^d$ that is part of a Yu datum. We therefore consider a single step, thus a tame twisted Levi subgroup $M$ of a connected adjoint group $G$ and a $G$-generic element $X \in \Lie^*(\Mscab)(F)$ of depth $-r$. In fact, it will be enough for $X$ to be $G$-good such that $\tx{ord}(\<X,H_\alpha\>)=-r$ for all tame maximal tori $T \subset M$ and all $\alpha \in R(T,G/M)$.  Here $\Mscab$ is the abelianization of the Levi subgroup \(\Msc\) of $\Gsc$ corresponding to $M \subset G$ as in \S\ref{sec:main-stmt}.

We will need the concepts of $\chi$-data, $\zeta$-data, and $a$-data, for which we refer the reader to \cite[\S4.6]{KalRSP}.

The term $\Delta^\tx{abs}_{II}$ involves certain $\chi$-data, denoted by $\chi'$ and defined in \cite[(4.7.2)]{KalRSP}, where we recall that $a_\alpha$ in that formula is defined as $\pair X{H_\alpha}$ (the element $X$ here was denoted by $X_{d-1}^*$ in loc.\ cit.). These $\chi$-data depend not just on $M$, $G$, and $X$, but also on a tame maximal torus $T \subset M$, with respect to which the coroot $H_\alpha$ is formed. We will show how the character $\epsilon_\flat = \epsilon_{\flat,0} \dotm \epsilon_{\flat,1} \dotm \epsilon_{\flat,2}$ of Definition \ref{dfn:lstk}, which is part of the restriction to $T(F)\subb$ of the character $\epsilon_x^{G/M}$ of Theorem \ref{thm:main}, measures the passage from these $\chi$-data to different ones, defined here and denoted by $\chi''$.  These alternative \(\chi\)-data have the advantage of depending only on $M$, $G$, and $X$, not on a choice of maximal torus $T$ of $M$, because they are inflated from $\chi$-data $\chi''_0$ for $Z_M$. The $\chi$-data $\chi''_0$ specify an $L$-embedding $\^LM \to \^LG$ as explained in \cite[\S6]{KalDC} and is therefore linked to functorial transfer from $M$ to $G$. This allows for inductive arguments, which are not possible if one uses $\chi'$.

In \S\ref{sub:char} we apply the results of this section to see that the Harish-Chandra characters of supercuspidal representations obtained from the twisted Yu construction of \S\ref{sub:Yu} involve the more natural $\chi''$.

Since \(G\) and \(M\) are fixed throughout this subsection, we suppress superscripts \(G/M\), writing, for example, \(\epsilon_x\) in place of \(\epsilon_x^{G/M}\).

We assume here that $p$ does not divide the order of any bond in the Dynkin diagram for the (possibly reducible) root system $R(T,G)$. In particular, the adjusted square length $\ell_{p'}(\alpha^\vee)$ used in Definition \ref{dfn:lstk} and defined in Definition \ref{dfn:regular-length} is just the ordinary square length \(\ell(\alpha^\vee)\).
Note that this restriction is not needed for Theorem \ref{thm:main}, only for this application.

Let $\Lambda : F \to \C^\times$ be a character that is non-trivial on $O_F$ but trivial on $\mf{p}_F$ and let $\Lambda^0$ be the non-trivial character of $k_F$ whose inflation to $O_F$ equals $\Lambda\res_{O_F}$. 

\begin{lem} \label{lem:ers1}
For every $\alpha \in R(T,G)_\tx{sym,ram}$ choose $\epsilon_\alpha \in \{\pm 1\}$, subject to $\epsilon_\alpha=\epsilon_{\sigma\alpha}$ for all $\sigma \in \Gamma$. Define $\epsilon_\alpha=1$ for $\alpha \in R(T,G)^\tx{sym,ram}$. Denote also by $\epsilon_\alpha$ the unramified quadratic character of $F_\alpha$ that sends any uniformizer to $\epsilon_\alpha$. Then $(\epsilon_\alpha)$ is a set of $\zeta$-data in the sense of \cite[Definition 4.6.4]{KalRSP} and the corresponding character \(\zeta_T\) of \cite[Definition 4.6.5]{KalRSP} has the formula
\[ t \mapsto \prod_{\substack{\alpha \in R(T,G)_\tx{sym,ram}/\Gamma\\ \alpha(t) \in -1 + \mf p_\alpha}} \epsilon_\alpha \]
for $t \in T(F)$.
\end{lem}
\begin{proof}
That $(\epsilon_\alpha)$ is a set of $\zeta$-data is immediate. The corresponding character of \cite[Definition 4.6.5]{KalRSP} sends $t$ to
\[ \prod_{\alpha \in R(T,G)_\tx{sym,ram}/\Gamma} \epsilon_\alpha^{e_\alpha\ord(\delta_\alpha)}, \]
where $\delta_\alpha \in F_\alpha^\times$ is any element with $\delta_\alpha/\sigma_\alpha(\delta_\alpha)=\alpha(t)$. Now $\alpha(t) \in F_\alpha^1 \subset O_{F_\alpha}^\times$ and its image in $k_\alpha^\times$ lies in $\{\pm 1\}$. If it is $+1 \in k_\alpha^\times$ we can choose $\delta_\alpha \in 1+\mf p_\alpha$ and the contribution of this $\alpha$ is $1$. If it is $-1 \in k_\alpha^\times$ we can choose $\delta_\alpha$ to be a uniformizer in $F_\alpha$ and the contribution of the corresponding factor to the character is $\epsilon_\alpha$.
\end{proof}

Given sets of $a$-data and $\chi$-data for $R(T,G/M)$ we define, following \cite[Definition 4.6.2]{KalRSP},
\[ \Delta_{II}^{G/M}[a,\chi] : T(F) \to \C^\times,\qquad \gamma \mapsto \prod_{\substack{\alpha \in R(T,G/M)_\tx{sym}/\Gamma \\ \alpha(\gamma) \neq 1}} \chi_\alpha\left(\frac{\alpha(\gamma)-1}{a_\alpha}\right). \]
Of course we have $\Delta_{II}^{G/M}(\gamma)=\Delta_{II}^{G,\tx{abs}}(\gamma)/\Delta_{II}^{M,\tx{abs}}(\gamma)$, where we arbitrarily extend the given $a$-data and $\chi$-data from $R(T,G/M)$ to $R(T,G)$. Thus the term $\Delta_{II}^{G/M}[a,\chi]$ is introduced purely for notational convenience.

\begin{notn}
\label{notn:onestep-a}
Remember that \(\Tsc\) stands for the pre-image of \(T\) in \(\Gsc\).  Via the surjective homomorphism $\Tsc \to \Mscab$ we map $X$ to $\Lie^*(\Tsc)(F)$ and for $\alpha \in R(T,G/M)$ define
\begin{equation}
a_\alpha = \pair X{H_\alpha}.
\end{equation}
It is immediate that this is a set of $a$-data for $R(T,G/M)$.
\end{notn}

Let $\chi_\alpha' : F_\alpha^\times \to \C^\times$ be
	\begin{itemize}
	\item trivial if $\alpha$ is asymmetric,
	\item the unique unramified quadratic character if $\alpha$ is unramified symmetric,
and	\item the unique extension of the inflation to $O_{F_\alpha}^\times$ of the quadratic character of $k_\alpha^\times$ that satisfies
\begin{equation}
\chi_\alpha'(2a_\alpha) = \lambda_{F_\alpha/F_{\pm\alpha}}(\Lambda \circ \tr_{F_{\pm\alpha}/F}),
\end{equation}
where $\lambda_{F_\alpha/F_{\pm\alpha}}$ is the Langlands $\lambda$-constant introduced in \cite{LanArt}, if \(\alpha\) is ramified symmeetric.
	\end{itemize}
These definitions of $a$-data and $\chi$-data are as in \cite[\S4.7]{KalRSP}. Recall that the Langlands $\lambda$-constant is closely related to a quadratic Gauss sum. More precisely, according to \cite[Lemma 1.5]{BH05b} we have
\[ \lambda_{F_\alpha/F_{\pm\alpha}}(\Lambda \circ \tr_{F_{\pm\alpha}/F}) = \kappa_\alpha(e_{\pm\alpha})\mf{G}_{k_\alpha}(\Lambda^0 \circ \tr_{k_\alpha/k}), \]
where
the character $\kappa_\alpha$ of $F_{\pm\alpha}^\times$ is the one corresponding to the quadratic extension $F_\alpha/F_{\pm\alpha}$ and \(\mf G_{k_\alpha}(\Lambda^0 \circ \tr_{k_\alpha/k})\) is the Gauss sum
\[
q_\alpha^{-1/2}\sum_{x \in k_\alpha^\times}\sgn_{k_\alpha}(x)\Lambda^0(\tr_{k_\alpha/k}(x)) = q_\alpha^{-1/2}\sum_{x \in k_\alpha} \Lambda^0(\tr_{k_\alpha/k}(x^2)).
\]

Let $\alpha_0$ be the restriction of $\alpha$ to $Z_M$. We define $\chi''_{\alpha_0}$ analogously to \(\chi'_\alpha\); namely, it is
	\begin{itemize}
	\item trivial if $\alpha_0$ is asymmetric,
	\item the unique unramified quadratic character if $\alpha_0$ is unramified symmetric,
and	\item the
unique extension of the inflation to $O_{F_{\alpha_0}}^\times$ of the quadratic character of $k_{\alpha_0}^\times$ that satisfies
\begin{equation}
\chi''_{\alpha_0}(\ell(\alpha^\vee)a_\alpha) = (-1)^{f_{\alpha_0}+1}\mf{G}_{k_{\alpha_0}}(\Lambda^0 \circ \tr_{k_{\alpha_0}/k}),
\end{equation}
where \(\ell(\alpha^\vee) \in \{1,2,3\}\) denotes the integer\-normalized square length of \(\alpha^\vee\) (so that \(\ell(\anondot)\) is identically 1 on every simply laced component of \(R^\vee(T, G)\),
and \(\ell(\alpha^\vee) = 1\) for all short coroots \(\alpha^\vee\)
in every non-simply laced component of \(R(T, G)\), cf.\ Definition \ref{dfn:regular-length}), if $\alpha_0$ is ramified symmetric.
	\end{itemize}

We claim that $\chi''_{\alpha_0}$ depends only on $\alpha_0$, and not on $\alpha$. As we prove in Lemma \ref{lem:pairing}, for every $\alpha \in R(T,G/M)$ the element $\ell(\alpha)a_\alpha$ depends only on $\alpha_0$. If $\ell \in \{1,2,3\}$ is the multiplicity of the strongest bond in the Dynkin diagram of the irreducible component containing $\alpha$, then $\ell(\alpha) \dotm \ell(\alpha^\vee) = \ell$. Since each fiber of the restriction map $R(T,G/M) \to R(Z_M,G)$ lies in a single irreducible component of $R(T,G)$, the element $\alpha_0$ determines $\ell$. Therefore $\ell(\alpha^\vee)^{-1}a_\alpha$ depends only on $\alpha_0$.
But the restriction of $\chi''_{\alpha_0}$ to $O_{F_\alpha}^\times$ is quadratic, so $\chi''_{\alpha_0}(\ell(\alpha^\vee)^2)=1$.

\begin{notn}
\label{notn:onestep-chi}
In this way we obtain a set of $\chi$-data $\chi''_0$ for $R(Z_M,G)$.
As in \cite[Remark 6.4]{KalDC} we obtain from $\chi''_0$ a set of $\chi$-data for $R(T,G/M)$, which we shall denote by $\chi''$. Explicitly, $\chi''_\alpha = \chi''_{\alpha_0} \circ N_{F_\alpha/F_{\alpha_0}}$.
\end{notn}

\begin{lem} \label{lem:ru}
Let $x \in \cB(T,F)$. For every $\gamma \in T(F)$ we have
\[ \Delta_{II}^{G/M}[a,\chi''](\gamma) = \Delta_{II}^{G/M}[a,\chi'](\gamma) \epsilon^{G/M}_{f}(\gamma) \epsilon^{G/M}_{\sharp, x}(\gamma)\epsilon^{G/M}_x(\gamma). \]
\end{lem}

We have used the superscripts $G/M$ in the statement of the lemma for emphasis, but will otherwise continue to omit them in this subsection, including in the proof.

\begin{proof}
We apply Theorem \ref{thm:main} to the maximal torus $T$ and see that the identity we are proving becomes
\begin{equation} \label{eq:eps_sharp}
\Delta_{II}[a,\chi''](\gamma) = \Delta_{II}[a,\chi'](\gamma)\dotm \epsilon_\flat(\gamma).
\end{equation}

We now consider the difference between $\chi'_\alpha$ and $\chi''_\alpha$. Both of these are tamely ramified $\chi$-data. By construction $\chi'$ is minimally ramified, in the sense of \cite[Definition 4.6.1]{KalRSP}. Although $\chi''_0$ was a set of minimally ramified $\chi$-data for $R(Z_M,G)$, its inflation $\chi''$ need not be minimally ramified, as discussed in \cite[\S5.4]{KalDC}. In the notation of loc.\ cit., we have $\chi''=\inf \chi''_0$, and we have the associated minimally ramified $\chi$-data $\min \chi''$. Our comparison of $\chi'_\alpha$ and $\chi''_\alpha$ will proceed through the intermediary $(\min \chi'')_\alpha$. We have $\chi''_\alpha=(\min \chi'')_\alpha$ when $\alpha$ is ramified symmetric, and $(\min \chi'')_\alpha=\chi'_\alpha$ when $\alpha$ is not ramified symmetric.

The comparison between $\chi''$ and $\min \chi''$ is handled by \cite[Proposition 5.27]{KalDC} and \cite[Lemma 4.6.6]{KalRSP}, which imply
\begin{equation} \label{eq:d21}
\Delta_{II}[a,\chi''](\gamma)=\Delta_{II}[a,\min \chi''](\gamma) \dotm \epsilon_{\flat,0}(\gamma).
\end{equation}
Turning to the comparison between $\min \chi''$ and $\chi'$, which only differ at ramified symmetric roots, let $\alpha$ be such a root. Then both $\chi_\alpha'$ and $\chi''_\alpha$ are tamely ramified characters of $F_\alpha^\times$ that restrict to the character $\kappa_\alpha : F_{\pm\alpha}^\times/N_{F_\alpha/F_{\pm\alpha}}(F_\alpha^\times) \to \{\pm 1\}$ of $F_{\pm\alpha}^\times$. Recall that the restriction of $\kappa_\alpha$ to $O_{\pm\alpha}^\times$ factors through $k_{\pm\alpha}^\times=k_{\alpha}^\times$ and is the quadratic character of this cyclic group of even order.
Therefore, any tamely ramified character of $F_\alpha^\times$ that extends $\kappa_\alpha : F_{\pm\alpha}^\times \to \{\pm 1\}$ also extends the inflation to $O_\alpha^\times$ of the sign charcter of $k_\alpha^\times$.

We see that $\chi'_\alpha$ and $(\min \chi'')_\alpha=\chi''_\alpha$ are either equal, or differ by the unramified quadratic character of $F_\alpha^\times$. To compute their quotient it is enough to evaluate them at an element of $F_\alpha^\times$ of odd valuation. We have that $\ell(\alpha)a_\alpha$ is such an element,
because its trace in $F_{\pm\alpha}$ is zero. By definition we have $\chi'_\alpha(\ell(\alpha^\vee)a_\alpha)=\kappa_\alpha(2\ell(\alpha^\vee))\lambda_{F_\alpha/F_{\pm\alpha}}(\Lambda \circ \tr_{F_{\pm\alpha}/F})$. By \cite[Lemma 1.5]{BH05b} the constant $\lambda_{F_{\alpha}/F_{\pm\alpha}}(\Lambda \circ \tr_{F_{\pm\alpha}/F})$ is equal to $\kappa_\alpha(e_{\pm\alpha})\mf{G}_{k_\alpha}(\Lambda^0 \circ \tr_{k_\alpha/k})$ and we arrive at
\[ \chi'_\alpha(\ell(\alpha^\vee)a_\alpha)=\kappa_\alpha(e_\alpha \dotm \ell(\alpha^\vee))\mf{G}_{k_\alpha}(\Lambda^0 \circ \tr_{k_\alpha/k}). \]

On the other hand, using that $\ell(\alpha^\vee)a_\alpha$ depends only on $\alpha_0$ and thus lies in $F_{\alpha_0}$, we compute
\begin{align*}
\chi''_\alpha(\ell(\alpha^\vee)a_\alpha)&{}=\chi''_{\alpha_0}(N_{F_\alpha/F_{\alpha_0}}(\ell(\alpha^\vee)a_\alpha))\\
&{}=\chi''_{\alpha_0}(\ell(\alpha^\vee)a_\alpha)^{[F_\alpha:F_{\alpha_0}]}\\
&{}=\bigl((-1)^{f_{\alpha_0}+1}\mf{G}_{k_{\alpha_0}}(\Lambda^0 \circ \tr_{k_{\alpha_0}/k})\bigr)^{e(\alpha/\alpha_0)f(\alpha/\alpha_0)}.
\end{align*}
By Lemma \ref{lem:r1} the integer $e(\alpha/\alpha_0)$ is odd. Therefore the relation $\kappa_{\alpha_0}(-1)=\mf{G}_{k_{\alpha_0}}(\Lambda^0 \circ \tr_{k_{\alpha_0}/k})^2$ implies
\begin{align*}
\chi''_\alpha(\ell(\alpha^\vee)a_\alpha)&{}=\bigl((-1)^{f_{\alpha_0}+1}\mf{G}_{k_{\alpha_0}}(\Lambda^0 \circ \tr_{k_{\alpha_0}/k}) \dotm \kappa_{\alpha_0}(-1)^{(e(\alpha/\alpha_0)-1)/2}\bigr)^{f(\alpha/\alpha_0)}\\
&{}=(-1)^{f_{\alpha}}\dotm(-\mf{G}_{\kappa_{\alpha_0}}(\Lambda^0 \circ \tr_{k_{\alpha_0}/k}))^{f(\alpha/\alpha_0)} \dotm \kappa_{\alpha}(-1)^{(e(\alpha/\alpha_0)-1)/2}\\
&{}=(-1)^{f_{\alpha}+1}\dotm\mf{G}_{\kappa_\alpha}(\Lambda^0 \circ \tr_{k_\alpha/k}) \dotm \kappa_{\alpha}(-1)^{(e(\alpha/\alpha_0)-1)/2},
\end{align*}
where the last equality is due to the Hasse--Davenport relation. We conclude that the ratio $\chi''_\alpha(\ell(\alpha^\vee)a_\alpha)/\chi'_\alpha(\ell(\alpha^\vee)a_\alpha)$ equals
\[ (-1)^{f_\alpha+1}\kappa_\alpha(e_{\alpha}\ell(\alpha^\vee))\kappa_\alpha(-1)^\frac{e(\alpha/\alpha_0)-1}{2}. \]
With \cite[Lemma 4.6.6]{KalRSP} and Lemma \ref {lem:ers1} we conclude
\begin{equation} \label{eq:d22}
\Delta_{II}[a,\min \chi''](\gamma)=\Delta_{II}[a,\chi'](\gamma) \dotm \epsilon_{\flat,1}(\gamma) \dotm \epsilon_{\flat,2}(\gamma).
\end{equation}
Combining \eqref{eq:eps_sharp}, \eqref{eq:d21}, and \eqref{eq:d22}, we obtain the result.
\end{proof}

\subsection{Character formula for regular supercuspidal representations} \label{sub:char}

In this section we give a formula for the Harish-Chandra character of a supercuspidal representation $\pi_{(S,\theta)}$ obtained from a tame elliptic pair that is regular (cf. \cite[Definition 3.7.5]{KalRSP}), or more generally $k_F$-non-singular (cf. \cite[Definition 3.4.1]{KalSLP}). For shallow elements this formula is based on \cite[\S4.4]{KalRSP}, which in turn is based on \cite[Theorem 4.28]{DS18}. For general elements, this formula is based on the forthcoming work \cite{Spice21}, which in turn rests on \cite{Spice17}, especially \cite[Theorem 5.3.11]{Spice17}. Using the twisted Yu construction not only makes it possible to obtain a formula for the Harish-Chandra character, but also yields a cleaner formula than those that have previously been available, by removing the auxiliary sign characters that were present in the formula obtained earlier in special cases from the original Yu construction.

Consider a tame $k_F$-non-singular elliptic pair $(S,\theta)$ and let $\pi_{(S,\theta)}$ be the associated supercuspidal representation. We emphasize that we are now using the twisted Yu construction as introduced in Theorem \ref{thm:twistedYu}. In this setting, this amounts to applying the original Yu construction to the pair $(S,\theta\dotm\epsilon)$, where $\epsilon$ is the product of $\epsilon^{G^i/G^{i-1}}_x$ over $i=0,\dotsc,d$, where $G^0 \subset \dotsb \subset G^d$ is the twisted Levi tower associated to $(S,\theta)$ and $x$ being an arbitrary point in $\cB(S,F)$.

We will state two forms of the character formula for $\pi_{(S,\theta)}$ --- one (Proposition \ref{pro:char}) for regular elements of $S(F)$ that are topologically semi-simple modulo $Z_G$, and one (Theorem \ref{thm:char}) for arbitrary regular semi-simple elements of $G(F)$. The first form is a consequence of the second, but we state it separately because it has a simpler form, requires fewer assumptions, and is moreover independent of \cite{Spice21}.

\begin{rem}
\label{rem:mod-Z-or-Z0}
Recall from \cite[Definition 2.23]{Spice08} that a topological Jordan decomposition modulo a central subgroup \(Z\) (there called a topological \(F\)-Jordan decomposition) of an element \(\gamma\) is a commuting decomposition \(\gamma = \gamma_0\dotm\gamma_{0+}\), where \(\gamma_0\) is topologically semisimple modulo \(Z\) and \(\gamma_{0+}\) is topologically unipotent modulo \(Z\) (there called absolutely \(F\)-semisimple and topologically \(F\)-unipotent modulo \(Z\)), in the sense of \cite[Definition 2.15]{Spice08}.  To be explicit, \(\gamma_0\) is semisimple and, for every maximal torus \(T \subset C_G(\gamma_0)\) and every character \(\chi \in X^*(T/Z)\), we have that \(\chi(\gamma_0)\) has finite order prime to \(p\); and the image \(\bar\gamma_{0+}\) of \(\gamma_{0+}\) in \(G/Z\) satisfies \(\lim_{n \to \infty} \bar\gamma_{0+}^{p^n} = 1\) \cite[Lemma 2.21]{Spice08}.  By \cite[Proposition 2.24]{Spice08}, the ingredients in a topological Jordan decomposition modulo \(Z\) are uniquely determined modulo \(Z(F)\), because a topological Jordan decomposition modulo \(Z\) projects to a topological Jordan decomposition (modulo the trivial subgroup) in \((G/Z)(F)\).

We can consider topological Jordan decompositions modulo \(Z = Z_G\) and \(Z = Z_G^\circ\), and the answers may differ, but not by much.  Namely, a topological Jordan decomposition modulo \(Z_G^\circ\) is also one modulo \(Z_G\); conversely, if \(\gamma = \gamma_0\dotm\gamma_{0+}\) is a topological Jordan decomposition modulo \(Z_G\), then there is some element \(z\) in the maximal bounded subgroup \(Z_G(F)\subb\) of \(Z_G(F)\) such that \(\gamma = (\gamma_0\dotm z)(z^{-1}\dotm \gamma_{0+})\) is a topological Jordan decomposition modulo \(Z_G^\circ\).
\end{rem}

We use the notations $\Lambda$ and $\Delta_{II}^\tx{abs}[a,\chi'']$ of \S\ref{sub:root}. Let $T_G$ be the minimal Levi subgroup of the quasi-split inner form of $G$ (well-defined up to conjugation, in particular up to isomorphism). Let $e(G)$ be the Kottwitz sign of $G$, cf. \cite{Kot83}.

\begin{pro} \label{pro:char}
Let $\gamma \in S(F)$ be topologically semisimple modulo \(Z_G\).
The value of the normalized Harish-Chandra character of $\pi_{(S,\theta)}$ at $\gamma$ is given by
\[ e(G)\epsilon_L(X^*(T_G)_\C-X^*(S)_\C,\Lambda)\sum_{w \in N(S,G)(F)/S(F)}\Delta_{II}^\tx{abs}[a,\chi''](\^w\gamma)\theta(\^w\gamma). \]
\end{pro}

Upon comparing this formula to that of \cite[Corollary 4.10.1]{KalRSP}, we see that the auxiliary characters $\epsilon_{f,\tx{ram}}$ and $\epsilon^\tx{ram}$ of loc.\ cit.\ have disappeared. As discussed in \S\ref{sub:root}, the $\chi$-data used here and denoted by $\chi''$ are more natural than the $\chi$-data $\chi'$ used in loc.\ cit., and allow for inductive arguments along the chain of groups $G^0 \subset G^1 \subset \dotsb \subset G^d$.

\begin{proof}
This follows from Lemma \ref{lem:ru} and \cite[Corollary 4.10.1]{KalRSP}, which gives a formula for $\pi_{(S,\theta)}$ obtained via the original Yu construction, so in this setting must be applied to $(S,\theta\dotm\epsilon)$.  The latter result is stated in terms of topological Jordan decompositions modulo \(Z_G^\circ\), but, by Remark \ref{rem:mod-Z-or-Z0}, this does not affect the validity of the resulting character formula.
\end{proof}

The statement of Theorem \ref{thm:char} will require a number of assumptions. 
We begin with the assumption that \(p\) is not a bad prime for \(G\), does not divide the order of the fundamental group of \(G\dergp\), and does not divide the order of $Z(G\dergp)$. These three assumptions are needed for \cite[Proposition 3.6.7]{KalRSP} (the final condition is not formally stated in loc. cit., but it is needed, as was pointed out by Charlotte Chan and Masao Oi; for a careful discussion of this we refer to the list of errata on the webpage of T.K.). We note that all these assumptions are implied by the assumption that $p$ is prime to the order of the absolute Weyl group of $G$.

In addition to the assumptions on \(p\) needed to apply the results of \cite[\S3.6]{KalRSP},
we assume that the characteristic of $F$ is zero and that the exponential map for $G$ converges on \(\fg(F)_{0+}\). 
This implies, for every tame, twisted Levi subgroup \(M\) of \(G\), that the exponential map for \(M\) converges on \(M(F)_{0+}\);
and, by \cite[Lemma 3.3.2]{KalRSP}, that the exponential map for the abelianization \(M\abgp = M/M\dergp\) also converges on \(M\abgp(F)_{0+}\).
By \cite[Lemma B.0.3]{DR09}, we have the necessary convergence whenever \(p \geq (2+e)n\), where $e$ is the ramification degree of $F/\Q_p$ and $n$ is the dimension of the smallest faithful algebraic representation of $G$.

One of the consequences of the convergence of the exponential map is that there exists a (non-unique) element $X \in \Lie^*(S)(F)$ that satisfies $\theta(\exp(Y))=\Lambda(\<X,Y\>)$ for all $Y \in \Lie(S)(F)_{0+}$. We will use such an element in Lemma \ref{lem:char1} below.

Recall the sign character $\tilde\epsilon^{G/G^0}$ introduced in \cite[\S4.3]{DS18} (see also \cite[(4.3.4)]{KalRSP}, where it was not decorated with a superscript, and was relative to $G/G^{d-1}$).

\begin{lem}
\label{lem:char1}
Let $\gamma \in G(F)$ be a regular semisimple element and let $\gamma=\gamma_0 \cdot \gamma_{0+}$ be a topological Jordan decomposition modulo $Z_G$.

Suppose that \(X\) belongs to \(\Lie^*(S)(F)\) and satisfies
\(\theta(\exp(Y)) = \Lambda(\<X, Y\>)\) for all \(Y \in \Lie(S)(F)_{0+}\).

The normalized character of the non-singular supercuspidal representation \(\pi_{(S, \theta)}\) of \(G\) at \(\gamma \in G\) is given by
\[
(-1)^{\operatorname{rk}_F(G^0)}\!\!\!\!\!\!\!\!\!\!\sum_{\substack{
	g \in S(F)\lmod G(F)/J(F) \\
	\^g{\gamma_0} \in S(F)
}}
	\!\!\!\!\!\!\!\!\!\!(-1)^{\operatorname{rk}_F({J^0_g})}
	\tilde e^{G/G^0}(\theta, \^g{\gamma_0})
	\epsilon^{G/G^0, \sharp}(\^g{\gamma_0})\dotm
	\theta(\^g{\gamma_0})
	\widehat O_{X^g}^J(\log \gamma_{0+}),
\]
where
	\begin{itemize}
	\item $J=C_G(\gamma_0)^\circ$, $J^0_g=C_{G^0}({^g}\gamma_0)^\circ$,
	\item $X^{g}=\Ad^*(g)^{-1}X$, i.e., \(\pair{X^g}Y = \pair X{\^g Y}\) for all \(Y \in \fg\),
	which we view as an element of \(\Lie^*(J)(F)\) using Remark \ref{rem:dual-Lie-subspace},
and	\item $\epsilon^{G/G^0,\sharp}(\^g{\gamma_0}) = \epsilon^{G/G^0}_{\sharp, x}(\^g{\gamma_0}) \dotm \epsilon^{G/G^0}_x(\^g{\gamma_0})$.
	\end{itemize}
\end{lem}

Note that, in this subsection, we use \(J\) for the connected centralizers \(C_G(\gamma_0)^\circ\), not the compact-modulo-center subgroup of \S\ref{sub:Yu}.

\begin{proof}[Proof of Lemma \ref{lem:char1}]
By \cite[Proposition 3.6.7]{KalRSP}, the pair \((S, \theta)\) has a Howe factorization \(((S, G^0, \dotsc, G^d), (\phi_{-1}, \phi_0, \dotsc, \phi_d))\), in the sense of \cite[Definition 3.6.2]{KalRSP}.
Note that condition (2) of that definition requires only that each factor \(\phi_i\) be trivial on \(G^i_\tx{sc}(F)\); but the proof of \cite[Proposition 3.6.7]{KalRSP} (specifically \cite[Lemma 3.6.9]{KalRSP}) actually produces characters trivial on \(G^i_\tx{der}(F)\).  We require that our factors satisfy this stronger condition, and will write again \(\phi_i\) for any character of \(G^i\abgp(F)\) from which \(\phi_i\) may be pulled back.

For each \(0 \le i \le d\), put \(D^i = G^i\abgp\).  The duality afforded by \(\Lambda\) identifies \(\phi_i \circ \exp\rvert_{\Lie(D^i)(F)_{0+}} \in \Hom_\tx{cts}(\Lie(D^i)(F)_{0+}, \mathbb C^\times)\) with an element of the quotient \(\Lie^*(D^i)(F)/\Lie^*(D^i)(F)_0\).  Let \(X_i \in \Lie^*(D^i)(F)\) be any representative of this element.
If \(i\) is less than \(d\) and we write \(r_i\) for the depth of \(\phi_i\), then \cite[Lemma 3.6.8]{KalRSP} shows that there is an element in the Moy--Prasad coset (i.e., ``dual blob'') \(X_i + \Lie^*(D^i)(F)_{-r_i+}\) whose pull back to $\Lie^*(G^i)(F)$ is \(G^{i+1}\)-generic of depth \(r_i\), in the sense of \cite[\S8]{Yu01}.  However, it is easy to see that, if the pull back of \emph{some} element of \(X_i + \Lie^*(D^i)(F)_{-r_i+}\) satisfies \cite[p.~596, \textbf{GE1}]{Yu01}, then the same holds true for \emph{all} elements; in particular, the pull back of \(X_i\) itself does.  Then \cite[Lemma 8.1]{Yu01} shows that the pull back of \(X_i\) is \(G\)-generic of depth \(r_i\).  

We now have that \(\Gamma \ldef X_0 + \dotsb + X_d\), viewed as an element of \(\Lie^*(S)(F)\), satisfies the same condition as $X$, due to \(\prod_{i = -1}^d \phi_i = \theta\). Therefore \(X - \Gamma\) lies in \(\Lie^*(S)(F)_0\).

We will now apply \cite[Theorem 10.2.1]{Spice21} and \cite[Lemmas 9.3.1, 10.0.4, 12.4.3]{DR09}. While in \cite{DR09} there is a blanket assumption that $G$ is a pure inner form of an unramified group, this assumption is not necessary for the cited results.

The depth \(r_{-1}\) is \(0\), so the number \(R_{-1} = \max \{r_{-1}, {0+}\}\) appearing in \cite[Theorem 10.2.1]{Spice21} equals \(0+\). The representation $\pi_0$ of loc. cit. is the representation associated to the pair $(S,\theta)$ but for the group $G^0$ rather than $G$. In other words, it is the product of the character $\prod_{i = 0}^d \phi_i$ of $G^0(F)$ with the depth-zero supercuspidal representation of $G^0(F)$ obtained by truncating the Yu-datum of $\pi$.

For every \(g \in G(F)\) such that \(\^g{\gamma_0}\) belongs to \(S(F)\), and every regular semisimple \(\delta_{0+} \in J^0_g(F)_{0+}\), the results of \cite{DR09} cited above imply that the normalized character of \(\pi_0\) at \(\^g{\gamma_0}\dotm\delta_{0+}\) is given by
\[
\sum_{\substack{
    g^0 \in S(F)\lmod G^0(F)/{^g}J^0(F) \\
    \^{g^0 g}{\gamma_0} \in S(F)
}}
    (-1)^{\operatorname{rk}_F(G^0)-\operatorname{rk}_F(J^0_g)}
    \theta(\^{g^0g}\gamma_0)
    \widehat O_{X^{g^0 g}}^{J^0_g}(\log \delta_{0+}).
\]
Since \(X^{g^0 g} - X_{-1}^{g^0 g}\) is fixed by the coadjoint action of \(J^0_g\), we note that this is the same as if we replaced $\theta$ and $X$ by $\theta_{-1}=\phi_{-1}$ and $X_{-1}$, but multiplied the result with $\prod_{i = 0}^d \phi_i(\^{g^0 g}{\gamma_0} \cdot \delta_{0+})$, and this is the form in which \cite{DR09} formulate the result.

We combine this formula with \cite[Theorem 10.2.1]{Spice21}, where the character formula for \(\pi_0\) is allowed to be a linear combination of various Fourier transforms of orbital integrals of elements `near' \(\Gamma^{g^0 g}\); we are specifying that only \(X^{g^0 g} \in \Gamma^{g^0 g} + \Lie^*(S)(F)_0\) is needed.

As indicated in \cite[Notation 8.2.5]{Spice21},
the quantity \(\widetilde{\mathfrak G}_G(X, \gamma_0)\) and its relatives are defined only `relatively', in the sense that we can specify \(\widetilde{\mathfrak G}_{G^0/J^0_g}(X, \^{g^0 g}{\gamma_0})\) arbitrarily, and it determines \(\widetilde{\mathfrak G}_{G/J}(X, \^g{\gamma_0})\).  By \cite[Proposition 8.2.4]{Spice21},
if we put \(\widetilde{\mathfrak G}_{G^0/J^0_g}(X, \^g{\gamma_0}) = 1\), then \(\widetilde{\mathfrak G}_{G/J}(X, \^g{\gamma_0})\) equals \(\tilde e^{G/G^0}(\theta, \^g{\gamma_0})\epsilon^{G/G^0, \sharp}(\^g{\gamma_0})\).

\end{proof}

We now combine Lemmas \ref{lem:char1} and \ref{lem:ru} to obtain the desired character formula. For $0 \leq i <d$, let $\tilde G^i$ be the preimage of $G^i$ in the simply connected cover of $G^{i+1}$, and denote the image of $X_i$ (as defined in the proof of Lemma \ref{lem:char1}) under the natural map $\Lie^*(G^i\abgp)(F) \to \Lie^*(\tilde G^i\abgp)(F)$ again by \(X_i\).

\begin{notn}
\label{notn:RSG-a+chi}
We put on $R(S,G^{i+1}/G^i)$ the $a$-data used in Notation \ref{notn:onestep-a}, namely $a_\alpha=\pair{X_i}{H_\alpha}$. They are the same as the $a$-data defined in \cite[(4.10.1)]{KalRSP}. Putting these together, and choosing arbitrarily $a$-data for $R(S,G^0)$ consisting of units, we obtain $a$-data for $R(S,G)$. We obtain $\chi$-data $\chi''$ for $R(S,G^{i+1}/G^i)$ as in Notation \ref{notn:onestep-chi}: $\chi''_\alpha = \chi''_{\alpha_i} \circ N_{F_\alpha/F_{\alpha_i}}$, where $\alpha_i$ is the restriction of $\alpha \in R(S,G^{i+1}/G^i)$ to $Z_{G^i}^\circ$.
Putting these together, and using the canonical unramified $\chi$-data for $R(S,G^0)$, we obtain $\chi$-data $\chi''$ for $R(S,G)$.
\end{notn}

\begin{thm} \label{thm:char}
Let $\gamma \in G(F)$ be a tame, regular semi-simple element with topological Jordan decomposition $\gamma=\gamma_0 \dotm \gamma_{0+}$ modulo \(Z_G\). Let \(X \in \Lie^*(S)(F)\)  satisfy \(\theta(\exp(Y)) = \Lambda(\<X, Y\>)\) for all \(Y \in \Lie(S)(F)_{0+}\).

The value of the normalized Harish-Chandra character of $\pi_{(S,\theta)}$ evaluated at $\gamma$ is given by
\begin{equation}
\label{eq:unwound-1}
\begin{aligned}
&e(G)e(J)\epsilon_L(X^*(T_{G})_\C-X^*(T_{J})_\C,\Lambda)\cdot\\
&\sum_{\substack
	{g \in S(F) \lmod G(F)/J(F) \\
	\^g{\gamma_0} \in S(F)}
}
\Delta_{II}^\tx{abs}[a,\chi''](\^g{\gamma_0}) \dotm\theta(\^g{\gamma_0})\dotm
\widehat O_{X^{g}}^{J}(\log \gamma_{0+}),
\end{aligned}
\end{equation}
where $J=C_G(\gamma_0)^\circ$ and $T_G$ and $T_J$ are the minimal Levi subgroups of the quasi-split inner forms of $G$ and $J$, respectively, $e(\anondot)$ is the Kottwitz sign \cite{Kot83}, and $\epsilon_L$ is the root number of the given virtual Galois representation.
\end{thm}

\begin{proof}
We apply Lemma \ref{lem:char1} and need to show that for an element $\gamma_0 \in S(F)$ that is topologically semisimple modulo \(Z_G\), the term
\begin{equation} \label{eq:s}
(-1)^{\rk_F(G^0)} (-1)^{\rk_F(J^0)} \tilde e^{G/G^0}(\theta,\gamma_0)\dotm \epsilon^{G/G^0,\sharp}(\gamma_{0})
\end{equation}
equals
\begin{equation} \label{eq:f}
e(G)e(J)\epsilon_L(X^*(T_{G})_\C-X^*(T_{J})_\C,\Lambda)\Delta_{II}^\tx{abs}[a,\chi''](\gamma_0),
\end{equation}
where $J^0 = G^0 \cap J=C_{G^0}(\gamma_0)^\circ$. Recall that \cite[Corollary 4.7.6]{KalRSP} gives an expression for the product $\tilde e(\theta,\gamma_0)\epsilon_\tx{sym,ram}(\gamma_0)$, see \cite[(4.3.2)]{KalRSP} for the definition of $\epsilon_\tx{sym,ram}$. Note that in our current notation, the terms $\epsilon_\tx{sym,ram}(\gamma_0)$ and $\tilde e(\theta,\gamma)$ occurring in loc. cit. should carry a superscript $G^d/G^{d-1}$. We claim that $\epsilon_\tx{sym,ram}(\gamma_0)=1$. Indeed, the definition of $\epsilon_\tx{sym,ram}(\gamma_0)$, cf. \cite[(4.3.2)]{KalRSP}, is as a product over ramified symmetric roots $\alpha$ with the property that $\alpha(\gamma_0) \neq 1$ and $r_{d-1}-\ord(\alpha(\gamma_0)-1) \in 2\ord_x(\alpha)$. By assumption, the root values of $\gamma_0$ have finite order prime to \(p\) (see Remark \ref{rem:mod-Z-or-Z0}), so this condition becomes $r_{d-1} \in 2\ord_x(\alpha)$. \cite[Proposition 4.5.1]{KalRSP} shows that $\ord_x(\alpha)=e_\alpha^{-1}\Z$. On the other hand, $\pair{X_{d-1}}{H_\alpha}$ is an element of $F_\alpha$ whose trace in $F_{\pm\alpha}$ vanishes, so $r_{d-1}=-\ord \pair{X_{d-1}}{H_\alpha} \in \ord(F_\alpha^\times) \sm \ord(F_{\pm\alpha}^\times)=e_\alpha^{-1}(2\Z+1)$. Therefore, the product defining $\epsilon_\tx{sym,ram}(\gamma_0)$ is empty.

Therefore \cite[Corollary 4.7.6]{KalRSP} gives an expression for $\tilde e^{G^d/G^{d-1}}(\theta,\gamma_0)$. Applying this Corollary inductively for the chain $G^0 \subset \dots \subset G^d$ we see that $\tilde e^{G/G^0}(\theta,\gamma_0)$ equals

\begin{equation} \label{eq:s1}
	\epsilon_f^{G/G^0}(\gamma_0)
\frac
	{e(G)e(J)}
	{e(G^0)e(J^0)}
\frac
	{\epsilon_L(X^*(T_{G})_\C-X^*(T_{J})_\C,\Lambda)}
	{\epsilon_L(X^*(T_{G^{0}})_\C-X^*(T_{J^{0}})_\C,\Lambda)}
	\Delta_{II}^{G/G^0}[a,\chi'](\gamma_0)
\end{equation}

Consider the denominators in \eqref{eq:s1}. The term $\epsilon_L(X^*(T_{G^{0}})_\C-X^*(T_{J^{0}})_\C,\Lambda)$ was computed in \cite[Lemma 4.9.1]{KalRSP} as $(-1)^{\rk_F(T_{G^0})-\rk_F(T_{J^0})}$ (note that $S$ is maximally unramified both in $G^0$ and $J^0$). By \cite{Kot83} we have
\[ e(G^0)e(J^0)=(-1)^{\rk_F(G^0)-\rk_F(T_{G^0})+\rk_F(J^0)-\rk_F(T_{J^0})}. \]
We see that the combined contribution of the denominators in \eqref{eq:s1} is equal to $(-1)^{\rk_F(G^0)-\rk_F(J^0)}$, hence \eqref{eq:s} becomes
\[ e(G)e(J)\epsilon_L(X^*(T_{G})_\C-X^*(T_{J})_\C,\Lambda)\Delta_{II}^{G/G^0}[a,\chi'](\gamma_0)
\dotm\epsilon_f^{G/G^0}(\gamma_0)\dotm\epsilon^{\sharp,G/G^0}(\gamma_0). \]

Lemma \ref{lem:ru} converts this to
\[ e(G)e(J)\epsilon_L(X^*(T_{G})_\C-X^*(T_{J})_\C,\Lambda)\Delta_{II}^{G/G^0}[a,\chi''](\gamma_0). \]

Finally we claim that $\Delta_{II}^{\tx{abs},G^0}[a,\chi''](\gamma_0)$ is trivial, which implies
\[ \Delta_{II}^{\tx{abs},G}[a,\chi''](\gamma_0)=\Delta_{II}^{G/G^0}[a,\chi''](\gamma_0). \]
To prove the claim, recall $S$ is maximally unramified in $G^0$. This implies that no element of $R(S,G^0)$ is ramified symmetric. By construction, the character $\chi''_\alpha$ is trivial when $\alpha$ is asymmetric, and unramified quadratic when $\alpha$ is symmetric. We have chosen $a_\alpha$ to be a unit in $O_\alpha$. The element $\alpha(\gamma_0) \in F_\alpha^\times$ lies in $O_\alpha^\times$ and is topologically semi-simple, therefore $\alpha(\gamma_0)-1$ is either zero or a unit in $O_\alpha$. This proves the claim.
\end{proof}

\subsection{Endoscopy for non-singular supercuspidal $L$-packets} \label{sub:endo}

We retain the assumptions on $p$ from the previous subsection. In this subsection we will use the notation \(G'\) for an inner form of \(G\), and not for a twisted Levi subgroup as in \cite{Yu01}.

Let $\varphi : W_F \to \^LG$ be a supercuspidal parameter. Due to the assumptions on $p$, in particular that it does not divide the order of the Weyl group,
$\varphi$ is torally wild in the sense of \cite[Definition 4.1.2]{KalSLP}. Let $S_\varphi=\Cent(\varphi,\wh G)$. We recall that there is a pair $(S,\theta)$ associated to $\varphi$, consisting of a torus $S$ and a character $\theta : S(F) \to \C^\times$, satisfying certain properties that we will review below. There is a canonical exact sequence
\[ 1 \to \wh S^\Gamma \to S_\varphi \to \Omega(S,G)(F)_\theta \to 1, \]
where $\Omega(S,G) \subset \Aut(S)$ is the Weyl group relative to $G$ (see below), and \(\Omega(S, G)_\theta\) is the stabilizer of \(\theta\).
When $\varphi$ is regular, so that \(\Omega(S, G)_\theta\) is trivial, this sequence becomes an isomorphism $\wh S^\Gamma \to S_\varphi$.

Let $\Pi_\varphi$ be the $L$-packet associated to $\varphi$ as in \cite[\S4]{KalSLP} or \cite[\S5]{KalRSP}. In this section we prove the stability of $\Pi_\varphi$, as well as the endoscopic character identities for all $s \in \wh S^\Gamma \subset S_\varphi$. 

We begin by briefly recalling the construction of the packet. To $\varphi$ one associates a torally wild supercuspidal $L$-packet datum $(S,\wh\jmath,\chi,\theta)$ in the sense of \cite[Definition 4.1.4]{KalSLP}. We recall that $S$ is a tame $F$-torus whose dimension is equal to the absolute rank of $G$, $\wh\jmath : \wh S \to \wh G$ is an embedding of complex reductive groups whose $\wh G$-conjugacy class is $\Gamma$-stable, $\chi$ is a set of tamely ramified $\chi$-data for $R(S,G)$ and $\theta : S(F) \to \C^\times$ is a character. The embedding $\wh\jmath : \wh S \to \wh G$ specifies a $\Gamma$-stable $G(\sepfield)$-conjugacy class of embeddings $S \to G$, cf. \cite[\S5.1]{KalRSP}. This conjugacy class allows us to define the root system $R(S,G)$, equipped with $\Gamma$-action, as well as the Weyl group $\Omega(S,G)$. It also gives a stable class of embeddings of $S$ into any inner form of $G$; we call the embeddings in this stable class ``admissible''. We have a subsystem $R_{0+} \subset R(S,G)$, which can be defined either as $\{\alpha \in R(S,G) \stbar \theta(N_{E/F}(\alpha^\vee(E_{0+}^\times)))=1\}$, where $E/F$ is the splitting field of $S$, or as the dual to the root system of the centralizer of $\varphi(P_F)$ in $\wh G$.  We write $S^0 \subset S$ for the connected component of the intersection of the kernels of the members of $R_{0+}$. Given an admissible embedding $j : S \to G'$ into an inner form \(G'\) of \(G\), we obtain the twisted Levi subgroup $G'^0$ with root system $R_{0+}$, and $S^0$ is identified with $Z(G'^0)^\circ$. Let $R(S^0,G) \subset X^*(S^0)$ be the restrictions to \(S^0\) of $R(S,G) \sm R_{0+}$. We require the $\chi$-data on $R(S,G) \sm R_{0+}$ to be inflated from $\chi$-data for $R(S^0,G)$, and the $\chi$-data for $R_{0+}$ to be unramified. The pair $(S,\theta)$ is $F$-non-singular in the sense of \cite[Definition 3.1.1]{KalSLP}.

Let $(a_\alpha)$ and $(\chi'')$ be the $a$-data and $\chi$-data for $R(S,G)$ computed in terms of $\theta$ as in Notation \ref{notn:RSG-a+chi}.  The $a$-data are the same as given in \cite[(4.10.1)]{KalRSP}, but the $\chi$-data are not the same as given in \cite[(4.7.2)]{KalRSP}. The \(a\)- and \(\chi\)-data both depend only on \(\theta\res_{S(F)_{0+}}\), hence are invariants of the isomorphism class of the torally wild supercuspidal $L$-packet datum $(S,\wh\jmath,\chi,\theta)$. After changing $(S,\wh\jmath,\chi,\theta)$ within its isomorphism class we may, and do, assume that $\chi=\chi''$.

Given a rigid inner twist $(G',\xi,z)$ of $G$ and an admissible embedding $j : S \to G'$ we let $\pi_j$ be the supercuspidal representation $\pi_{(jS,j\theta)}$, computed in terms of the twisted Yu construction of \S\ref{sub:Yu}. When $\varphi$ is regular, $\pi_j$ is an irreducible regular supercuspidal representation and coincides with the one constructed in \cite[\S5.3]{KalRSP}. There we used the untwisted Yu construction, the $\chi$-data $\chi'$ of \cite[(4.7.2)]{KalRSP}, and the auxiliary characters $\epsilon_{f,\tx{ram}} \dotm \epsilon_{\sharp, x}$. That the resulting $\pi_j$ is the same follows from Lemma \ref{lem:ru}. The $L$-packet then consists of the representations $\pi_j$ for all rational classes of admissible embeddings $j$. When $\varphi$ is no long regular, $\pi_j$ may be reducible, where we are using the twisted Yu construction with a possibly reducible (but semi-simple) depth-zero supercuspidal representation of the smallest twisted Levi subgroup $G'^0$ by additivity. This is how $\pi_j$ was constructed in \cite[\S4.2]{KalSLP}. We refer to the set of irreducible constituents $[\pi_j]$ as a ``Deligne--Lusztig'' packet. The $L$-packet $\Pi_\varphi$ is the union of all Deligne--Lusztig packets, as $j$ runs over the rational classes of admissible embeddings $j$.

In Proposition \ref{pro:cf1} we use the notation $\gamma^g$, where \(g \in G(F)\), for the conjugation $g^{-1}\gamma g$ and $\gamma^j$, where \(j : S \to G\) is an embedding as above, for $j^{-1}(\gamma)$. Analogously, we write $\^gY$ for \(\Ad^*(g)Y\)
for any $Y \in \fg^*(F)$ and $^jX$ for the image of $X$ in $\Lie^*(j(S))(F)$ under the isomorphism $\Lie^*(j)^{-1}$, and then as an element of $\fg^*$ via Remark \ref{rem:dual-Lie-subspace}.

\begin{pro} \label{pro:cf1}
Let $(S,\wh\jmath,\chi,\theta)$, $(G',\xi,z)$, and $j$ be as just discussed. Let \(X \in \Lie^*(S)(F)\) satisfy \(\theta(\exp(Y)) = \Lambda(\<X, Y\>)\) for all \(Y \in \Lie(S)(F)_{0+}\). The normalized character of the supercuspidal representation $\pi_j$ evaluated at the strongly regular semisimple element $\gamma' = \gamma'_0 \dotm \gamma'_{0+} \in G'(F)$ has the formula
\[
e(G')e(J')\epsilon_L(T_G-T_J)\sum_{\substack{g \in J'(F) \lmod G'(F) / jS(F)\\ \gamma_{0}'^g \in jS(F)}}\Delta_{II}^\tx{abs}[a,\chi](\gamma_{0}'^{gj}) \theta(\gamma_{0}'^{gj}) \hO^{J'}_{\^{gj}X}(\log(\gamma_{0+}')),
\]
where $J'=C_{G'}(\gamma_0')^\circ$, $T_G$ and $T_J$ are the minimal Levi subgroups in the quasi-split inner forms of $G'$ and $J'$, and $\epsilon_L(T_G-T_J)$ is the root number of the virtual Galois representation $X^*(T_G)_\C - X^*(T_J)_\C$ with respect to the fixed additive character $\Lambda$.
\end{pro}
\begin{proof}
This is immediate from Theorem \ref{thm:char} and the construction of $\pi_j$.
\end{proof}

We note that the formula in Proposition \ref{pro:cf1} is independent of which representative of the isomorphism class of $(S,\wh\jmath,\chi,\theta)$ is being used.

In the following we use notation from \cite{KalRI}. In particular, $S_\varphi^+$ is the preimage of $S_\varphi \subset \wh G$ in a suitable finite cover of $\wh G$, or alternatively in the universal cover of $\wh G$, and $[\wh{\bar S}]^+$ is the preimage of $\wh S^\Gamma$.

Let $\dot s \in S_\varphi^+$ be a semi-simple element. Let $\mf{w}$ be a Whittaker datum for the quasi-split group $G$. The $\dot s$-stable character of $\varphi$ on a rigid inner form $(G',\xi,z)$ is defined as
\[ \Theta_{\varphi,\mf{w},z}^{\dot s} = e(G')\sum_{\rho \in \Irr(\pi_0(S_\varphi^+),[z])}\tr\rho(\dot s)\dotm\Theta_{\pi_\rho}. \]
This definition uses the bijection $\Irr(\pi_0(S_\varphi^+)) \to \Pi_\varphi$ constructed in \cite[\S\S4.4, 4.5]{KalSLP}. We briefly recall that this construction involves two steps. First, there is a natural simply transitive action of the character group $\pi_0([\wh{\bar S}]^+)^*$ on the set of Deligne--Lusztig packets in $\Pi_\varphi$. A choice of Whittaker datum for $G$ selects a base point in this set, because there is a unique rational class of admissible embeddings $j : S \to G$ for which the supercuspidal representation $\pi_j$ is $\mf{w}$-generic. This follows from \cite[Lemma 6.2.2]{KalRSP}, whose proof is valid in our setting once \cite[Lemma 6.2.1]{KalRSP} is replaced by Proposition \ref{pro:cf1}. Therefore the choice of $\mf{w}$ gives a bijection between $\pi_0([\wh{\bar S}]^+)^*$ and the set of Deligne--Lusztig packets in $\Pi_\varphi$. The second step is to create a bijection, for every $\eta \in \pi_0([\wh{\bar S}]^+)^*$, between $\Irr(\pi_0(S_\varphi^+),\eta)$ and the corresponding Deligne--Lusztig packet. This is done in \cite{KalSLP} up to the ambiguity of twisting by a character of $\pi_0(S_\varphi^+)$ inflated from $\Omega(S,G)(F)_\theta$. Resolving this ambiguity is currently work in progress. However, in this paper we will be concerned only with elements $\dot s$ that lie in the subgroup $[\wh{\bar S}]^+ \subset S_\varphi^+$.
For such elements, this ambiguity is immaterial, and $\Theta_{\varphi,\mf{w},z}^{\dot s}$ is unambiguously defined.

\begin{thm} \label{thm:schar}
Let $(S,\wh\jmath,\chi,\theta)$, $(G',\xi,z)$, and $j$ be as just discussed. Let \(X \in \Lie^*(S)(F)\) satisfy \(\theta(\exp(Y)) = \Lambda(\<X, Y\>)\) for all \(Y \in \Lie(S)(F)_{0+}\). Let $j_\mf{w} : S \to G$ be the unique rational class of admissible embeddings for which $\pi_{j_\mf{w}}$ is $\mf{w}$-generic. For $\dot s \in [\wh{\bar S}]^+ \subset S_\varphi^+$, the value of $\Theta_{\varphi,\mf{w},z}^{\dot s}$  at a strongly regular semi-simple element $\gamma' \in G'(F)$ with topological Jordan decomposition $\gamma' = \gamma'_0 \cdot \gamma'_{0+}$ modulo $Z_G$ is given by
\[ e(J')\epsilon_L(T_G-T_J)\sum_{j : S \to J'}\Delta_{II}^\tx{abs}[a,\chi](\gamma_{0}'^{j}) \theta(\gamma_{0}'^{j}) \sum_{k : S \to J'} \pair{\operatorname{inv}(j_\mf{w},k)}{\dot s}\hO^{J'}_{\^kX}(\log(\gamma_{0+}')),  \]
where $J'=C_{G'}(\gamma_{0}')^\circ$, $T_G$ and $T_J$ are the minimal Levi subgroups in the quasi-split inner forms of $G'$ and $J'$, $\epsilon_L(T_G-T_J)$ is the root number of the virtual Galois representation $X^*(T_G)_\C - X^*(T_J)_\C$ with respect to the fixed additive character $\Lambda$, and $\inv(j_\mf{w},k) \in H^1(u \to W,Z(G) \to S)$ is the invariant defined in \cite[\S5.1]{KalRSP}. The sum over \(j\) runs over the set of $J$-stable classes of embeddings $j : S \to J'$ whose composition with the inclusion $J' \to G'$ is admissible, and the sum over $k$ runs over the set of $J'(F)$-conjugacy classes in the stable class of $j$.
\end{thm}
\begin{proof}
The proof is the same as for \cite[Lemma 6.3.1]{KalRSP}, but with \cite[Lemma 6.2.1]{KalRSP} replaced by Proposition \ref{pro:cf1}.
\end{proof}

As an immediate application of Theorem \ref{thm:schar} we see that Conjecture \cite[\S4.3]{KalDC} holds for the supercuspidal $L$-packets constructed in \cite{KalSLP}, subject to the conditions on $F$ stipulated in this section:

\begin{cor} \label{cor:stabchar}
Let $(S,\wh\jmath,\chi,\theta)$, $(G',\xi,z)$, and $j$ be as just discussed. Let \(X \in \Lie^*(S)(F)\) satisfy \(\theta(\exp(Y)) = \Lambda(\<X, Y\>)\) for all \(Y \in \Lie(S)(F)_{0+}\).
The value of $S\Theta_{\varphi,z}$  at a strongly regular semi-simple element $\gamma' \in G'(F)$ with topological Jordan decomposition $\gamma' = \gamma'_0 \cdot \gamma'_{0+}$ modulo $Z_G$ is given by
\[ e(J')\epsilon_L(T_G-T_J)\sum_{j : S \to J'}\Delta_{II}^\tx{abs}[a,\chi](\gamma_{0}'^{j}) \theta(\gamma_{0}'^{j}) \wh{SO}^{J'}_{^jX}(\log(\gamma_{0+}')),  \]
where $J'$, $T_G$, $T_J$, and $j$, are as in Theorem \ref{thm:schar}.
\end{cor}

\begin{thm}\ \\[-20pt] \label{thm:endo}
\begin{enumerate}
	\item The virtual character $\STheta_{\varphi,\mf{w},*}=\Theta_{\varphi,\mf{w},*}^1$ is stable across inner forms. That is, for any two rigid inner twists $(G_1',\xi_1,z_1)$ and $(G_2',\xi_2,z_2)$ and stably conjugate strongly regular elements $\gamma_1' \in G_1'(F)$ and $\gamma_2' \in G_2'(F)$ we have
	\[ \STheta_{\varphi,z_1}(\gamma_1') = \STheta_{\varphi,z_2}(\gamma_2'). \]
	\item For any $\dot s \in [\wh{\bar S}]^+ \subset S_\varphi^+$ the endoscopic character identity holds: if $\varphi^{\dot s}$ is the parameter for the endoscopic group $H$ corresponding to $(\varphi,\dot s)$, then
	\[ \Theta_{\varphi,\mf{w},z}^{\dot s}(f') = \STheta_{\varphi^{\dot s},1}(f^{\dot s}), \]
	where $(G',\xi,z)$ is a rigid inner twist, $f' \in \mc{C}^\infty_c(G'(F))$ is any test function and $f^{\dot s} \in \mc{C}^\infty_c(H(F))$ is its transfer with respect to the transfer factor $\Delta[\dot s,\mf{w},z]$ of \cite[(5.10)]{KalRI}.
\end{enumerate}
\end{thm}
\begin{proof}
The proof is the same as for \cite[Theorems 6.3.2, 6.3.4]{KalRSP}, with \cite[Lemma 6.3.1]{KalRSP} replaced by Theorem \ref{thm:schar}.
\end{proof}

As discussed in \cite[\S4.4]{KalDC}, Corollary \ref{cor:stabchar} and Theorem \ref{thm:endo} uniquely characterize the local Langlands correspondence for regular supercuspidal parameters. The local Langlands correspondence for all supercuspidal parameters will be uniquely characterized in the same way once Theorem \ref{thm:endo} has been extended to all $\dot s \in S_\varphi^+$.

\section{Construction of $\epsilon^{G/M}_x$ and proof of Theorem \ref{thm:main}}

The goal of this section is to prove Theorem \ref{thm:main} by constructing $\eps_x^{G/M}$ and computing its restriction to every tame maximal torus whose building contains $x$. The construction of $\eps_x^{G/M}$ involves three different techniques that we introduce in \S\S\ref{sub:sgn1f}, \ref{sub:sp-norm} and \ref{sec:sgn3} in a general set-up before specializing to our set-up in \S\ref{sub:epsxGM}.

\subsection{Sign characters from hypercohomology} \label{sub:sgn1f}\label{sec:sgn1}

Let $k$ be a field of odd characteristic with absolute Galois group $\Gamma$ relative to a fixed separable extension $\bar k$. As before we set $\Sigma = \Gamma \times \{\pm 1\}$. Let $\ms{G}$ be an algebraic group. We do not assume that $\ms{G}$ is connected or reductive.
Consider the abelian group of algebraic characters $X^*(\ms{G})$, written additively. It has an action of $\Sigma$, where $\Gamma$ acts according to the $k$-structure of $\ms{G}$ and $\{\pm 1\}$ acts by multiplication. We view the multiplication-by-$2$ map $X^*(\ms{G}) \to X^*(\ms{G})$ as a complex of $\Gamma$-modules of length $2$, which we place in degrees
$0$ and $1$. We recall the explicit description of the first
Galois-hypercohomology group of this complex.

\begin{dfn}
\begin{enumerate}
	\item The abelian group
$Z^1(\Gamma,X^*(\ms{G}) \to X^*(\ms{G}))$
of \textit{degree-$1$ hypercocycles} consists of pairs
$(\rho,\delta)$, where $\rho \in Z^1(\Gamma,X^*(\ms{G}))$
and $\delta \in X^*(\ms{G})$ satisfy $(1-\sigma)\delta=2\rho_\sigma$ for
all $\sigma \in \Gamma$. Addition is inherited from $X^*(\ms{G})$.
	\item The subgroup $B^1(\Gamma,X^*(\ms{G}) \to X^*(\ms{G}))$ of
\textit{degree-$1$ hypercoboundaries} consists of the pairs
$((1-\sigma)\chi,2\chi)$ for $\chi \in X^*(\ms{G})$.
	\item The \textit{first hypercohomology} group
$H^1(\Gamma,X^*(\ms{G}) \to X^*(\ms{G}))$ is the quotient $Z^1/B^1$.
\end{enumerate}
\end{dfn}

\begin{dfn}
Given $(\rho,\delta) \in Z^1(\Gamma,X^*(\ms{G}) \to X^*(\ms{G}))$
we define for each $g \in \ms{G}(k)$ and $\sigma \in \Gamma$ an element
$\epsilon_{\rho,\delta}(g,\sigma) \in \bar k^\times$ by choosing
arbitrarily a square root \(\sqrt{\delta(g)} \in \bar k^\times\)
of $\delta(g) \in \bar k^\times$ and setting
\[ \epsilon_{\rho,\delta}(g,\sigma)
= \rho_\sigma(g) \dotm \frac{\sigma\sqrt{\delta(g)}}{\sqrt{\delta(g)}}. \]
\end{dfn}

\begin{lem}\hfill
\begin{enumerate}
	\item The element
$\epsilon_{\rho,\delta}(g,\sigma) \in \bar k^\times$
is independent of the choice of square root.
	\item For fixed $g \in \ms G(k)$, the function
$\sigma \mapsto \epsilon_{\rho,\delta}(g,\sigma)$
lies in $Z^1(\Gamma,\mu_2)$.
	\item The function $\ms G(k) \to Z^1(\Gamma,\mu_2)$ thus defined
is a group homomorphism.
\end{enumerate}
\end{lem}
\begin{proof}
Fix \(g \in \ms G(k)\) and \(\sqrt{\delta(g)}\).
Both $\sigma \mapsto \rho_\sigma(g)$ and
$\sigma \mapsto (\delta/2)(g,\sigma)
\ldef(\sqrt{\delta(g)})^{-1}\sigma(\sqrt{\delta(g)})$
are elements of
$Z^1(\Gamma,\bar k^\times)$. The first one, $\rho_\sigma(g)$, is by
definition multiplicative in $g$. The second, $(\delta/2)(g,\sigma)$, is
independent of the choice of square root, for a different choice of
square root will result in both numerator and denominator being
multiplied by $-1 \in k$, so the result is unchanged. This independence
then shows that
$(\delta/2)(g \dotm g',\sigma)
=(\delta/2)(g,\sigma) \dotm (\delta/2)(g',\sigma)$.
It follows that $\epsilon_{\rho,\delta}(g,\anondot)$
lies in $Z^1(\Gamma,\bar k^\times)$
and is multiplicative in $g$.
Now $\epsilon_{\rho,\delta}(g,\sigma)^2=1$ by the relation
$(1-\sigma)\delta=2\rho_\sigma$.
\end{proof}

Define $\epsilon_{\rho,\delta}(g) \in k^\times/k^{\times,2}$ to be the image of
$\epsilon_{\rho,\delta}(g,\anondot)$ under the isomorphism
$H^1(\Gamma,\mu_2) \cong k^\times/k^{\times,2}$. Thus
\begin{equation}
\label{eq:abs-hyper-char}
\epsilon_{\rho,\delta} : \ms{G}(k) \to k^\times/k^{\times,2}
\end{equation}
is a character.

\begin{rem}
\label{rem:hyperbolic-hyper}
\begin{enumerate}
    \item If \(\rho = 0\), so that \(\delta\) is \(\Gamma\)-fixed (i.e., defined over \(k\)), then \(\epsilon_{\rho, \delta}\) is the composition \(\ms G(k) \xrightarrow\delta k^\times \to k^\times/k^{\times, 2}\).
    \item In general, Hilbert's Theorem 90 implies the existence of $\zeta_g \in \bar k^\times$ such that $\rho_\sigma(g)=\zeta_g^{-1}\cdot \sigma\zeta_g$. Then $\epsilon_{\rho,\delta}(g) = \zeta_g^2 \cdot \delta(g) \in k^\times/k^{\times,2}$.
\end{enumerate}

\end{rem}

\begin{lem} \label{lem:sgn1func}
The assignment $(\rho,\delta) \mapsto \epsilon_{\rho,\delta}$ is
functorial. More precisely, if $f : \ms{H} \to \ms{G}$ is a morphism of
algebraic groups over $k$, then
\[ \epsilon_{\rho\circ f,\delta\circ f}
= \epsilon_{\rho,\delta}\circ f.\]
\end{lem}
\begin{proof}
Immediate.
\end{proof}

\begin{lem}
\label{lem:hyper-char-additive}
The assignment $(\rho,\delta) \mapsto \epsilon_{\rho,\delta}$ is a
group homomorphism
\[ Z^1(\Gamma,X^*(\ms{G}) \to X^*(\ms{G}))
\to \Hom(\ms{G}(k), k^\times/k^{\times,2}) \]
whose kernel contains $B^1(\Gamma,X^*(\ms{G}) \to X^*(\ms{G}))$ and hence descends to a
group homomorphism
\[ H^1(\Gamma,X^*(\ms{G}) \to X^*(\ms{G}))
\to \Hom(\ms{G}(k), k^\times/k^{\times,2}) \]
\end{lem}
\begin{proof}
The homomorphism statement reduces to showing that for fixed $g$ and
$\sigma$ the element $\epsilon_{\rho,\delta}(g,\sigma) \in \mu_2$ is
multiplicative in the pair $(\rho,\delta)$, which follows from the
independence of choice of square root. That this homomorphism kills
coboundaries is immediate from the formula for $\epsilon(g,\sigma)$.
\end{proof}

\begin{rem}
The construction of the character $\epsilon_{\rho,\delta}$ can be
rephrased more abstractly as follows. Consider $\ms{G}$ as a complex of
length $1$ placed in degree $0$. Let $\mb{G}_m \to \mb{G}_m$ be the
squaring map, considered as a complex placed in degrees $0$ and $1$. Let
$\Hom(\ms{G},\mb{G}_m \to \mb{G}_m)$ be the $\Hom$-complex. It equals
$X^*(\ms{G}) \to X^*(\ms{G})$. Now we have the obvious map
\[ H^1(\Gamma,\Hom(\ms{G},\mb{G}_m \to \mb{G}_m))
\to \Hom_\tx{grp}(\ms{G}(k),H^1(\Gamma,\mb{G}_m
\to \mb{G}_m)). \]
But, the squaring map on $\mb{G}_m$ being surjective, the complex
$\mb{G}_m \to \mb{G}_m$ is quasi-isomorphic to its kernel $\mu_2$.
\end{rem}

The following example gives a particular family of hypercocycles to which we can apply the construction \eqref{eq:abs-hyper-char}.

\begin{exa}
\label{exa:conc-hyper}
Let \(\fS\) be any set equipped with an action of \(\Sigma\)
for which \(\{\pm1\}\) has no fixed points,
and let \(i \mapsto \chi_i\) be a \(\Sigma\)-equivariant map
\(\fS \to X^*(\ms G)\).
We allow \(\chi_{-i}\) to equal \(\chi_i\).
Let \(\fS^+\) be any subset of \(\fS\) such that \(\fS\) is the
disjoint union of \(\fS^+\) and \(\fS^- \ldef -\fS^+\), and
define
\[
\delta_{\fS^+} = \sum_{i \in \fS^+} \chi_i
\quad\text{and, for all \(\sigma \in \Gamma\),}\quad
\rho_{\fS^+, \sigma} = \sum_{i \in \fS^+ \cap \sigma\fS^-} \chi_i.
\]
\end{exa}

\begin{lem}
\label{lem:conc-hyper}
In the notation of Example \ref{exa:conc-hyper},
$\sigma \mapsto \rho_{\fS^+, \sigma}$ lies in $Z^1(\Gamma,X^*(\ms G))$, and
we have the relation $(1-\sigma)\delta_{\fS^+} = 2\rho_{\fS^+, \sigma}$ in
$X^*(\ms G)$.
Therefore $(\rho_{\fS^+},\delta_{\fS^+})$
lies in $Z^1(\Gamma,X^*(\ms G) \to X^*(\ms G))$.
The class of this hypercocycle does not depend on the choice of
$\fS^+$.
\end{lem}
\begin{proof}
We abbreviate \(\rho_{\fS^+}\) and \(\delta_{\fS^+}\)
to \(\rho\) and \(\delta\), respectively.
The character $\rho_\sigma$ is the sum of $\chi_i$ for $i$ running
over the set $\fS^+ \cap \sigma \fS^-$.
In the following computation, for convenience, we abbreviate \(\sum_{i \in \fS'}\chi_i\) to just \([\fS']\) when \(\fS'\) is a subset of \(\fS\). Then $[\sigma \fS']=\sigma[\fS']$ for all $\sigma \in \Sigma$. For $\sigma,\tau \in \Gamma$ we have
\begin{align*}
\rho_\sigma+\sigma\rho_\tau
={} &[\fS^+\cap\sigma \fS^-]+[\sigma \fS^+\cap\sigma\tau \fS^-]\\
={}&[\cancel{\fS^+\cap\sigma \fS^-\cap\sigma\tau \fS^+}]
	+[\fS^+\cap\sigma \fS^-\cap\sigma\tau \fS^-]\\
	&{}+[\fS^+\cap\sigma \fS^+\cap\sigma\tau \fS^-]+[\cancel{\fS^-\cap\sigma \fS^+\cap\sigma\tau \fS^-}]\\
={}&\fS^+\cap\sigma\tau \fS^-\\
={}&\rho_{\sigma\tau},
\intertext{justifying the first claim, and}
(1-\sigma)\delta
={}&\fS^+-\sigma \fS^+\\
={}&[\fS^+\cap\sigma \fS^-]+[\cancel{\fS^+ \cap \sigma \fS^+}] \\
	&{}-\sigma[\cancel{\fS^+ \cap \sigma^{-1}\fS^+}]
	-\sigma[\fS^+ \cap \sigma^{-1}\fS^-]\\
={}&[\fS^+ \cap \sigma \fS^-]-[\fS^- \cap \sigma \fS^+]\\
={}&2[\fS^+ \cap \sigma \fS^-],
\end{align*}
justifying the second claim.

For the independence statement it is enough to consider the effect of
replacing \(\fS^+\) by \(\{-i\} \cup \fS^+ \sm \{i\}\)
for some \(i \in \fS^+\).
This change replaces $\delta$ by $\delta - 2\chi_i$,
while $\rho_\sigma$ is replaced by $\rho_\sigma + (\sigma-1)\chi_i$.
\end{proof}

Lemma \ref{lem:conc-hyper} allows us to make the following definition.

\begin{dfn} \label{dfn:epsX}
	Let $\fS$ and $i \mapsto \chi_i$ be as in Example \ref{exa:conc-hyper}. Then we write
	 $$\epsilon_\fS : \ms{G}(k) \to k^\times/k^{\times,2}$$
	 for the character of \eqref{eq:abs-hyper-char} associated to the class of the hypercocycle $(\rho_{\fS^+},\delta_{\fS^+})$ for any choice of $\fS^+ \subset \fS$  as in Example \ref{exa:conc-hyper}.
\end{dfn}

\begin{rem}
\label{rem:refine-hyper-char}
With the notation of Example \ref{exa:conc-hyper}, consider another set \(\fS'\) with \(\Sigma\)-action for which \(\{\pm1\}\) has no fixed points and a \(\Sigma\)-equivariant surjective map \(\pi : \fS \to \fS'\). Choose a subset \(\fS^{\prime, +}\) of \(\fS'\) such that $\fS' = \fS^{\prime,+} \sqcup -\fS^{\prime,+}$ and define \(\fS^+\) to be the pre-image of $\fS^{\prime,+}$. Put
\[
\chi_{i'} = \sum_{\substack{
    i \in \fS \\
    \pi(i) = i'
}}
    \chi_i
\]
for all \(i' \in \fS'\). Then \(\delta_{\fS^{\prime, +}}\) equals \(\delta_{\fS^+}\), \(\rho_{\fS^{\prime, +}}\) equals \(\rho_{\fS^+}\), and hence $\epsilon_{\fS'}=\epsilon_{\fS}$.
\end{rem}

Recall that $\Lang_i$ denotes the isomorphism $k_i^1 \to k_i^\times/k_{\pm i}^\times$ that is inverse to $a \mapsto a/\sigma_i(a)$. Then we can describe $\epsilon_\fS$ as follows.

\begin{pro}
\label{pro:conc-hyper}
For all \(\gamma \in \ms G(k)\) we have
\[\epsilon_\fS(\gamma)= \prod_{i \in \fS_\tx{asym}/\Sigma}
	\fNorm_{k_i/k}(\chi_i(\gamma))
\dotm \prod_{i \in \fS_\tx{sym}/\Gamma}
	\fNorm_{k_i/k}(\Lang_i(\chi_i(\gamma))) \mod k^{\times,2}. \]
\end{pro}
\begin{proof}
If \(\fS\) admits a \(\Sigma\)-invariant, disjoint
decomposition \(\fS = \fS_1 \cup \fS_2\),
then \(\delta_{\fS^+}\) equals \(\delta_{\fS_1^+} + \delta_{\fS_2^+}\)
and \(\rho_{\fS^+}\) equals \(\rho_{\fS_1^+} + \rho_{\fS_2^+}\);
so, by Lemma \ref{lem:hyper-char-additive} and the fact that the claimed
formula also respects such a decomposition,
we may, and do, assume that \(\Sigma\) acts transitively on \(\fS\).
The result is obvious if \(\fS\) is empty, so we may, and do,
further suppose that there is some element \(i \in \fS\).

If \(i\) is asymmetric and we choose \(\fS^+ = \Gamma\dota i\),
then \(\rho\) is equal to \(0\), \(\delta\) is equal to \(\sum_{\sigma \in \Gamma/\Gamma_{\chi_i}} \sigma\chi_i\), and the result follows from Remark \ref{rem:hyperbolic-hyper}(1).

If \(i\) is symmetric, then
\(\fNorm_{k_i/k_{\pm i}}(\chi_i(\gamma))\) equals \(1\). Choose a lift \(\delta_i\)
of \(\Lang_i(\chi_i(\gamma))
	\in k_i^\times/k_{\pm i}^\times\)
to \(k_i^\times\),
and put \(\delta_{\sigma i} = \sigma\delta_i\)
for all \(\sigma \in \Gamma\). Then
\(\delta_j/\delta_{-j} = \chi_j(\gamma)\)
for all \(j \in \fS\). For all \(\sigma \in \Gamma\), we have that
\begin{multline*}
\pleft(\prod_{j \in \fS^+} \delta_j^{-1}\pright)
	\sigma\pleft(\prod_{j \in \fS^+} \delta_j\pright)
= \prod_{j \in \fS^+} \delta_j^{-1}
\dotm\prod_{j \in \sigma\fS^+} \delta_j
= \prod_{\substack{
	j \in \fS^+ \\
	j \notin \sigma\fS^+
}} \delta_j^{-1}
\dotm\prod_{\substack{
	j \notin \fS^+ \\
	j \in \sigma\fS^+
}}
	\delta_j \\
= \prod_{\substack{
	j \in \fS^+ \\
	j \notin \sigma\fS^+
}}
	\delta_j^{-1}\delta_{-j}
= \prod_{\substack{
	j \in \fS^+ \\
	j \notin \sigma\fS^+
}}
	\chi_j(\gamma)^{-1}
= \rho_\sigma(\gamma)^{-1}.
\end{multline*}
According to Remark \ref{rem:hyperbolic-hyper}(2), \(\epsilon_{\rho, \delta}(\gamma)\) is the class modulo \(k^{\times, 2}\) of
\begin{multline*}
\delta(\gamma) \cdot \pleft(
	\prod_{j \in \fS^+} \delta_j^{-1}
\pright)^2 
= \delta(\gamma)
\dotm\prod_{j \in \fS^+} \bigl(
	(\delta_j/\delta_{-j})
	(\delta_j\delta_{-j})
\bigr)^{-1} \\
= \delta(\gamma)
\dotm\prod_{j \in \fS^+} \chi_j(\gamma)^{-1}
\dotm\prod_{j \in \fS} \delta_j^{-1}
= \fNorm_{k_i/k}\delta_i^{-1}.\qedhere
\end{multline*}
\end{proof}

\subsection{Sign characters via the spinor norm}
\label{sub:sp-norm}

We continue with an arbitrary field $k$ of odd characteristic, with absolute Galois group $\Gamma$ relative to a fixed separable extension $\bar k$, as well as an algebraic group $\ms{G}$ defined over $k$, which we again do not assume to be reductive or connected.  Recall that we put \(\Sigma = \Gamma \times \{\pm1\}\).

Let $(V,\varphi)$ be a quadratic space over $k$.
We write \(\beta_\varphi\) for the associated symmetric bilinear form on
\(V\), given by
\(\beta_\varphi(v, w) = \varphi(v + w) - \varphi(v) - \varphi(w)\)
for all \(v, w \in V\).
(In \cite[Chapter 1, Definition 1.6]{scharlau:quadratic},
the associated symmetric bilinear form
is taken to be \(\frac1 2\beta_\varphi\) instead.
However, this will not trouble us.)
If \(\varphi\) is understood, then we may abbreviate \(\beta_\varphi\)
to \(\beta\).
We denote
the orthogonal group of \((V, \varphi)\) by \(\Or(V, \varphi)\)
and its special orthogonal group by \(\SO(V, \varphi)\)
(or by \(\Or(\varphi)\) and \(\SO(\varphi)\), respectively,
if \(V\) is understood).

The spinor norm
$\Or(V, \varphi)(k) \to k^\times/k^{\times,2}$
\cite[Chapter 9, Definition 3.4]{scharlau:quadratic}
is a homomorphism
determined by the following property: For an anisotropic vector
$v \in V$ the reflection $\tau_v \in \Or(V, \varphi)$ is defined by
$\tau_v(w)=w-2\frac{\beta(v,w)}{\beta(v,v)}v$ and its spinor norm is
$\varphi(v)$.

\begin{rem} \label{rem:snind}
The spinor norm also has a cohomological interpretation.
As an algebraic group defined over $k$, \(\Or(V, \varphi)\) has a central
extension $\Pin(V, \varphi)$ with kernel $\mu_2$.
The connecting homomorphism
$\Or(V, \varphi)(k) \to H^1(\Gamma,\mu_2)=k^\times/k^{\times,2}$
is the spinor norm.
The pre-image \(\Spin(V, \varphi)\) of \(\SO(V, \varphi)\) in \(\Pin(V, \varphi)\)
is the identity component of the latter, and the simply connected cover
of \(\SO(V, \varphi)\).
In particular, if \(\ms G\) is a connected, simply connected \(k\)-group,
then every orthogonal representation \(\ms G \to \Or(V, \varphi)\)
factors through \(\ms G \to \Spin(V, \varphi)\),
so that the spinor norm of the action of any element of \(\ms G(k)\) is
trivial.

It is possible that two distinct quadratic forms on $V$ that give the
same orthogonal group $\Or(\varphi)$ lead to different spinor norm homomorphisms.
For example, if we replace $\varphi$ by $a\varphi$ with $a \in k^\times$, the
spinor norm of each reflection will be multiplied by
$a \in k^\times/k^{\times,2}$.
However, the restriction of the spinor norm to
$\SO(\varphi)$ depends only on $\SO(\varphi)$ and not on $\varphi$,
because $\Spin(\varphi)$ is determined by $\SO(\varphi)$
(as its simply connected cover).
For example, since elements of $\SO(\varphi)$ are
products of an even number of reflections, replacing $\varphi$ by
$a\varphi$ multiplies the spinor norm of such elements by a power of $a^2$,
hence leaves it unchanged (as a square class).
\end{rem}

Every orthogonal representation $\ms{G} \to \Or(V,\varphi)$ of $\ms{G}$ gives rise to the character
\begin{equation}
\label{eq:spinor-char}
\ms{G}(k) \to \Or(V, \varphi)(k) \to k^\times/k^{\times,2},
\end{equation}
where the rightmost map is the spinor norm.
We would like to have a more explicit formula for the values of this character. In order to formulate the result, Lemma \ref{lem:spinor-ss}, we need the structure described in Definition \ref{dfn:grading}.

\begin{dfn}
\label{dfn:grading}
Let $(V,\varphi)$ be a non-degenerate quadratic space over \(k\), and let $\fS$ be a set equipped with an action of $\Sigma$. An \textit{$\fS$-grading} of $(V,\varphi)$ is a $\Gamma$-equivariant map from $\fS$ to the set of subspaces of $V_{\bar k}$, denoted by $i \mapsto (V_i)_{\bar k}$, such that $V_{\bar k}=\bigoplus_{i \in \fS} (V_i)_{\bar k}$ and $(V_i)_{\bar k}$ is orthogonal to $(V_j)_{\bar k}$ unless $j=-i$. Given a representation $\ms{G} \to \Or(V,\varphi)$, we say that the $\fS$-grading is \textit{$\ms{G}$-stable} if each $(V_i)_{\bar k}$ is $\ms{G}_{\bar k}$-stable.
\end{dfn}

Note that we are allowing the action of $-1 \in \Sigma$ to fix elements of $\fS$. A non-zero subspace $(V_i)_{\bar k}$ is anisotropic if and only if $i \in \fS$ is fixed by $-1$. Note further that it is implicit in the definition that $(V_i)_{\bar k} \neq (V_j)_{\bar k}$ whenever $i \neq j$ and $(V_i)_{\bar k} \neq \{0\}$.

For $i \in \fS$, let $k_i$ be the fixed field in $\bar k$ of the stabilizer $\Gamma_i$ of \(i\). 
The $\bar k$-vector space $(V_i)_{\bar k}$ has a natural $k_i$-structure given by $V_i = (V_i)_{\bar k}^{\Gamma_i}$.

Lemma \ref{lem:spinor-ss} gives an explicit formula for \eqref{eq:spinor-char} in a particular situation.

\begin{lem}
\label{lem:spinor-ss}
Let $\fS$ be a set equipped with a $\Sigma$-action, and $\ms{G} \to \Or(V,\varphi)$ be an orthogonal representation endowed with a $\ms{G}$-stable $\fS$-grading. Assume that for every $g \in \ms{G}(k)$ and $i \in \fS$ there exists $n$ such that the semisimple part of $g^{p^n}$ acts on $(V_i)_{\bar k}$ by a scalar, where $p$ is the characteristic exponent of $k$.

For all \(g \in \ms G(k)\), the spinor norm of \(g\) is
\begin{multline*}
\prod_{\substack{
	i = -i \\
	\text{\(-g\res_{V_i}\) unipotent}
}}
	\bigl(
		\disc(k_i/k)^{\dim_{k_i}(V_i)}\dotm
		\fNorm_{k_i/k}(\det\nolimits_{k_i}(\varphi\res_{V_i}))
	\bigr)
\dotm\prod_{\Gamma\dota i \ne -\Gamma\dota i}
	\fNorm_{k_i/k}(\det\nolimits_{k_i}(g\res_{V_i})) \\
\dotm\prod_{\substack{
	i \ne -i \\
	\Gamma\dota i = -\Gamma\dota i
}}
	\fNorm_{k_i/k}(\Lang_i(\det\nolimits_{k_i}(g\res_{V_i}))),
\end{multline*}
where the products are taken over the \(\Sigma\)-orbits of elements
\(i \in \fS\) satisfying the stated conditions.
\end{lem}

\begin{proof}
By \cite[55:4]{omeara:quadratic}, we may assume that \(\Sigma\) acts
transitively on \(\mathfrak S\).
The result is obvious if \(\mathfrak S\) is empty, so we may
further assume that there is some element \(i \in \mathfrak S\).
Upon replacing \(g\) by a suitable power \(g^{p^n}\),
which does not affect the spinor norm (since \(p\) is odd),
we may assume that \(g\) has a rational Jordan decomposition
and that its semisimple part acts as a scalar (necessarily in the separable closure \(\bar k\)).

Since unipotent elements of \(\Or(\varphi)\) lie in \(\SO(\varphi)\)
(because they have \(p\)-power order in the component group \(\mu_2\), and \(p \ne 2\)),
but the spinor norm is trivial on unipotent elements of \(\SO(\varphi)\)
(by, for example, its cohomological description
in Remark \ref{rem:snind}),
we may further assume that \(g\) is semisimple.
We make all of these assumptions. Then, by the assumption of the lemma, $g$ acts by a scalar on $V$.

Assume first that \(i = -i\). Then \(V_i\) is non-degenerate and $V=\Res_{k_i/k} V_i$, where the notation $\tx{Res}_{k_i/k} V_i$ means that we view $V_i$ as a $k$-vector space via the inclusion $k \to k_i$. The only way for \(g\) to be an orthogonal operator is
if the scalar \(\lambda\) by which it operates satisfies
\(\lambda^2 = 1\).
If \(\lambda\) is \(1\) (so that \(-g\res_{V}\) is not unipotent),
then the spinor norm of \(g\) is also \(1\).
If \(\lambda\) is \(-1\) (so that \(-g\res_{V}\) is unipotent),
then, by \cite
	[Chapter 9, Example 3.6
	and Chapter 2, Theorem 5.12]
{scharlau:quadratic},
the spinor norm of \(g\) is
\[
\det\bigl(\varphi\res_{\Res_{k_i/k} V_i}\bigr)
= \disc(k_i/k)^{\dim_{k_i}(V_i)}\fNorm_{k_i/k}(\det(\varphi\res_{V_i})).
\]

Assume now that \(i \ne -i\). Then $V_i \oplus V_{-i}$ is a hyperbolic plane defined over $k_i$ that descends to $k_{\pm i}$, and \(V = \Res_{k_{\pm i}/k} (V_i \oplus V_{-i})^{\Gamma_{\pm i}}\).  Since \(\det(g\res_{V_{-i}}) = \det(g\res_{V_i})^{-1}\), the image of \(g\) lies in \(\SO(\varphi)(k)\). By \cite[Chapter 9, Example 3.7]{scharlau:quadratic}, it suffices to assume that \(k_{\pm i}\) equals \(k\).

If \(i\) is asymmetric,
then \(V\) is the hyperbolic space
\(\mb H(V_i)\), in the notation of
\cite[Chapter 1, Definition 4.3]{scharlau:quadratic}.
By \cite[Chapter 9, Example 3.5]{scharlau:quadratic},
the spinor norm of \(g\)
is \(\det(g\res_{V_i})\).

If \(i\) is symmetric,
then we have an isometric isomorphism between \((V, \varphi)\)
and \((\tx{Res}_{k_i/k_{\pm i}}V_i, \varphi_i)\) where \(\varphi_i(v)= \varphi(v+\sigma_i v)=\beta(v,\sigma_i(v))\) for all \(v \in \tx{Res}_{k_i/k_{\pm i}}V_i\). Now $h_i(v,w)=\beta(v,\sigma_i(w))$ is a Hermitian form on the $k_i$-vector space $V_i$ with respect to the quadratic extension $k_i/k_{\pm i}$, and $\varphi_i(v)=h_i(v,v)$. The element $g$ preserves $h_i$. By \cite[Chapter 10, Theorem 1.5]{scharlau:quadratic}, the spinor norm of \(g\) is \(\fNorm_{k_i/k_{\pm i}}\Lang_i\det(g\res_{V_i})\).
\end{proof}

\begin{rem}
The situation of Lemma \ref{lem:spinor-ss} might seem rather
specialized. But, for any \(g \in \Or(V, \varphi)(k)\),
we can take \(\ms G\) to be the closed subgroup generated by \(g\),
\(\mathfrak S\) to be \(\bar k^\times\),
and
\((V_i)_{\bar k}\) to be the \(i\)-eigenspace
of the semisimple part of a suitable power \(g^{p^n}\) of $g$ that has rational semisimple part.
\end{rem}

\subsection{Sign characters from the Moy--Prasad filtration} \label{sec:sgn3}

In this subsection we work with a connected adjoint group $G$ over a $p$-adic field $F$, a tame twisted Levi subgroup $M \subset G$, and a point \(x \in \cB(M, F) \subset \cB(G, F)\).

\begin{dfn}
\label{dfn:g/m-by-ZM-and-R}
The adjoint action of $Z_M$ on $\fg$ decomposes $\fg$ as
\[ \fg = \fm \oplus \fm^\perp,\qquad \fm^\perp= \bigoplus_{\alpha_0 \in R(Z_M,G)} \fg_{\alpha_0}, \]
where \(\fg_{\alpha_0}\) is the \(\alpha_0\)-weight space of \(Z_M\) on \(\fg\).

\begin{enumerate}
	\item For $O \in R(Z_M,G)/I$, let $\fg_O=\bigoplus_{\alpha_0 \in O} \fg_{\alpha_0}$.
	\item For \((O, t) \in R(Z_M, G)/I \times \R\), let
\(V_{(x,O,t)}=\fg_O(F^\tx u)_{x, t:t+}\).
\end{enumerate}
Then \(V_{(x,O, t)}\) carries an action of \(M(F^\tx u)_x\),
which descends to an algebraic action of \((\ms{M}_x)_{\bar k}\).
\end{dfn}

Note that $e_{\alpha_0}$ depends only on the $\Sigma$-orbit of $\alpha_0$, so we may, and do, write $e_O$ for the common value of $e_{\alpha_0}$ for all $\alpha_0 \in O$, where $O$ is contained in an orbit of $\Sigma$ on $R(Z_M,G)$.

Proposition \ref{pro:sym-ram-extends} is a general fact about
the determinant of the action of \(\ms M_x\)
on certain weight- and Moy--Prasad-subspaces.

\begin{pro}
\label{pro:sym-ram-extends}
Let $R_M' \subset R(Z_M,G)$ be a $\Sigma$-invariant subset. Choose \(d_\omega \in e_\omega^{-1}\Z\) for every \(\omega \in R'_M/\Gamma\),
write \(d_a = d_O = d_\omega\) whenever \(a \in O \subset \omega\),
put
\[
V_{\bar k} \ldef \bigoplus_{O \in R'_M/I}  \quad\bigoplus_{t \in (0, d_O/2)}
	V_{(x,O, t)}
\]
and let $V=V_{\bar k}^\Gamma$. Given a tame maximal torus $T \subset M$ with $x \in \cB(T,F)$, write \(R'_T\) for the pre-image in \(R(T, G/M)\) of \(R'_M\).
For every \(\gamma \in T(F)\subb\), we have
\[
\sgn_k(\det(\gamma\res_V))=\!\!\!\!\prod_{\substack{
	\alpha \in R'_{T,\tx{asym}}/\Sigma \\
	0, d_{\alpha_0}/2 \notin \ord_x(\alpha) \\
	2 \nmid (d_{\alpha_0}\dotm e_\alpha)
}}
	\sgn_{k_\alpha}(\alpha(\gamma))
\dotm\prod_{\substack{
	\alpha \in R'_{T, \tx{sym,ram}}/\Gamma \\
	\alpha(\gamma) \in -1 + \fp_\alpha
}}
	\sgn_{k_\alpha}(-1)^{\lfloor
		(e_\alpha d_{\alpha_0} - 1)/2
	\rfloor}.
\]
\end{pro}
\begin{proof}
For every \(\gamma \in T(F)\subb\),
the value of \(\sgn_k(\det(\gamma\res_V))\) is
\[
\prod_{\alpha \in R'_T/\Gamma}
	\sgn_k(
		\fNorm_{k_\alpha/k}(\alpha(\gamma))
	)^{\card{\ord_x(\alpha) \cap (0,d_{\alpha_0}/2)}}.
\]
Choose \(\alpha \in R'_T\), write \(\alpha_0 = \alpha\res_{Z_M}\), and put \(d = d_{\alpha_0}\).
Recall
$\sgn_k(\fNorm_{k_\alpha/k}(\alpha(\gamma)))
=\sgn_{k_\alpha}(\alpha(\gamma))$.

Consider first an asymmetric root \(\alpha\).
Since $\sgn_{k_\alpha}(\alpha(\gamma))$ is
insensitive to replacing $\alpha$ by $-\alpha$,
the orbits of \(\alpha\) and \(-\alpha\) contribute
\[ \sgn_{k_\alpha}(\alpha(\gamma))
	^{\card{\ord_x(\alpha)
		\cap ((-d/2,0) \cup (0,d/2))}}. \]
Remember that \(\ord_x(\alpha)\) is an
\(e_\alpha^{-1}\Z\)-torsor and, in particular, is stable under
translation by \(d\).
If \(\ord_x(\alpha)\) contains \(0\) or \(d/2\),
then \(-\ord_x(\alpha)\) contains \(0\) or \(-d/2 + d = d/2\), respectively;
so, since \(\ord_x(\alpha)\) and \(-\ord_x(\alpha)\) are both
\(e_\alpha^{-1}\Z\)-torsors, they are equal.
Thus, in this case,
\(\ord_x(\alpha) \cap ((-d/2, 0) \cup (0, d/2))\)
is a set on which multiplication by \(-1\) is an involution
without fixed points, so it has even cardinality.
Otherwise, \(\ord_x(\alpha) \cap ((-d/2, 0) \cup (0, d/2))\)
equals \(\ord_x(\alpha) \cap (-d/2, d/2]\),
which, since
\(d\) lies in \(e_{\alpha_0}^{-1}\Z \subset e_\alpha^{-1}\Z\),
has \(d e_\alpha\) elements. Thus, the contribution of the asymmetric roots \(\alpha \in R(T, G/M)\) is
\[
\prod_{\substack{
    \alpha \in R'_{T,\tx{asym}}/\Sigma \\
    0, d_{\alpha_0}/2 \notin \ord_x(\alpha) \\
    2\nmid (d_{\alpha_0} \dotm e_\alpha)
}}
    \sgn_{k_\alpha}(\alpha(\gamma)).
\]

Next consider an unramified symmetric $\alpha$. Since
$\alpha(\gamma) \in k_\alpha^\times$ has trivial norm to
$k_{\pm\alpha}^\times$, it also has trivial norm to $k^\times$,
so its contribution to \(\sgn_k(\det(\gamma\res_V))\) equals $1$. 

Finally consider a ramified symmetric $\alpha$.
Note that in this case $\alpha(\gamma) \in \pm 1 + \mathfrak p_\alpha$.
By
\cite[Proposition 4.5.1]{KalRSP} we have $0 \in \ord_x(\alpha)$,
so \(\ord_x(\alpha)\) equals \(e_\alpha^{-1}\Z\).
Thus the intersection $\ord_x(\alpha) \cap (0, d/2)$
has cardinality $\lfloor(e_\alpha d-1)/2\rfloor$.
\end{proof}

\subsection{Good elements and quadratic forms}

We work with a connected adjoint group $G$ over $F$ and a twisted Levi subgroup $M \subset G$,
 and we assume that there exists, and fix, a good element \(X \in \Lie^*(\Mscab)(F)\). The main results of this section hold without assuming that $G$ is adjoint, cf. Remark \ref{rem:nonadjointform}, but we impose this condition for convenience.

The spinor-norm construction \eqref{eq:spinor-char} discussed in \S\ref{sub:sp-norm} takes as an input an orthogonal representation of an algebraic group. In this section we discuss the source of such representations for our purposes.

Since $G$ is adjoint, the center $Z_M$ of $M$ is a torus. We set
\[ \fm^\perp = \bigoplus_{\alpha_0 \in R(Z_M,G)} \fg_{\alpha_0}. \]
Our goal is to define a natural non-degenerate quadratic form on $\fm^\perp(F)$ that is invariant under $M(F)$. Such a form can be obtained by Galois descent from a non-degenerate quadratic form on $\fm^\perp(\sepfield)$ that is invariant under $M(\sepfield)$ and $\Gamma$, where $\sepfield$ is the separable closure of $F$.

We recall that $\Gsc$ denotes the simply connected cover of $G\dergp$ and $\Msc$ the preimage of $M$ in $\Gsc$. The derived subgroup $\Mscder$ of $\Msc$ is simply connected.
We have the $M$-stable decomposition $\gsc = \msc \oplus \fm^\perp$, which we use to identify \(\msc^*\), respectively \(\fm^{\perp, *}\), with the annihilator of \(\fm^\perp\), respectively of \(\msc\), and so obtain the \(M\)-stable decomposition \(\gsc^* = \msc^* \oplus \fm^{\perp, *}\).
We also have the $M$-invariant inclusion $\Lie^*(\Mscab) \to \msc^*$, where \(\Mscab\) is the abelianization of \(\Msc\).

\begin{dfn}
\label{dfn:regular-length}
Let $R$ be a reduced root system. 
\begin{enumerate}
    \item Assume $R$ is irreducible. For $\alpha \in R$ we write \(\ell(\alpha)\) for the integer\-normalized square length of \(\alpha\), so that \(\ell(\alpha)=1\) if $R$ is simply laced, and \(\ell(\alpha) = 1\) for all short roots \(\alpha\) if $R$ is not simply laced. Thus $\ell(\alpha) \in \{1,2,3\}$. Define $\ell_{p'}(\alpha) = \gcd(\ell(\alpha), p)^{-1}\ell(\alpha)$. 
    \item If $R$ is not necessarily irreducible, define $\ell(\alpha)$ and $\ell_{p'}(\alpha)$ for each irreducible factor separately. 
\end{enumerate}
\end{dfn}

\begin{rem}
This definition can also be applied to the dual root system $R^\vee$. Then the function $\ell(\alpha) \cdot \ell(\alpha^\vee)$ is constant on each irreducible factor of $R$, where it takes one of the values $\{1,2,3\}$. If this value equals $p$, then $\ell_{p'}(-)$ is identically $1$ on this factor of $R$ and the corresponding factor of $R^\vee$.
\end{rem}

\begin{dfn}\label{dfn:q-pmalpha}
	Let $T \subset M$ be a maximal torus.
As in \cite[\S1.3.3]{KotWeil}, for every \(\alpha \in R(T, G)\), we define a quadratic pairing \(Q_{\pm\alpha}\) on \(\fg_\alpha(\sepfield) \oplus \fg_{-\alpha}(\sepfield)\) as follows.  For \(Y_{\pm\alpha} \in \fg_{\pm\alpha}(\sepfield)\), we lift \(Y_{\pm\alpha}\) canonically to elements of \(\Lie(\Gsc)_{\pm\alpha}(\sepfield)\), and compute the commutator \([Y_\alpha, Y_{-\alpha}] \in \Lie(\Tsc)(\sepfield)\) of their lifts.  Since \(H_\alpha \ne 0\) in \(\Lie(\Tsc)(\sepfield)\), there is a unique scalar \(c\) such that \([Y_\alpha, Y_{-\alpha}] = cH_\alpha\) and we define \(Q_{\pm\alpha}(Y_\alpha + Y_{-\alpha}) = c\). For convenience, we denote the scalar $c$ by  $\frac{[Y_\alpha,Y_{-\alpha}]}{H_\alpha}$.
\end{dfn}

\begin{lem}
\label{lem:q-depth}
With the notation of Definition \ref{dfn:q-pmalpha}, let \(E/F\) be a tame extension of finite ramification degree over which \(T\) splits, and define \(d_x : \fg(E) \to \R\) by \(d_x(Y) = \sup \{t \in \R \stbar Y \in \fg(E)_{x, t}\}\).
Given $Y_{\pm\alpha} \in \fg_{\pm\alpha}(\sepfield)$, we have $\ord(Q_{\pm\alpha}(Y_\alpha + Y_{-\alpha}))=d_x(Y_\alpha)+d_x(Y_{-\alpha})$.
\end{lem}

\begin{proof}
For $a \in E^\times$ we have $d_x(aY_\alpha)=d_x(Y_\alpha)+\ord(a)$;
and for $y \in \cB(T,E)$ and \(\alpha \in O\) we have $d_y(Y_\alpha)=d_x(Y_\alpha)+\pair\alpha{x-y}$.

It therefore suffices to prove this statement for \emph{some} \(x \in \cB(T, E)\); we choose a Chevalley basis for $\fg_E$ with respect to $\mf{t}_E$, and let $x$ be the corresponding point of $\mc{A}(T,E)$. It also suffices to prove it for \emph{some} non-zero elements \(Y_{\pm\alpha} \in \fg_{\pm\alpha}(E)\), so we take both \(Y_\alpha\) and \(Y_{-\alpha}\) in the chosen Chevalley basis.
By the definition of a Chevalley basis, $[Y_\alpha,Y_{-\alpha}]=H_\alpha$, so $Q_{\pm\alpha}(Y_\alpha + Y_{-\alpha})=1$ and hence $\ord(Q_{\pm\alpha}(Y_\alpha + Y_{-\alpha}))=0$.

Therefore it suffices to show that \(d_x(Y_\alpha) = d_x(Y_{-\alpha}) = 0\).
For this, note that, since \(G_E\) is split, we
can easily write down the integral model \(\mathscr G_{E, x}^\circ\) for our parahoric \(G(E)_{x, 0}\), and observe that \(\{Y_\alpha\}\) is an \(O_E\)-basis for \(\fg_\alpha(E) \cap \Lie(\mathscr G_{E, x}^\circ)(O_E)\).  By \cite[\S3.2]{MP94} and \cite[\S3.3]{MP96}, we have \(\fg_\alpha(E)_{x, 0} = \fg_\alpha(E) \cap \Lie(\mathscr G_{E, x}^\circ)(O_E)\) and \(\fg_\alpha(E)_{x, 0+} = \fg_\alpha(E) \cap \fp_E\Lie(\mathscr G_{E, x}^\circ)(O_E)\), so that \(d_x(Y_\alpha) = 0\).  It follows similarly that \(d_x(Y_{-\alpha}) = 0\).
\end{proof}

We recall that a quadratic form $\varphi$ is said to be non-degenerate if the associated bilinear form given by \(\beta_\varphi(v, w) = \varphi(v + w) - \varphi(v) - \varphi(w)\) is non-degenerate.

\begin{dfn} \label{dfn:q-on-g/m} Let $T \subset M$ be a maximal torus, so that $\fm^\perp = \bigoplus_{\alpha \in R(T,G/M)}\fg_\alpha$.
Let \(\varphi\) be the orthogonal sum of the quadratic forms \(\ell_{p'}(\alpha^\vee)Q_{\pm\alpha}\), with \(\ell_{p'}(\alpha^\vee)\) as in Definition \ref{dfn:regular-length} and \(Q_{\pm\alpha}\) as in Definition \ref{dfn:q-pmalpha}, over the set \(R(T, G/M)/\{\pm1\}\).
\end{dfn}

\begin{pro} \label{pro:minv}
The form $\varphi$ is an $M(\sepfield)$-invariant, $\Gamma$-invariant, non-degenerate quadratic form $\fm^\perp(\sepfield) \to \sepfield$. It is independent of the choice of $T$.
\end{pro}

\begin{rem} \label{rem:nonadjointform}
If we do not assume that $G$ is adjoint, we can apply the discussion in this section to $G\adform$ instead. Writing $\Mad$ for the image of $M$ in $G\adform$, the natural map $\mf{m}^\perp \to \mf{m}\adform^\perp$ is an $M(\bar F)$-equivariant isomorphism. Pulling back $\varphi$ along this isomorphism we obtain a form on $\mf{m}^\perp(\bar F)$ that satisfies Proposition \ref{pro:minv} without the assumption that $G$ is adjoint.
\end{rem}	

Before we prove Proposition \ref{pro:minv}, we record the way in which we will use it. Note that for any Galois extension $E/F$ we obtain from $\varphi$ a non-degenerate quadratic form $\fm^\perp(E) \to E$ that is $\Gamma$-invariant and $M(E)$-invariant. Let \(x \in \cB(M, F) \subset \cB(G, F)\). Recall from Definition \ref{dfn:g/m-by-ZM-and-R} the $\bar k$-vector space $V_{(x,O,t)}$ associated to any $O \in R(Z_M,G)/I$ and $t \in \bR$.

\begin{dfnlem}  \label{dfn:q-on-Vot}
For every finite subset $\fS$ of $R(Z_M,G)/I \times \bR$, the bilinear form $\beta_\varphi : \fm^\perp(F^\tx{un}) \times \fm^\perp(F^\tx{un}) \to F^\tx{un}$ descends to a non-degenerate, \(\ms M_x(\bar k)\)-invariant pairing
\[ \bigoplus_{(O, t) \in \fS} V_{(x,O,t)} \times \bigoplus_{(O, t) \in \fS} V_{(x,-O, -t)} \to \bar k. \]

If \(\fS\) is stable under \((O, t) \mapsto (-O, -t)\), then this yields a non-degenerate, \(\ms M_x(\bar k)\)-invariant symmetric bilinear form, which we again denote by $\beta_\varphi$. We denote by $\varphi(v)=\frac{1}{2}\beta_\varphi(v,v)$ the associated quadratic form on the $\bar k$-vector space \(V_{\fS, \bar k} \ldef \bigoplus_{(O, t) \in \fS} V_{(x,O,t)}\).
\end{dfnlem}
\begin{proof}
It is enough to consider a single $O \in R(Z_M,G)/I$ and \(t \in \R\). The restriction of $\beta_\varphi$ to $\fg_O(F^\tx{un}) \times \fg_{-O}(F^\tx{un})$ is a non-degenerate pairing taking values in $F^\tx{un}$ and invariant under $M(F^\tx{un})$ and $\Gamma$. Therefore, it is enough to show that the further restriction to $\fg_O(F^\tx{un})_{x,t} \times \fg_{-O}(F^\tx{un})_{x,-t}$ takes values in $O_{F^\tx{un}}$ and induces a non-degenerate pairing $V_{(x,O,t)} \times V_{(x,-O,-t)} \to \bar k$.

Fix a tame maximal torus $T \subset M$ with $x \in \mc{B}(T,F)$ and let $E/F^\tx{un}$ be a finite tame Galois extension splitting $T$.

Note that we have the decomposition of $E$-vector spaces
\[ \fg_O=\bigoplus_{\substack{\alpha \in R(T,G)\\ \alpha\res_{Z_M} \in O}} \fg_\alpha. \]

According to Lemma \ref{lem:q-depth}, \(\beta_\varphi\) restricts to an \(O_E\)-bilinear pairing 
\[ \fg_\alpha(E)_{x, t} \times \fg_{-\alpha}(E)_{x, -t} \to O_E \] 
whose values on \(\fg_\alpha(E)_{x, t} \times \fg_{-\alpha}(E)_{x, (-t)+}\) and \(\fg_\alpha(E)_{x, t+} \times \fg_{-\alpha}(E)_{x, -t}\) lie in \(\fp_E\).  It thus induces a \(\bar k\)-bilinear pairing 
\[ \fg_\alpha(E)_{x, t:t+} \times \fg_{-\alpha}(E)_{x, (-t):(-t)+} \to k_E = \bar k, \] 
which we will call \(\bar\beta_{\varphi, \alpha}\). Given $O \in R(Z_M,G)/I$, the orthogonal sum of \(\bar\beta_{\varphi, \alpha}\) for \(\alpha \in O\) gives a pairing on \(V_{(x,O, t)} \times V_{(x,-O, -t)}\) by restriction, which we will call \(\bar\beta_{\varphi, O}\).
Since \(M(F^\tx{ur})_x \to \ms M_x(\bar k)\) is surjective, the \(\ms M_x(\bar k)\)-invariance of \(\bar\beta_{\varphi, O}\) follows from the \(M(F^\tx{ur})_x\)-invariance of \(\beta_\varphi\).

It remains only to show that \(\bar\beta_{\varphi, O}\) is non-degenerate.
For this, fix \(\alpha \in O\).  Notice that, by tameness, the ramification degree \(e_\alpha = \card{I/I_\alpha}\) is non-zero in \(\bar k\); and, since \(I\) acts trivially on \(\bar k\), that \(\bar\beta_{\varphi, \sigma\alpha}(\sigma Y_\alpha, \sigma Y_{-\alpha}) = \bar\beta_{\varphi, \alpha}(Y_\alpha, Y_{-\alpha})\) for all \(Y_{\pm\alpha} \in \fg_{\pm\alpha}(E)_{x, (\pm t):(\pm t)+}\) and all \(\sigma \in I\).  Let \(\gamma\) be a square root of \(e_\alpha\) in \(\bar k\).  Then we have the isomorphisms \(\pi_{x, \pm t} : \fg_{\pm\alpha}(E)_{x, (\pm t):(\pm t)+} \to \fg_{\pm O}(F^\tx{ur})_{x, (\pm t):(\pm t)+}\) given by \(Y_{\pm\alpha} \mapsto \gamma^{-1}\sum_{\sigma \in I/I_\alpha} \sigma Y_{\pm\alpha}\).  For \(Y_{\pm\alpha}\) in \(\fg_{\pm\alpha}(E)_{x, (\pm t):(\pm t)+}\), we have that
\begin{multline*}
\bar\beta_{\varphi, O}(\pi_{x, t}(Y_\alpha), \pi_{x, -t}(Y_{-\alpha}))
= \gamma^{-2}\sum_{\sigma \in I/I_\alpha}
	\bar\beta_{\varphi, \sigma\alpha}(\sigma Y_\alpha, \sigma Y_{-\alpha}) \\
= e_\alpha^{-1}\sum_{\sigma \in I/I_\alpha}
	\bar\beta_{\varphi, \alpha}(Y_\alpha, Y_{-\alpha}) =\bar\beta_{\varphi, \alpha}(Y_\alpha, Y_{-\alpha}).
\end{multline*}
That is, the bilinear forms \(\bar\beta_{\varphi, \alpha}\) and \(\bar\beta_{\varphi, O}\) are isomorphic in the obvious sense, so that we need only show the non-degeneracy of \(\bar\beta_{\varphi, \alpha}\).

For this, fix \(Y_\alpha \in \fg_\alpha(E)_{x, t}\), and choose \(Y_{-\alpha} \in \fg_{-\alpha}(E)\) such that \([Y_\alpha, Y_{-\alpha}] = H_\alpha\).  Since \(\beta_\varphi(Y_\alpha, Y_{-\alpha}) = 1\), we have that \(\ord(\beta_\varphi(Y_\alpha, Y_{-\alpha})) = 0\) and hence by Lemma \ref{lem:q-depth} that \(d_x(Y_{-\alpha}) = -d_x(Y_\alpha)\).  It follows that either \(d_x(Y_{-\alpha}) \ge -t\), so that the images \(\bar Y_{\pm\alpha}\) of \(Y_{\pm\alpha}\) in \(\fg_{\pm\alpha}(E)_{x, (\pm t):(\pm t)+}\) satisfy \(\bar\beta_{\varphi, \alpha}(\bar Y_\alpha, \bar Y_{-\alpha}) = 1\), or \(d_x(Y_{-\alpha}) < -t\), in which case \(d_x(Y_\alpha) > t\) and so the image \(\bar Y_\alpha\) equals \(0\).
\end{proof}

The remainder of this section is devoted to the proof of Proposition \ref{pro:minv}. That $\varphi$ is non-degenerate is clear from its construction as orthogonal sum of hyperbolic planes. It is also clear that it is $\Gamma$-invariant. We need to prove that it is $M(\sepfield)$-invariant. Once this is done, the independence of $T$ will follow at once, since all maximal tori in $M$ are $M(\sepfield)$-conjugate.

The $M(\sepfield)$-invariance follows from some basic observations made in \cite{KotWeil}. In fact, the discussion of \cite[\S5.1, \S I.3]{KotWeil} shows that, when no root in \(R(T, G/M)\) satisfies $\ell(\alpha)=p$, so that $\ell_{p'}(\anondot^\vee)=\ell(\anondot^\vee)$, the form $\varphi$ is the restriction to $\fm^\perp(\sepfield)$ of a $G(\sepfield)$-invariant non-degenerate quadratic form on $\gsc$, so the proof of Proposition \ref{pro:minv} is complete in that case.

The rest of this section is concerned with the remaining case where some root in \(R(T, G/M)\) does satisfy $\ell(\alpha)=p$, although Lemma \ref{lem:pairing} is valid without this assumption.

\begin{lem}
\label{lem:pairing}
Let $Y \in \Lie^*(\Mscab)(F)$, \(T \subset M\) a maximal torus, \(\alpha \in R(T, G/M)\).
\begin{enumerate}
	\item\label{lem:pairing:with-p} The element $\ell(\alpha)\pair Y{H_\alpha}\in \sepfield$ depends only on \(\alpha\res_{Z_M}\).
	\item\label{lem:pairing:without-p} If in addition \(Y\) is good, i.e., there exists $r \in \R$ such that  $\ord(\pair Y{H_\beta})=-r$ for all $\beta \in R(T,G/M)$, then the element $\ell_{p'}(\alpha)\pair Y{H_\alpha}\in \sepfield$ depends only on \(\alpha\res_{Z_M}\).
\end{enumerate}
\end{lem}
\begin{proof}
Recall that $M_\tx{sc}$ and $T_\tx{sc}$ are the preimages of $M$ and $T$ in $G_\tx{sc}$, respectively. Note that the character lattice \(X^*(\Mscab)\) is the annihilator \(R^\vee(T, M)^\perp\) in \(X^*(\Tsc)\) of the coroot system \(R^\vee(T, M)\) of \(M\) relative to \(T\).
The element $Y$ lies in $\Lie^*(\Mscab)(\sepfield)=X^*(\Mscab)\otimes_\Z \sepfield = R^\vee(T, M)^\perp \otimes_\Z \sepfield$.  Therefore it is enough to check, for every $\alpha \in R(T,G/M)$ and $\chi \in R^\vee(T, M)^\perp$, that $\ell(\alpha)\pair\chi{\alpha^\vee} \in \Z$ for (\ref{lem:pairing:with-p}), and $\ell_{p'}(\alpha)\pair\chi{\alpha^\vee} \in \Z$ for (\ref{lem:pairing:without-p}), depend only on $\alpha\res_{Z_M}$.

(\ref{lem:pairing:with-p}) Suppose that \(\alpha, \beta \in R(T, G/M)\) satisfy \(\alpha\res_{Z_M} = \beta\res_{Z_M}\).  Then \(\alpha - \beta\) lies in \(\Z R(T, M)\).  Since the map \(R(T, M) \to \Z\R^\vee(T, M)\) given by \(r \mapsto \ell(r)r^\vee\) is additive in the obvious sense, it extends to a homomorphism \(\Z R(T, M) \to \Z R^\vee(T, M)\) by 
\cite[Corollary 2 of Proposition VI.1.6.19]{BourLie4-6}. 
It follows that \(\ell(\alpha)\alpha^\vee - \ell(\beta)\beta^\vee\) lies in \(\Z R^\vee(T, M)\), and hence that \(\ell(\alpha)\pair\chi{\alpha^\vee} = \ell(\beta)\pair\chi{\beta^\vee}\).

(\ref{lem:pairing:without-p}) The only case not covered by (\ref{lem:pairing:with-p}) is $\ell(\alpha)=p$, which we now assume. Given $\alpha,\beta \in R(T,G/M)$ with equal restriction to $Z_M$, they belong to the same irreducible component of $R(T,G)$ and hence $\ell(\alpha),\ell(\beta) \in \{1,p\}$. Since $\ord(\pair Y{H_\alpha})=\ord(\pair Y{H_\beta}) \in \R$, the equality \(\ell(\alpha)\pair Y{H_\alpha} = \ell(\beta)\pair Y{H_\beta}\) from (1) implies $\ell(\alpha)=\ell(\beta)$. 
Both sides of the equality \(\ell(\alpha)\pair\chi{\alpha^\vee} = \ell(\beta)\pair\chi{\beta^\vee}\) lie in \(\Z\), not in \(F\), so we may divide by \(\ell(\alpha) = \ell(\beta)\).
Noting that $\ell_{p'}(\alpha)=\ell_{p'}(\beta)=1$ we find that \(\ell_{p'}(\alpha)\pair\chi{\alpha^\vee} = \ell_{p'}(\beta)\pair\chi{\beta^\vee}\).
\end{proof}

Recall that  \(X \in \Lie^*(\Mscab)(F)\) denotes a good element.
For convenience, we assume from now on that \(G\) is absolutely simple, in addition to being adjoint.
This is mostly a notational convenience, but there is also no loss of generality in this assumption;
to reduce from the general case to this case,
we base change so that \(G\) is a product of absolutely simple factors, and then check each simple factor whose intersection with $M$ is a proper twisted Levi subgroup (with the restriction of $X$ to this factor) separately.

By absolute simplicity, there is a unique integer \(\ell\) such that \(\ell(\alpha) \cdot \ell(\alpha^\vee) = \ell\) for all \(\alpha \in R(T, G)\).  We finally make the key assumption that $p=\ell$ (so that, in particular, \(R(T, G)\) is not simply laced).
By combining both parts of Lemma \ref{lem:pairing} and the fact that \(X\) is good,
we may unambiguously call a root \(\alpha_0 \in R(Z_M, G)\) ``long'' or ``short'' according as it is the restriction to \(Z_M\) of a long or short element of \(R(T, G/M)\).
Put
\[ \fm^\perp_\tx{long} = \bigoplus_{\alpha_0 \in R(Z_M,G)_\tx{long}} \fg_{\alpha_0},\qquad \fm^\perp_\tx{short} = \bigoplus_{\alpha_0 \in R(Z_M,G)_\tx{short}} \fg_{\alpha_0}. \]
These are subspaces of the spaces
\[ V_\tx{long} = \bigoplus_{\alpha \in R(T,G)_\tx{long}} \fg_{\alpha},\qquad V_\tx{short} = \bigoplus_{\alpha \in R(T,G)_\tx{short}} \fg_{\alpha}, \]
where the sum is taken over the long or short roots, respectively.

We equip $V_\tx{long}(\sepfield)$ and $V_\tx{short}(\sepfield)$ with the quadratic forms $Q_{V_\tx{long}}$ and $Q_{V_\tx{short}}$ that are the orthogonal sums of the quadratic forms $Q_{\pm\alpha}: Y_\alpha+Y_{-\alpha} \mapsto \frac{[Y_\alpha,Y_{-\alpha}]}{H_\alpha}$ over appropriate $\alpha \in R(T,G)$.

Consider the group \(\tilde G = (\Gsc \times \Tsc)/Z_{\Gsc}\) and its maximal torus \(\tilde T = (\Tsc \times \Tsc)/Z_{\Gsc}\).  In \cite[\S3, \S4, \S5.1]{KotWeil} Kottwitz constructs a quadratic form $Q_{\tilde T}$ on $\tilde\ft(\sepfield)$ and shows that the quadratic form $Q_{\tilde G}$ on $\tilde\fg(\sepfield) = \tilde\ft(\sepfield) \oplus V_\tx{long}(\sepfield) \oplus V_\tx{short}(\sepfield)$ given by $Q_{\tilde T} \oplus Q_{V_\tx{long}} \oplus \ell Q_{V_\tx{short}}$ is $G(\sepfield)$-equivariant.
(Note, when comparing our notation with Kottwitz's, that he writes \(\mathbb G\) where we write \(G\); \(V'\) where we write \(V_\tx{long}\); and \(V''\) where we write \(V_\tx{short}\).)

Consider first the case that $\ell$ is non-zero in $F$.
Since the restrictions of $\varphi$ to $\fm^\perp_\tx{long}(\sepfield)$ and $\fm^\perp_\tx{short}(\sepfield)$ coincide with the restrictions of $Q_{\tilde G}$ and $\frac{1}{\ell} \dotm Q_{\tilde G}$, respectively, we see that $\varphi$ is $M(\sepfield)$-invariant.

Consider now the case that $\ell$ is zero in $F$.
In \cite[\S\S4.3, 5.4]{KotWeil}, Kottwitz constructs \(N_G(T)(\sepfield)\)-modules \(\tilde\ft'(\sepfield)\) and \(\tilde\ft''(\sepfield)\) such that we have a \(N_G(T)(\sepfield)\)-equivariant exact sequence \(0 \to \tilde\ft''(\sepfield) \to \tilde\ft(\sepfield) \to \tilde\ft'(\sepfield) \to 0\); and then puts \(G(\sepfield)\)-module structures on \(\tilde\fg''(\sepfield) = \tilde\ft''(\sepfield) \oplus V_\tx{short}(\sepfield)\) and \(\tilde\fg'(\sepfield) = \tilde\ft'(\sepfield) \oplus V_\tx{long}(\sepfield)\) such that \(Q_{\tilde\fg'} \ldef Q_{\tilde\ft'} \oplus Q_{V_\tx{long}}\) and \(Q_{\tilde\fg''} \ldef Q_{\tilde\ft''} \oplus Q_{V_\tx{short}}\) are \(G(\sepfield)\)-invariant, where $Q_{\tilde\ft'}$ and $Q_{\tilde\ft''}$ are quadratic forms on $\tilde\ft'(F)$ and $\tilde\ft''(F)$, respectively, and such that
\[ 0 \to \tilde\fg''(\sepfield) \to \tilde\fg(\sepfield) \to \tilde\fg'(\sepfield) \to 0 \]
is a \(G(\sepfield)\)-equivariant (automatically exact) sequence.
In particular, the \(M(\sepfield)\)-actions on \(\fm^\perp_\tx{long}(\sepfield)\) are the same whether we regard it as a subset of \(\tilde\fg''(\sepfield)\) or of \(\tilde\fg(\sepfield)\), and similarly for the \(M(\sepfield)\)-actions on \(\fm^\perp_\tx{short}(\sepfield)\) regarded as a subset of \(\tilde\fg'(\sepfield)\) or of \(\tilde\fg(\sepfield)\).

Now the restrictions of \(Q_{\tilde\fg'}\) and \(\varphi\) to $\fm^\perp_\tx{long}(\sepfield) \subset V_\tx{long}(\sepfield) \subset \tilde\fg'(\sepfield)$ agree, so that \(\varphi\res_{\fm^\perp_\tx{long}(\sepfield)}\) is \(M(\sepfield)\)-invariant; and similarly for the restrictions of \(Q_{\tilde\fg''}\) and \(\varphi\) to $\fm^\perp_\tx{short}(\sepfield) \subset \tilde\fg''(\sepfield)$, so that \(\varphi\res_{\fm^\perp_\tx{short}(\sepfield)}\) is \(M(\sepfield)\)-invariant.  Since \(\fm^\perp_\tx{long}(\sepfield)\) and \(\fm^\perp_\tx{short}(\sepfield)\) are \(\varphi\)-orthogonal, we have shown the desired invariance.  The proof of Proposition \ref{pro:minv} is now complete.

\begin{rem}
While we can always form the subspaces
\[
\fm^\perp_\tx{short} = \bigoplus\limits_{\alpha \in R(T,G/M)_\tx{short}} \fg_\alpha
\quad\text{and}\quad
\fm^\perp_\tx{long} = \bigoplus\limits_{\alpha \in R(T,G/M)_\tx{long}} \fg_\alpha,
\]
whether or not \(p = \ell\),
they will rarely be $M$-invariant when \(p \ne \ell\), even when a good element exists. For example, consider $G=\operatorname{PGSp}_4$ and a `Siegel' Levi subgroup \(M \cong \operatorname{GL}_2/\mu_2\).  When \(p \ne 2\), there is a good element of \(\Lie(\Mscab)\), but this invariance fails, essentially because there are a long root and a short root in \(R(T, G/M)\) that have the same restriction to \(Z_M\).
(In fact, this situation even arises in the setting of \cite{Yu01}; as shown in \cite[\S1.4.2]{kim-yu:sp4}, we may arrange by twisting that \(M\) becomes an anisotropic unitary group in \(2\) variables.)

Specifically, choose a maximal torus \(T \subset M\) and a basis of simple roots \(\Delta_M = \{\alpha_\tx{short}\}\) for \(R(T, M)\), and let \(\{\alpha_\tx{short}, \alpha_\tx{long}\}\) be a basis of simple roots for \(R(T, G)\) containing \(\Delta_M\).  Then \(X = d\alpha_\tx{short} + d\alpha_\tx{long}\) is a good element of \(\Lie^*(\Mscab)(F)\).
(Twisting to make \(M\) anisotropic might destroy the rationality of \(X\); but, since \(\Lie^*(\Mscab)\) is \(1\)-dimensional, there is some \(F\)-rational element of \(\sepfield X \sm \{0\}\), which is still good.)
If \(m\) is a non-trivial element of the \(\alpha_\tx{short}\)-root subgroup of \(M\) and \(Y\) is a non-trivial element of the root space corresponding to the short root \(\alpha_\tx{short} + \alpha_\tx{long} \in R(T, G/M)\), then \(\Ad(m)Y - Y\) is a non-trivial element of the root space corresponding to the long root \(2\alpha_\tx{short} + \alpha_\tx{long}\).  Both of these statements use \(p \ne 2\).

\end{rem}

\subsection{Construction of $\epsilon^{G/M}_x$}
\label{sub:epsxGM}

From now on and for the rest of the paper we return to the notation of
\S\ref{sec:main-stmt}. Thus we have
a connected adjoint group $G$ over the non-archimedean local field $F$,
a tame twisted Levi subgroup $M \subset G$, a point \(x \in \cB(M, F) \subset \cB(G, F)\),
and an $G$-good element $X \in \Lie^*(\Mscab)(F)$. We set $r=-\ord(\pair X{H_\alpha})$ for some (hence by goodness for any) $\alpha \in R(T,G/M)$, and put \(s = r/2\).

We now construct the character $\epsilon^{G/M}_x$ as the product of three pieces using the general constructions outlined so far.

For the first piece, we use the notation from Definition \ref{dfn:g/m-by-ZM-and-R}.

\begin{dfn} \label{dfn:esr}
Let $\_M\epsilon_\tx{sym,ram} : \ms M_x(k) \to \{\pm 1\}$ be the character $\sgn_k(\det(\anondot\res_V))$, where
\[ V_{\bar k} \ldef \bigoplus_{O \in R(Z_M,G)_\tx{sym,ram}/I} \quad\bigoplus_{t \in (0, e_O^{-1}/2)}
	V_{(x,O, t)}, \]
$V=(V_{\bar k})^\Gamma$, and $e_O$ is the common value of $e_{\alpha_0}$ for every $\alpha_0 \in O$.
\end{dfn}

To construct the second piece, we also adopt the notation of Definition \ref{dfn:g/m-by-ZM-and-R}.

\begin{lem} \label{lem:r-shift}
For $O\in R(Z_M, G)/I$ and $t \in \R$ the $\ms M_x(\bar k)$-modules $V_{(x,O, t)}$ and $V_{(x,O, t-r)}$ are isomorphic.
\end{lem}
\begin{proof}
Note that \((Y_1,Y_2) \mapsto \pair X{[Y_1, Y_2]}\) is a non-degenerate, equivariant pairing between the \(\ms M_x(\bar k)\)-modules \(V_{(x,O, t)}\) and \(V_{(x,O, r - t)}\), so that they are dual modules. At the same time, \(V_{(x,O, r - t)}\) and \(V_{(x,O, t - r)}\) are dual modules according to Definition/Lemma \ref{dfn:q-on-Vot}. Therefore  \(V_{(x,O, t)}\) and \(V_{(x,O, t - r)}\) are isomorphic as \(\ms M_x(\bar k)\)-modules.
\end{proof}

\begin{rem}
The isomorphism of Lemma \ref{lem:r-shift} can be made explicit using the definition of the quadratic form $\varphi$ of Definition \ref{dfn:q-on-g/m}. Choosing a maximal torus $T \subset M$ that splits over a tame Galois extension $E/F$ and such that $x \in \cB(T,F)$, \cite[Corollary 2.3]{Yu01} implies 
\[ V_{(x,O,t)}=\left(\bigoplus_{\alpha_0 \in O} \mf{g}_{\alpha_0}(E^\tx{ur})_{x,t:t+}\right)^{\tx{Gal}(E^\tx{ur}/F^\tx{ur})}. \]
It is therefore enough to produce an $M(E^\tx{ur})_x$-equivariant isomorphism 
\[ \bigoplus_{\alpha_0 \in O} \mf{g}_{\alpha_0}(E^\tx{ur})_{x,t:t+} \to \bigoplus_{\alpha_0 \in O} \mf{g}_{\alpha_0}(E^\tx{ur})_{x,(t-r):(t-r)+}.\]
On the component $\mf{g}_{\alpha_0}(E^\tx{ur})_{x,t:t+}$ this isomorphism is given by multiplication by the scalar $\ell_{p'}(\alpha^\vee)^{-1}\<X,H_\alpha\>$, where $\alpha \in R(T,G/M)$ is any element restricting to $\alpha$. Since $\ell_{p'}(\alpha^\vee)^{-1}=\ell_{p'}(\alpha)/\ell_{p'}$, where $\ell_{p'}$ is the largest value of $\ell_{p'}(\beta)$ for all $\beta$ belonging to the same irreducible component as $\alpha$, and since this irreducible component is uniquely determined by $\alpha_0$, we see from Lemma \ref{lem:pairing} that  the scalar $\ell_{p'}(\alpha^\vee)^{-1}\<X,H_\alpha\>$ depends only on $\alpha_0$. It is moreover assigned to $\alpha$ (and hence to $\alpha_0$) in a Galois-equivariant way. Therefore the isomorphisms on the various components indexed by $\alpha_0$ splice together to a Galois-equivariant isomorphism between the direct sums. That this isomorphism is the explication of the double duality of the proof of Lemma \ref{lem:r-shift} follows from Definition \ref{dfn:q-on-g/m}.
\end{rem}

We extend the action of $\Sigma$ on $R(Z_M, G)/I$ to an action on $R(Z_M, G)/I \times \R/r\Z$,  by letting $-1 \in \Sigma$ act on $\R/r\Z$ by multiplication by $-1$ and $\Gamma \subset \Sigma$ act trivially on $\R/r\Z$.

\begin{lem} \label{lem:dual}
	The map \(R(Z_M, G)/I \times \R \to X^*(\ms M_x)\)
sending \((O, t)\) to \(\chi_{(O,t)}\ldef\det(\anondot\res_{V_{(x,O, t)}})\) is \(\Sigma\)-equivariant and descends to $R(Z_M,G)/I \times \R/r\R$.
\end{lem}
\begin{proof}
Since \(\sigma V_{(x,O, t)}\) equals \(V_{(x,\sigma O, t)}\)
for all \(\sigma \in \Gamma\),
it suffices to verify
that \(\chi_{(-O, -t)}\) equals \(\chi_{(O, t)}^{-1}\).
This follows from the fact that the \(\ms M_x(\bar k)\)-modules \(V_{(x,O, t)}\) and \(V_{(x,-O, -t)}\)
are dual to one another
(Definition/Lemma \ref{dfn:q-on-Vot}). The second statement follows from Lemma \ref{lem:r-shift}.
\end{proof}

Since $-s$ and $s$ are congruent modulo $r$, if we put \(\chi_O = \det(\anondot\res V_{(x,O, s)})\)
for all \(O \in R(Z_M, G)/I\),
then \(O \mapsto \chi_O\) is \(\Sigma\)-equivariant. Given any \(\Sigma\)-stable subset $R' \subset R(Z_M, G)^\tx{sym,ram}$ we may set \(\fS = R'/I\) and use Definition \ref{dfn:epsX} to associate to it a character \(\epsilon_{\fS} : \ms M_x(k) \to k^\times/k^{\times, 2}\), which can be computed by Proposition \ref{pro:conc-hyper}. Of immediate use for us is the special case $R' = R(Z_M, G)^\tx{sym,ram}$.

\begin{dfn} \label{dfn:Mepss}
Let \(\_M\epsilon_s^\tx{sym,ram}:\ms M_x(k) \to \{\pm1\} \) be the composition of the isomorphism $\sgn_k : k^\times/k^{\times,2} \xrightarrow{\simeq} \{\pm 1\}$ with the character $\epsilon_\fS: \ms M_x(k) \to k^\times/k^{\times,2}$ obtained by applying Definition \ref{dfn:epsX} to the set $\fS=R(Z_M, G)^\tx{sym,ram}/I$ and the assignment $\fS \ni O \mapsto \chi_O= \det(\anondot\res V_{(x,O, s)}) \in X^*(\ms M_x)$.
\end{dfn}

We now come to the third piece. Let \(R'_M\) be a \(\Sigma\)-stable subset of \(R(Z_M, G)\).
In the notation of Definition \ref{dfn:g/m-by-ZM-and-R},
we have an action of \(\ms M_x\) on
\[
V_{\bar k} \ldef \bigoplus_{\substack{
	O \in R'_M/I \\
	t \in (-s, s)
}} V_{(x,O, t)}
\]
that preserves the form \(\varphi\) of Definition/Lemma \ref{dfn:q-on-Vot},
and that stabilizes each of the subspaces
\[
V_{0, \bar k} \ldef \bigoplus_{O \in R'_M/I} V_{(x,O, 0)}
\quad\text{and}\quad
V^\pm_{\bar k} \ldef \bigoplus_{\substack{
	O \in R'_M/I \\
	t \in \pm(0, s)
}} V_{(x,O, t)}.
\]
Then \(\varphi\) yields non-degenerate quadratic forms on \(V\), \(V_0\), and \(V^+ \oplus V^-\), where, as usual,
\(V = V_{\bar k}^\Gamma\), \(V_0 = V_{0, \bar k}^\Gamma\), and \(V^\pm = (V^\pm_{\bar k})^\Gamma\).
Since \(W\), for \(W\) = \(V\), \(V_0\), or \(V^+ \oplus V^-\), is \(\Gamma\)-stable and \(\varphi\) is \(\Gamma\)-fixed, we obtain a rational orthogonal representation
\[
\ms M_x \to \Or(W, \varphi),
\]
hence a homomorphism
\begin{equation} \label{eq:char_sn}
\ms M_x(k)
	\to \Or(W, \varphi)(k)
	\xrightarrow{\operatorname{sn}} k^\times/k^{\times, 2}
	\xrightarrow{\sgn_k} \{\pm1\},
\end{equation}
where $\operatorname{sn}$ denotes the spinor norm on the appropriate orthogonal group, as in \S\ref{sub:sp-norm}.
We temporarily write \(\epsilon_{\operatorname{sn}}\),
\(\epsilon_{\operatorname{sn}, 0}\),
and \(\epsilon_{\operatorname{sn}}^0\)
for the resulting characters.
By \cite[55:4]{omeara:quadratic}, we have that
\(\epsilon_{\operatorname{sn}}\)
equals \(\epsilon_{\operatorname{sn}, 0} \dotm \epsilon_{\operatorname{sn}}^0\).

\begin{dfn} \label{dfn:ep0}
Let $\_M\epsilon_0$ be the character $\epsilon_{\operatorname{sn}, 0}$ of $\ms M_x(k)$ given by \eqref{eq:char_sn} relative to $R_M'=R(Z_M,G)_\tx{sym,ram}$.
\end{dfn}

\begin{dfn}
Let $\epsilon^{G/M}_x$ be the product of the characters $\_M\epsilon_\tx{sym,ram}$ (Definition \ref{dfn:esr}), \(\_M\epsilon_s^\tx{sym,ram}\) (Definition \ref{dfn:Mepss}), and $\_M\epsilon_0$ (Definition \ref{dfn:ep0}).
\end{dfn}

The character $\epsilon^{G/M}_x$ has now been constructed. In \S\ref{subsec:main-proof} we prove that its values are as required by Theorem \ref{thm:main}. As preparation, we compute the values of the character $\epsilon_{\operatorname{sn}, 0}$ constructed with respect to an arbitrary subset $R'_M$.

\begin{pro} \label{pro:all-at-0}
Fix a tame maximal torus \(T \subset M\) with \(x \in \cB(T, F)\),
and let \(R'_T\) be the pre-image in \(R(T, G/M)\) of \(R'_M\). For every \(\gamma \in \ms T(k)\),
the value \(\epsilon_{\operatorname{sn}, 0}(\gamma)\)
equals
\begin{multline*}
\prod_{\substack{
	\alpha \in R'_{T, \tx{asym}}/\Sigma \\
	0 \in \ord_x(\alpha)
}}
	\sgn_{k_\alpha}(\alpha(\gamma))
\dotm\prod_{\substack{
	\alpha \in R'_{T, \tx{sym,unram}}/\Gamma \\
	0 \in \ord_x(\alpha)
}}
	\sgn_{k_\alpha^1}(\alpha(\gamma)) \\
\dotm\prod_{\alpha \in R'_{T, \tx{sym,ram}}/\Gamma}
	\bigl(
		(-1)^{[k_\alpha : k] + 1}
		\sgn_{k_\alpha}(e_\alpha\ell_{p'}(\alpha^\vee))
		f_{(G, T)}(\alpha)
	\bigr)
\end{multline*}
\end{pro}

\begin{proof}
Fix \(\gamma \in \ms T(k)\).
We use Lemma \ref{lem:spinor-ss},
applied to \(\ms G = \ms T\)
and \(\fS = \{\alpha \in R'_T \stbar 0 \in \ord_x(\alpha)\}/I\),
with \((V_O)_{\bar k} = \fg_O(F^\tx{ur})_{x, 0:0+}\)
for all \(O \in \fS\),
to write the desired spinor norm as a product over the
\(\Sigma_k\)-orbits in \(\fS\),
and then apply \(\sgn_k\) to it.

For all \(\alpha \in R'_T\), we have that
\(k_{I\dota\alpha}\) is the residue field \(k_\alpha\) of \(F_\alpha\);
that
\(\sgn_k \circ \fNorm_{k_\alpha/k}\) equals \(\sgn_{k_\alpha}\);
and that \(\sgn_{k_\alpha} \circ \Lang_\alpha\)
equals \(\sgn_{k_\alpha^1}\)
if \(\alpha\) is unramified symmetric.
Thus, the contribution of the roots that are not ramified symmetric
is as stated.

Now fix \(\alpha \in R'_{T, \tx{sym,ram}}\),
and let \(O\) be the \(I\)-orbit of \(\alpha\).
By \cite[Proposition 4.5.1]{KalRSP},
we have that
\(\fg_\alpha(F_\alpha)_{x, 0} \ne \fg_\alpha(F_\alpha)_{x, 0+}\),
and hence that \(O\) lies in \(\fS\).
Since \(\alpha(\gamma)\) lies in \(F_\alpha^1\),
it belongs to \(\pm 1 + \fp_\alpha\).
If \(\alpha(\gamma) \in 1 + \fp_\alpha\), then \(O\) contributes only trivially to \(\epsilon_{\operatorname{sn}, 0}\) and to the product in the statement.
Thus, we may, and do, restrict to the case that \(\alpha(\gamma) \in -1 + \fp_\alpha\).
The contribution of the \(\Gamma\)-orbit of \(\alpha\)
to the spinor norm is
\[ \disc(k_\alpha/k)\fNorm_{k_\alpha/k}(
	\det(\varphi\res_{V_{(x,O, 0)}^{\Gal(\bar k/k_\alpha)}})
), \]
where we note that the $\bar k$-vector space $V_{(x,O,0)}$ is $1$-dimensional. 
By \cite[p.~24, Theorem 1, and p.~25, Lemma 1]{JMV90},
we have that \(\disc(k_\alpha/k)\) is the trivial square class
if and only if \([k_\alpha : k]\) is odd. To compute the determinant it is enough to evaluate the quadratic form $\varphi$ on any non-zero element of the 1-dimensional vector space $V_{(x,O,0)}$. Let \(Y_\alpha\) be an element of
\(\fg_\alpha(F_\alpha)_{x, 0} \sm \fg_\alpha(F_\alpha)_{x, 0+}\).
Then \(Y_O \ldef \sum_{\sigma \in I/I_\alpha} \sigma(Y_\alpha)\)
is a non-zero element of \(V_{(x,O, 0)}\),
and \(\varphi(Y_O)\) equals
$$
 \ell_{p'}(\alpha^\vee)\tr_{F_\alpha/F_O}\left(\frac{[Y_\alpha, \sigma_\alpha Y_\alpha]}{H_\alpha}\right) = e_\alpha\ell_{p'}(\alpha^\vee)\dotm \frac{[Y_\alpha, \sigma_\alpha Y_\alpha]}{H_\alpha}
$$
in \(k_O = k_\alpha\).
Since \(\sgn_k\left(\frac{[Y_\alpha, \sigma_\alpha Y_\alpha]}{H_\alpha}\right)\) equals \(f_{(G, T)}(\alpha)\), by definition,
the value of \(\sgn_k\) at the entire product equals
\[
(-1)^{[k_\alpha : k] + 1}f_{(G, T)}(\alpha)
\sgn_k(e_\alpha\ell_{p'}(\alpha^\vee)).\qedhere
\]
\end{proof}

\begin{rem} \label{rem:spinor-unifies}
One can interpret all three pieces of $\epsilon^{G/M}_x$ in terms of the spinor norm as follows.  Note that \(\_M\epsilon_0\) is already defined in those terms.

For the piece \(\_M\epsilon^\tx{sym,ram}_s\) constructed via hypercohomology, we consider more generally a subset \(\fS\) of \(R(Z_M, G)/I\) that contains no orbits of ramified symmetric weights (with \(\_M\epsilon^\tx{sym,ram}_s\) corresponding to \(\fS = R(Z_M, G) \sm R(Z_M, G)_\tx{sym,ram}\)).
A choice of subset \(\fS^+ \subset \fS\) as in Example \ref{exa:conc-hyper}
gives rise to subspaces
\[
V^\pm_s = \bigoplus_{O \in \pm \fS^+} V_{(x,O, s)},
\]
and so allows us to define an \(\ms M_x(\bar k)\)-invariant quadratic form
 \(\varphi_s\) on \(V^+_s \oplus V^-_s\) by \(\varphi_s(Y^+ + Y^-) = \pair X{[Y^+, Y^-]}\)
for all \(Y^\pm \in V^\pm_s\).

Then Lemma \ref{lem:spinor-ss} allows us explicitly to compute
\[ \ms M_x(k) \to \Or(V^+_s \oplus V^-_s, \varphi_s)(k) \xrightarrow{\operatorname{sn}} k^\times/k^{\times, 2}, \]
and Proposition \ref{pro:conc-hyper} shows that it agrees with the hypercohomology character \(\epsilon_{\fS}\),
so that the spinor-norm construction \eqref{eq:spinor-char} generalizes the construction of
\S\ref{sec:sgn3}.

For the piece \(\_M\epsilon_\tx{sym,ram}\) constructed via Moy--Prasad filtrations,
we can again use Lemma \ref{lem:spinor-ss} to show more generally that, in the notation of Proposition \ref{pro:sym-ram-extends}, the composition
\[
\ms M_x(k) \to
	\Or\bigl(
		\bigoplus_{O \in R'_M/I}
		\bigoplus_{\substack{
			t \in (-d_O^{-1}/2, d_O^{-1}/2) \\
			t \ne 0
		}}
			V_{(x,O, t)},
		\varphi
	\bigr)(k)
\xrightarrow{\operatorname{sn}}	k^\times/k^{\times, 2}
\]
agrees with \(\det(\anondot\res_V)\).
\end{rem}

\subsection{Proof of Theorem \ref{thm:main}}
\label{subsec:main-proof}

We now come to the proof of Theorem \ref{thm:main}. Fix a tame maximal torus $T \subset M$ such that $x \in \cB(T,F)$. For a finite $\Sigma$-stable
subset $\Phi \subset X^*(T)$
we can define
a character $\epsilon_\Phi : \ms{T}(k) \to \{\pm 1\}$ by
\begin{equation} \label{eq:charphi}
\epsilon_\Phi(\gamma)
= \prod_{\alpha \in \Phi_\tx{asym}/\Sigma}
	\sgn_{k_\alpha}(\alpha(\gamma))
\dotm\prod_{\alpha \in \Phi_\tx{sym,unram}/\Gamma}
	\sgn_{k_\alpha^1}(\alpha(\gamma)).
\end{equation}

\begin{exa}
\label{exa:sharp-and-flat0}
The characters $\epsilon_{\sharp, x}$
and $\epsilon_{\flat,0}$
of Definition \ref{dfn:lstk}
come via the construction \eqref{eq:charphi}
from the subsets
\[ \Phi_{\sharp}
= \{\alpha \in R(T,G/M)\stbar s \in \ord_x(\alpha) \}\]
and
\[ \Phi_{\flat,0}
= \{\alpha \in R(T,G/M)
	\stbar\alpha_0 \in R(Z_M,G/M)_\tx{sym,ram},2 \nmid e(\alpha/\alpha_0)\}. \]
\end{exa}

\begin{lem} \label{lem:symdif}
If $\Phi_1,\Phi_2 \subset X^*(T)$ are finite $\Sigma$-stable subsets, then so is their
symmetric difference
\[ \Phi_1 \sdiff \Phi_2
= (\Phi_1 \cup \Phi_2) \sm (\Phi_1 \cap \Phi_2), \]
and moreover
\[ \epsilon_{\Phi_1} \dotm \epsilon_{\Phi_2}
= \epsilon_{\Phi_1\sdiff\Phi_2}. \]
\end{lem}
\begin{proof}
Immediate.
\end{proof}

\begin{dfn}
\label{dfn:some-chars}
Let
$\epsilon_{0, \tx{sym,ram}}$, $\epsilon_s^\tx{sym,ram}$,
and $\epsilon^{0,s}_\tx{sym,ram}$
be the characters of \(\ms T(k)\) defined via the construction
\eqref{eq:charphi} using the following subsets of $R(T,G/M)$:
\begin{align*}
\Phi_{0, \tx{sym,ram}} ={} & \{\alpha_0 \in R(Z_M, G)_\tx{sym,ram},\
	0 \in \ord_x(\alpha)\}, \\
\Phi_s^\tx{sym,ram}
={} & \{\alpha_0 \notin R(Z_M,G)_\tx{sym,ram},\
	s \in \ord_x(\alpha)\},
\intertext{and}
\Phi^{0,s}_\tx{sym,ram}
={} &\{\alpha_0 \in R(Z_M,G)_\tx{sym,ram},\
	2\nmid e(\alpha/\alpha_0),\
	0 \notin \ord_x(\alpha)\not\ni s\}.
\end{align*}
\end{dfn}

\begin{lem}\label{lem:r1}
If $\alpha \in R(T,G/M)$ is such that $\alpha_0 = \alpha\res_{Z_M}$ is ramified symmetric, then the
rational number $e_{\alpha_0}r$ is an odd integer.
If $\alpha$ itself is
ramified symmetric, then the rational number $e_\alpha r$ is an odd
integer, and in particular the natural number $e(\alpha/\alpha_0)$ is
odd.
\end{lem}
\begin{proof}
By Lemma \ref{lem:pairing}(2), \(c:=\ell_{p'}(\alpha)\pair X{H_\alpha}\) depends only on \(\alpha_0\) and is therefore fixed by \(\Gamma_{\alpha_0}\) and negated by \(\sigma_{\alpha_0} \in \Gamma_{\pm\alpha_0} \sm \Gamma_{\alpha_0}\). Thus \(c \in F_{\alpha_0}\) satisfies \(\tr_{F_{\alpha_0}/F_{\pm\alpha_0}}(c) = 0\). Since the extension $F_{\alpha_0}/F_{\pm\alpha_0}$ is ramified,
this implies that
$-r = \ord(c)
\in \ord(F_{\alpha_0}^\times) \sm \ord(F_{\pm\alpha_0}^\times)
=e_{\alpha_0}^{-1}\Z \sm 2e_{\alpha_0}^{-1}\Z$,
and therefore $e_{\alpha_0}r$ is odd.

If \(\alpha\) is ramified symmetric, then an analogous argument gives that also \(e_\alpha r\),
and hence \(\frac{e_\alpha r}{e_{\alpha_0}r} = e(\alpha/\alpha_0)\),
is odd.
\end{proof}

\begin{cor} \label{cor:esr_asym}
The contribution of $R(T,G)_\tx{sym,unram}$ to the character $\epsilon^{0,s}_\tx{sym,ram}$ is trivial.
\end{cor}
\begin{proof}
The definition of \(\epsilon^{0,s}_\tx{sym,ram}\) involves a product over roots $\alpha$ that satisfy the conditions that $\alpha_0$ is ramified symmetric, $e(\alpha/\alpha_0)$ is odd, and $0 \notin \ord_x(\alpha) \not\ni s$. Let \(\alpha\) be such a root.  Lemma \ref{lem:r1} implies that $e_{\alpha_0}r$ is odd. Then $e_\alpha r$ is also odd, hence $e_\alpha s \in \Z+\frac{1}{2}$. Now $e_\alpha \dotm \ord_x(\alpha)$ is a $\Z$-torsor that, since it contains neither \(e_\alpha\dotm0 = 0\) nor the half-integer \(e_\alpha\dotm s\), equals neither $\Z$ nor $\Z+\frac{1}{2}$.  Thus $\ord_x(\alpha) \neq -\ord_x(\alpha)=\ord_x(-\alpha)$, so $\alpha \notin -\Gamma \dota \alpha$.
\end{proof}
   
\begin{lem} \label{lem:repack}
With notation as in Example \ref{exa:sharp-and-flat0} and Definitions \ref{dfn:lstk} and \ref{dfn:some-chars},
we have the equality
\[ \epsilon_{\sharp, x} \dotm \epsilon_{\flat,0}
= \epsilon_{0, \tx{sym,ram}}
\dotm \epsilon_s^\tx{sym,ram} \dotm \epsilon^{0,s}_\tx{sym,ram}. \]
\end{lem}
\begin{proof}
By Lemma \ref{lem:symdif} it is enough to show
\[ \Phi_{\sharp}\sdiff\Phi_{\flat,0}
= \Phi_{0, \tx{sym,ram}}
\sdiff \Phi_s^\tx{sym,ram} \sdiff \Phi^{0,s}_\tx{sym,ram}. \]
In fact the three sets on the right are pairwise disjoint, so
their symmetric difference is just their union.
The set $\Phi_{\flat,0}$ can be expressed as the disjoint union of the
following subsets of $R(T,G/M)$, where ``r.s.'' stands for
``ramified symmetric'', and ``n.r.s.'' for ``not ramified symmetric''
(i.e., asymmetric or unramified symmetric):
\begin{align*}
\{\tx{\(\alpha_0\) r.s.},
2 \nmid e(\alpha/\alpha_0),
0 \in \ord_x(\alpha) \ni s\}
{}\cup{}
\{\tx{\(\alpha_0\) r.s.},
2 \nmid e(\alpha/\alpha_0),
0 \notin \ord_x(\alpha) \ni s\}\\
{}\cup{}
\{\tx{\(\alpha_0\) r.s.},
2 \nmid e(\alpha/\alpha_0),
0 \in \ord_x(\alpha) \not\ni s\}
{}\cup{}
\{\tx{\(\alpha_0\) r.s.},
2 \nmid e(\alpha/\alpha_0),
0 \notin \ord_x(\alpha) \not\ni s\}
\intertext{while $\Phi_{\sharp}$ has the analogous subdivision}
\{\tx{\(\alpha_0\) r.s.},
2 \nmid e(\alpha/\alpha_0),
0 \in \ord_x(\alpha) \ni s \}
{}\cup{}
\{\tx{\(\alpha_0\) r.s.},
2 \nmid e(\alpha/\alpha_0),
0 \notin \ord_x(\alpha) \ni s \}\\
{}\cup{}
\{\tx{\(\alpha_0\) r.s.},
2 \mid e(\alpha/\alpha_0),
0 \in \ord_x(\alpha) \ni s \}
{}\cup{}
\{\tx{\(\alpha_0\) r.s.},
2 \mid e(\alpha/\alpha_0),
0 \notin \ord_x(\alpha) \ni s \}\\
{}\cup{}
\{\tx{\(\alpha_0\) n.r.s.},
\ord_x(\alpha) \ni s\}.\hskip 8.1cm
\intertext{The symmetric difference $\Phi_{\sharp}\sdiff\Phi_{\flat,0}$ is
therefore the disjoint union}
\{\tx{\(\alpha_0\) r.s.},
2 \nmid e(\alpha/\alpha_0),
0 \in \ord_x(\alpha) \not\ni s\}
{}\cup{}
\{\tx{\(\alpha_0\) r.s.},
2 \nmid e(\alpha/\alpha_0),
0 \notin \ord_x(\alpha) \not\ni s\}\\
{}\cup{}
\{\tx{\(\alpha_0\) r.s.},
2 \mid e(\alpha/\alpha_0),
0 \in \ord_x(\alpha) \ni s \}
{}\cup{}
\{\tx{\(\alpha_0\) r.s.},
2 \mid e(\alpha/\alpha_0),
0 \notin \ord_x(\alpha) \ni s \}\\
{}\cup{}
\{\tx{\(\alpha_0\) n.r.s.},
\ord_x(\alpha) \ni s \}.\hskip 8.1cm
\end{align*}
By Lemma \ref{lem:r1}, $e_{\alpha_0}r$ is odd when $\alpha_0$ is ramified symmetric, so the parities of $e_\alpha r$ and $e(\alpha/\alpha_0)$ are
the same. If $e(\alpha/\alpha_0)$ is even, then $e_\alpha s = \tfrac1 2 e_\alpha r$ is an
integer, so, since \(\ord_x(\alpha)\) is an \(e_\alpha^{-1}\Z\)-torsor,
the conditions $s \in \ord_x(\alpha)$ and
$0 \in \ord_x(\alpha)$ are equivalent.
If $e(\alpha/\alpha_0)$ is odd, then
$e_\alpha s$ is a half-integer, but not an integer, and so the
conditions $0 \in \ord_x(\alpha)$ and $s \in \ord_x(\alpha)$ are
mutually exclusive. Therefore, the above set of roots becomes the
disjoint union of
\begin{align*}
\Phi_{0, \tx{sym,ram}}
={} & \{\tx{\(\alpha_0\) r.s.},
0 \in \ord_x(\alpha)\}, \\
\Phi^{0,s}_\tx{sym,ram}
={} & \{\tx{\(\alpha_0\) r.s.},
2 \nmid e(\alpha/\alpha_0),
0 \notin \ord_x(\alpha) \not\ni s\},
\intertext{and}
\Phi_s^\tx{sym,ram}
={} & \{\tx{\(\alpha_0\) n.r.s.},
\ord_x(\alpha) \ni s\}.\qedhere
\end{align*}
\end{proof}

\begin{proof}[Proof of Theorem \ref{thm:main}]
We apply Proposition \ref{pro:sym-ram-extends} to compute the restriction of $\_M\epsilon_\tx{sym,ram}$ to $T(F)\subb$.  The contribution of the roots \(\alpha \in R(T, G)_\tx{sym,ram}\) equals \(\epsilon_{\flat,2}(\gamma)\).  In view of Corollary \ref{cor:esr_asym}, all the remaining contribution comes from \(\alpha \in R(T, G)_\tx{asym} \cap \Phi_\tx{sym,ram}^{0, s}\), for which \(e(\alpha/\alpha_0)\) is odd.  By Lemma \ref{lem:r1}, for such roots, $e_{\alpha_0}r$, hence also \(e_\alpha r\), is an odd integer, so \(e_\alpha(e_{\alpha_0}^{-1}/2 - s) = \tfrac1 2(e(\alpha/\alpha_0) - e_\alpha r)\) lies in \(\Z\).  Since \(\ord_x(\alpha)\) is an \(e_\alpha^{-1}\Z\)-torsor, we therefore have that the statements \(e_{\alpha_0}^{-1}/2 \notin \ord_x(\alpha)\) and \(s \notin \ord_x(\alpha)\) are equivalent, so that the remaining contribution is precisely \(\epsilon^{0,s}_\tx{sym,ram}\).

On the other hand, Lemma \ref{lem:sgn1func}, Remark \ref{rem:refine-hyper-char} (applied to the restriction map \(R(T, G/M)/I \to R(Z_M, G)/I\)), and Proposition \ref{pro:conc-hyper} imply that the restriction of $\_M\epsilon_s^\tx{sym,ram}$ to $T(F)\subb$ equals \(\epsilon_s^\tx{sym,ram}\). Finally, Proposition \ref{pro:all-at-0} shows that the restriction of $\_M\epsilon_0$ to $T(F)\subb$ equals $\epsilon_{0,\tx{sym},\tx{ram}} \dotm \epsilon_f \dotm \epsilon_{\flat,1}$. The result now follows from Lemma \ref{lem:repack}.
\end{proof}

\bibliographystyle{amsalpha}
\bibliography{bibliography}

\end{document}